\documentclass[11pt,twoside,openright]{report}

\title{The Signature Modulo $8$ \\ of \\ Fibre Bundles}
\author{Carmen Rovi}

\usepackage[phd]{edmaths}
\usepackage{amscd,amssymb}
\usepackage{amsthm,amsmath,amssymb}
\usepackage{amsmath}
\usepackage{fancybox}
\usepackage{fancyhdr}
\usepackage[usenames,dvipsnames]{color}
\usepackage{yfonts}
\usepackage[all]{xy}
\usepackage{calrsfs}
\usepackage{makeidx}
\usepackage{listings}
\usepackage{rotating}
\usepackage{float}
\usepackage{longtable}
\usepackage{textcomp}

\include{pythonlisting}

\usepackage{setspace}
\usepackage{hyperref}
\usepackage{graphicx}
\usepackage{cite}
\usepackage{pinlabel}
\newcommand\phantomarrow[2]{%
  \setbox0=\hbox{$\displaystyle #1\to$}%
  \hbox to \wd0{%
    $#2\mapstochar
     \cleaders\hbox{$\mkern-1mu\relbar\mkern-3mu$}\hfill
     \mkern-7mu\rightarrow$}%
  \,}

\makeindex

\newtheorem{theorem}{Theorem}[section]
\newtheorem*{theorem*}{Theorem}
\newtheorem*{conjecture*}{Conjecture}
\newtheorem*{theorem-non}{Theorem}
\newtheorem{lemma}[theorem]{Lemma}
\newtheorem{proposition}[theorem]{Proposition}
\newtheorem{corollary}[theorem]{Corollary}
\newtheorem{definition}[theorem]{Definition}
\newtheorem{example}[theorem]{Example}

\theoremstyle{definition}

\theoremstyle{remark}
\newtheorem{remark}[equation]{Remark}

\newcommand{\zz}{\mathbb Z}

\newcommand{\bb}[1]{\mathbb{#1}}
\newcommand{\mc}[1]{\mathcal{#1}}

\newcommand{\cP}{\mathcal P}

\newcommand{\pe}{\pi_1(E)}
\newcommand{\bbDA}{\mathbb{D}(\mathbb{A})}

\newcommand{\Fd}{F^{\textnormal{dual}}}

\newcommand{\wtE}{\widetilde{E}}
\newcommand{\wtB}{\widetilde{B}}
\newcommand{\wtX}{\widetilde{X}}
\newcommand{\wtF}{\widetilde{F}}

\newcommand{\T}{T\widehat{{\color{White}a}}}
\newcommand{\U}{U\widehat{{\color{White}a}}}

\def\:{\colon}
\def\geq{\geqslant}
\def\leq{\leqslant}

\def\emptyset{\varnothing}
\def\epsilon{\varepsilon}
\def\phi{\varphi}

\begin{document}

\pagenumbering{roman}

\maketitle

\declaration




\chapter*{Lay summary}


Topology studies the geometric properties of spaces that are preserved by continuous deformations. Manifolds are the main examples of  topological spaces, with the local properties of Euclidean space in an arbitrary dimension $n$.
They are the higher dimensional analogs of curves and surfaces. For example a circle is a one-dimensional manifold. Balloons and doughnuts are examples of two-dimensional manifolds. A balloon cannot be deformed continuously into a doughnut, so we see that there are essential topological differences between them.

An ``invariant" of a topological space is a number or an algebraic structure such that topologically equivalent spaces have the same invariant. For example the essential topological difference between the balloon and the doughnut is calculated by the ``Euler characteristic", which is $2$ for a balloon and $0$ for a doughnut.

In this thesis I  investigate the relation between three different but related invariants of manifolds with dimension divisible by $4$: the signature, the Brown-Kervaire invariant and the Arf invariant.

The signature invariant  takes values in the set $\{ \dots, -3, -2, -1, 0, 1, 2, 3, \dots \}$ of integers.
In this thesis we  focus on the signature invariant modulo $8$, that is its remainder after division by $8$.

The Brown-Kervaire invariant takes values in the set $\{0, 1, 2, 3, 4, 5, 6, 7 \}$.

The Arf invariant takes values  in the set $\{ 0, 1 \}$.

The main result of the thesis uses the Brown-Kervaire invariant to prove that for a manifold with signature divisible by $4$, the divisibility by $8$ is decided by the Arf invariant.

The thesis is entirely concerned with pure mathematics. However it is possible that it may have applications in mathematical physics, where the signature modulo $8$ plays a significant role.

\chapter*{Abstract}

This thesis is divided into three Parts, which are concerned with the residues modulo $4$ and $8$ of the signature $\sigma(M) \in \zz$ of an oriented $d$-dimensional geometric Poincar\'e complex $M^d$: the usual signature if $d=4k$, and $0$ otherwise.

In Part One we give a new chain complex proof of Morita's theorem (\hspace{-1pt}\cite[Theorem 1.1]{Morita})
 $$\sigma(M^{4k}) = \textnormal{BK}(H^{2k}(M;\zz_2),\lambda, q) \in \zz_8$$
with $\textnormal{BK}$  the Brown-Kervaire invariant, $\lambda:H^{2k}(M;\zz_2) \times H^{2k}(M;\zz_2) \to \zz_2$ the nonsingular intersection pairing and  $$q=\mc{P}_2:H^{2k}(M;\zz_2) \to H^{4k}(M;\zz_4)=\zz_4$$
the $\zz_4$-valued quadratic enhancement  of $\lambda$ defined by the Pontryagin square $\mc{P}_2$.  The mod $4$ signature is given by $\sigma(M) = q(v_{2k}(M)) \in \zz_4$ with $v_{2k}(M) \in H^{2k}(M; \zz_2)$ the Wu class.

When $\sigma(M) =0 \in \zz_4$ we identify $\sigma(M) \in \zz_8$ with a
$\zz_2$-valued Arf invariant, using the new construction of the maximal isotropic subquotient of $(H^{2k}(M; \zz_2), \lambda, q)$.

\noindent \textbf{Theorem \ref{4Arf-topology}} \textit{A $4k$-dimensional geometric Poincar\'e complex $M$ has signature $\sigma(M) =0 \in \zz_4$ if and only if
$L=\langle v_{2k}(M) \rangle \subset H^{2k}(M;\zz_2)$ is a sublagrangian of $(H^{2k}(M;\zz_2),\lambda,q)$, which we call the \textbf{Wu sublagrangian}. If such is the case, the maximal isotropic subquotient is a nonsingular symmetric form over $\zz_2$ with a $\zz_2$-valued enhancement
$$(W,\mu,h) = (L^{\perp}/L , [\lambda] , h), \textnormal{ with } h= [q]/2 \in \zz_2,$$
and the signature mod $8$ is given by
$$\sigma(M) = \textnormal{BK}(H^{2k}(M;\zz_2),\lambda,q) = 4\textnormal{Arf}(W,\mu,h) \in  4\zz_2 \subset \zz_8.$$ }
The proof is a direct application of an algebraic result, which we give in Proposition \ref{BK-and-4Arf}.

In Part Two we construct a chain complex model for the intersection form $H^{2k}(E)$ of the total space of a Poincar\'e duality fibration $F^{2m} \to E^{4k} \to B^{2n}$  with $m + n = 2k,$ following \cite{SurTransfer} and \cite{Korzen}.

In Part Three we apply the methods of Parts One and Two to obtain new results
on $\sigma(E)$  modulo $8$ for the total space $E$ of an Poincar\'e duality fibration $F^{2m} \to E^{4k} \to B^{2n}.$

By the results in \cite{Meyerpaper} and \cite{modfour} we know that the signature of a Poincar\'e duality  fibration
$F^{2m} \to E^{4k} \to B^{2n}$ is multiplicative mod $4$, that is,
$\sigma(E) \equiv \sigma(B) \sigma(F) \pmod{4} .$
Now $M = E \sqcup - (B \times F)$ (where $-$ reverses the orientation) has
$\sigma(M) = \sigma(E) - \sigma(B)\sigma(F) \in \zz$,
so that $\sigma(M) \equiv 0 \pmod{4}$, and Theorem \ref{4Arf-topology} can be applied to $M$.

\vspace{2pt}
\noindent \textbf{Theorem \ref{4Arf-general-fibration}} \textit{The signatures mod $8$ of the fibre, base and total space are related by
$$\sigma(E) - \sigma(B)\sigma(F) = 4 \textnormal{Arf}(W,\mu,h) \in 4\zz_2 \subset \zz_8$$
with $(W,\mu,h)$ the maximal isotropic subquotient of $(H^{2k}(M;\zz_2),\lambda,q)$ constructed in \ref{4Arf-topology} }

\vspace{2pt}
\noindent \textbf{Theorem \ref{mod8-theorem}} \textit{ If the action of $\pi_1(B)$ on $(H^m(F;\zz)/torsion) \otimes  \zz_4$ is trivial, then
the Arf invariant in \ref{4Arf-general-fibration} is trivial, and the signatures in $F^{2m} \to E^{4k} \to B^{2n}$ are multiplicative mod $8$,
$$\sigma(E) - \sigma(B)\sigma(F) = 0 \in \zz_8$$}
\vspace{-2pt}

Finally we study surface bundles $F^2 \to E^4 \to B^2$ which have  $\sigma(E)= 4\in \zz$, which provide examples of non-multiplicativity modulo $8$.  Such examples were first constructed by Endo. In Endo's example the action of $\pi_1(B)$ on $H^{1}(F; \zz_2)$ is non-trivial.
A Python module is used to find some further nontrivial examples of non-multiplicativity modulo $8$.

\chapter*{Acknowledgements}

My first acknowledgement is to thank my supervisor Andrew Ranicki.
Andrew is a very enthusiastic and generous supervisor and I am extremely grateful for all the help and advice that he has given me during my PhD.
His mathematical insights have been a great source of motivation for my work, and without his contributions and guidance this work would not have been possible.
I must emphasize that Andrew has been both friend and mentor to me, going beyond the call of duty as a supervisor.

\vspace{2pt}
During my time in Edinburgh I have been very lucky to be part of the Edinburgh Surgery Theory Study Group. It was wonderful to spend many hours learning about surgery theory with Chris, Dee, Patrick, Spiros and Filipp.

\vspace{2pt}
I would like to thank Michael Weiss for inviting me to M\"unster during the autumn of 2014. I am also grateful to the whole of the Leray seminar group for having made me feel so welcome during my stay in M\"unster.

\vspace{2pt}
I would like to thank Saul Schleimer for helping me understand the use of Dehn twists in Endo's paper and for suggesting that I use Python for computations of examples in chapter \ref{examples-chapter}.

\vspace{2pt}
I am also grateful to Andrew Korzeniewski for clarifying conversations and emails about the work in his thesis.

\vspace{2pt}
Many thanks  go to Michelle Bucher-Karlsson and Caterina Campagnolo for all the enjoyable hours doing mathematics with them  in Geneva, Les Diablerets and Edinburgh and for patiently listening to me talk about parts of my thesis when the work was still in progress. I look forward to more mathematical conversations with them.

\vspace{2pt}
Another gratifying mathematical experience took place during my second year, when  I became involved in a long term project to connect women with similar interests in topology. This project culminated in a conference
in Banff and the writing of joint papers in small groups. Even though I have not included the work from this research experience in my thesis, it was very satisfactory both personally and mathematically to have been part of the ``orbifolds team". 


\vspace{2pt}
I would also like to take this opportunity to thank Des Sheiham\textquotesingle s parents Ivan and Teresa for having funded high-quality recording equipment for the topology group.
This has been essential for the production of a collection of more than 200 videos on topology with a special focus on surgery theory, which I have recorded during my PhD.

\vspace{2pt}
On the non-mathematical side, I am deeply grateful to Ida Thompson for being such a good friend. My years in Edinburgh would certainly not have been so wonderful without her.

\vspace{2pt}
The few sentences that I can write in this page of acknowledgements to thank my parents and my sister Ana cannot do justice to my feelings for them. My parents have always encouraged me to pursue my dreams and have always supported me in important decisions in my life. Not being a mathematician himself, my father was able to shed light on the beauty of mathematics when I was a child. The pursuit of this beauty is the ultimate reason why I am writing this today.

\cleardoublepage \pagenumbering{arabic}

\tableofcontents

\setcounter{chapter}{-1}

\chapter{Introduction}\label{Foundations}

\section*{Overview}
The following chart presents the structure of the thesis showing the main topics that will be developed and how they relate to each other.

\begin{figure}[ht!]
\labellist
\small\hair 2pt
\pinlabel \small {\pageref{figure}} at 315 222
\endlabellist
\centering
\includegraphics[scale=0.9]{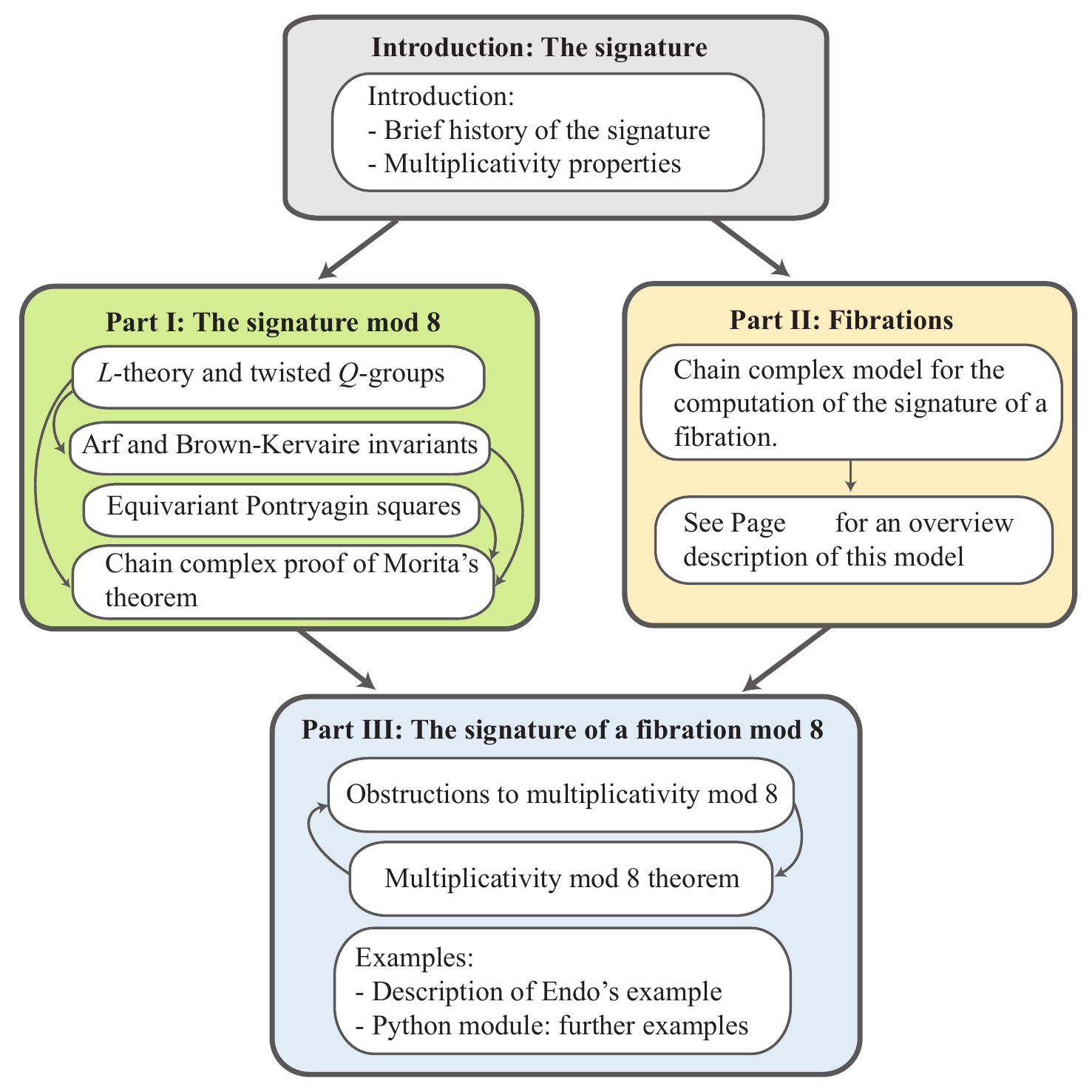}
\label{fig:cobo}
\end{figure}

\section{The signature: an introduction}

  The signature $\sigma(M) \in \mathbb{Z}$ of an oriented $n$-dimensional Poincar\'e complex $M^n$ is defined as follows,
\begin{definition}
\begin{itemize}
\item[(i)] If $n=4k$ then $\sigma(M)$ is defined to be the number of positive eigenvalues minus the number of negative eigenvalues of the non-singular symmetric intersection form $(H^{2k}(M; \mathbb{R}), \lambda)$, where
$$\lambda : H^{2k}(M; \mathbb{R}) \times H^{2k}(M; \mathbb{R}) \longrightarrow \mathbb{R} ; (u, v) \mapsto \left<u \cup v, [M] \right> $$
is the intersection pairing.
\item[(ii)] If $n \neq 4k$ then $\sigma(M) =0 \in \mathbb{Z}$.
\end{itemize}
\end{definition}

In general if we take the cartesian product of Poincar\'e complexes $X$ and $Y$ then the signature has multiplicative properties,
$$\sigma(X \times Y) - \sigma(X)  \sigma(Y)=0 \in \bb{Z}.$$

 Chern, Hirzebruch and
Serre \cite{HirzebruchSerreChern} were the first to consider the  problem of the multiplicativity of the signature of a fibre bundle  $F^{2m} \to E^{4k} \to B^{2n}$ of manifolds. They determined that if $\pi_1(B)$
acts trivially on the cohomology ring $H^*(F, \mathbb{Q})$, then
the signature is multiplicative,
$$\sigma(E) - \sigma(F)  \sigma(B) =0 \in \bb{Z}.$$

Kodaira \cite{Kodaira}, Hirzebruch \cite{Ramified-Hirz} and Atiyah \cite{Atiyah-cov} considered the situation when the action of the fundamental group on $H^*(F; \bb{Q})$  is non-trivial and independently constructed examples of fibre bundles  with non-multiplicative signature.

Meyer\cite{Meyerpaper}  proved that  $\sigma(E)=0 \in \zz_4$ for a surface bundle $F^2 \to E^4 \to B^2$. Later on,  Hambleton, Korzeniewski and Ranicki \cite{modfour} provided a high-dimensional version of this result. They proved that with $F\to E \to B$ a fibre bundle of closed, connected, compatibly oriented manifolds,
$$\sigma(E) - \sigma(F)\sigma(B)= 0\in \zz_4.$$

Here we identify the signature modulo $8$ with an Arf invariant.

\vspace{1pt}
\noindent \textbf{Theorem \ref{4Arf-general-fibration}} \textit{Let $F^{2m} \to E^{4k} \to B^{2n}$ be a   Poincar\'e duality fibration. With $(V, \lambda)= \left(H^{2k}(E, \zz_2), \lambda \right)$ and $(V', \lambda')= \left(H^{2k}(B \times F), \zz_2), \lambda' \right)$, the signatures mod $8$ of the fibre, base and total space are related by
$$
\sigma(E) - \sigma(B \times F) = 4 \textnormal{Arf}\left( L^{\perp}/L , [\lambda \oplus -\lambda'], \frac{\left[ \mc{P}_2 \oplus-\mc{P}'_2 \right]}{2} \right) \in 4\zz_2 \subset \zz_8,
$$
where
\begin{itemize}
\item[(i)] $\lambda: H^{2k}(E;\zz_2) \times H^{2k}(E;\zz_2) \to \zz_2$ and $\lambda': H^{2k}(B \times F;\zz_2) \times H^{2k}(B \times F;\zz_2) \to \zz_2$ are given by the cup product structures.
\item[(ii)]  $\mc{P}_2 : H^{2k}(E;\zz_2) \to \zz_4$ and $\mc{P}'_2 : H^{2k}(B\times F;\zz_2) \to \zz_4$ are the Pontryagin squares.
\item[(iii)] $L^{\perp} =\left\{(x, x') \in V \oplus V' \vert \lambda(x,x) = \lambda'(x', x') \in \zz_2 \right\}$,
\item [(iv)] $L = \langle v_{2k} \rangle \subset L^{\perp}$, with $v_{2k}=(v_{2k}(E), v_{2k}(B \times F)) \in V \oplus V'$ the Wu class of $E \sqcup -(B \times F).$
\end{itemize}
}
\vspace{3pt}

It is of interest to consider a weaker version of the trivial action condition of \cite{HirzebruchSerreChern}.
The following result was conjectured by Klaus and Teichner \cite{KT}:
\begin{conjecture*} \label{conjecture-KT} (\hspace{-1pt}\cite{KT})
Let $F^{2m} \to E \to B^{2n}$ be a fibration of oriented Poincar\'e complexes such that the action of $\pi_1(B)$ on $H^{m}(F, \bb{Z}_2)$ is trivial, then
$$\sigma(E) -\sigma(F) \sigma(B)  = 0 \in \zz_8.$$
\end{conjecture*}
In \cite{Korzen} there is a tentative proof of this result for the case when both $m$ and $n$ are even. Unfortunately that proof had a gap which we discuss in chapter \ref{mod-eight-proof}.
Our Theorem \ref{mod8-theorem} differs from this conjecture in that the action of $\pi_1(B)$ is on $H^m(F, \zz)/torsion \otimes \zz_4$.

\vspace{4pt}

\noindent \textbf{Theorem \ref{mod8-theorem}}  \textit{Let $F^{2m} \to E \to B^{2n}$ be a fibration of oriented Poincar\'e complexes. If the action of $\pi_1(B)$ on $H^m(F, \bb{Z})/torsion \otimes \bb{Z}_4$ is trivial, then
$$\sigma(E)  -\sigma(F) \sigma(B) =0 \in \zz_8.$$
}
\section{Chapter outline}

\subsubsection{Part I}
In Part \ref{signature modulo 8} we shall investigate properties of the signature of a symmetric Poincar\'e complex modulo $8$. The material about hyperquadratic $L$-theory and twisted $Q$-groups presented in chapter \ref{hyperquad} is very relevant for this purpose.

In chapter \ref{Arf-BK} we give definitions of the Arf and Brown-Kervaire invariants and discuss how they relate to each other. The main result in this chapter is Proposition \ref{BK-and-4Arf} where we use a maximal isotropic subquotient construction to express a Brown-Kervaire invariant which is divisible by $4$ as an Arf invariant.

In chapter \ref{Pontryagin-squares chapter} we shall concentrate on the definition of classical and equivariant Pontryagin squares.

Making use of the theory about the Brown-Kervaire invariant presented in \cite{BanRan}, we provide in chapter \ref{Morita-chain-cx} new chain complex proofs of the results in \cite{Morita}. One advantage of using chain complexes and the results in \cite{BanRan} is that the original proofs in \cite{Morita} and specially that of  \cite[Thm 1.1]{Morita} are simplified considerably. Another advantage is that we will be able to apply them in the algebraic proof of Theorem \ref{mod8-theorem}.

\subsubsection{Part II}

In Part II we shall give the construction of an appropriate algebraic model of the chain complex of the total space $C(E)$ of a fibration $F \to E \to B$ in terms of the chain complex of the fibre $C(F)$ and the universal cover of the base $C(\wtB)$ that will allow us to compute the signature of a fibration. This model was developed in \cite{Korzen} and uses the idea of transfer maps in $L$-theory from \cite{SurTransfer}.

\subsubsection{Part III}
In Part III we show how the obstruction to multiplicativity of the signature  modulo $8$ of a fibration with no condition on the action of the fundamental group of the base is in general given by the Arf invariant. The precise statement is given in Theorem \ref{4Arf-general-fibration}.

In Part III we prove Theorem \ref{mod8-theorem} about multiplicativity of the signature of a fibration modulo $8$ when there is a $\zz_4$-trivial action, which has already been mentioned above. To prove this theorem, we first state its algebraic analogue.
All the material presented in parts I and II will be relevant in the proof. Note that the statement of Theorem \ref{mod8-theorem} includes the case when $m$ and $n$ are odd, so that base and fibre are of dimensions $4i+2$ and $4j+2$ respectively. In this case Theorem \ref{mod8-theorem} takes the following form:

\begin{corollary}
Let $F^{4i+2} \to E^{4k} \to B^{4j+2}$ be a fibration of oriented Poincar\'e complexes such that the action of $\pi_1(B)$ on $H^{2m+1}(F, \bb{Z})/torsion \otimes{\zz}_4$ is trivial, then
$$\sigma(E)  \equiv 0  \in \zz_8.$$
\end{corollary}

This is specially interesting because it includes the case of surface bundles which has been widely studied in the literature.

One major problem in the context of Theorem \ref{mod8-theorem} has been to find non-trivial examples that satisfy this theorem.
The surface bundle examples of Atiyah and Kodaira have signature equal to $8$ or a multiple of $8$.

Endo used Meyer's arguments from \cite{Meyerpaper} to construct a surface bundle with signature $-4$ which has as basis an orientable surface of genus 111 and as fibre an orientable surface of genus 3. The action of the fundamental group is not given explicitly in the paper. I  have written  down the action explicitly to confirm that in this example the action of $\pi_1(B)$ on $H^1(F, \bb{Z}_2)$ is non-trivial, as expected from  Theorem \ref{mod8-theorem}.
 Lefschetz fibrations are another important source of examples of surface bundles with non-trivial signature, see for example \cite{BDS}, \cite{Ozbagci}, \cite{Stipsicz}. Nevertheless these fibrations have singular fibres and the computation of the signature of the total space depends also on the blow-up of the singularities.
An interesting construction of surface bundles with signature $4$ is given in \cite{Endo-alt}.

Meyer \cite[section 9]{Meyer-thesis} constructed two examples of local coefficient systems with non-trivial signature. One of them has signature 4 and the other has signature 8.

The appendix contains a computer module using the Python programming language which calculates the signature of the total space of a surface bundle by giving the explicit action as input. This module is used in chapter \ref{examples-chapter} to construct new examples with non-trivial signatures modulo $8$.

\part{The signature modulo $8$} \label{signature modulo 8}

\chapter*{Introduction to Part I}
\addcontentsline{toc}{chapter}{Introduction to Part I}

This part of the thesis is concerned with results about the signature modulo $4$ and modulo $8$ of symmetric Poincar\'e chain complexes over $\zz$.
The signature modulo $8$ of a symmetric Poincar\'e complex can be expressed in general by a $\zz_8$-valued Brown-Kervaire invariant. When the signature is 0 mod $4$, we show that it can be expressed as $4$ times a $\zz_2$-valued Arf invariant modulo $8$.

 \cite{Morita} proved that for a $4k$-dimensional Poincar\'e space $X$, the signature mod $8$ is the Brown-Kervaire invariant of a $\zz_4$-valued quadratic enhancement of the symmetric form on $H^{2k}(X, \zz_2).$ In \cite{Morita} Morita also proved that the signature mod $4$ is given by the Pontryagin square evaluated on the Wu class $v_{2k} \in H^{2k}(X, \zz_2)$.
We shall prove both results using symmetric Poincar\'e complexes over $\zz$. One advantage of doing this is that the results are stronger since they refer to chain complexes and not only to spaces. Another advantage is that the proof is significantly simplified. We shall need chain complex versions of the results in \cite{Morita} in part III of the thesis.

The overview of chapters in this part of the thesis is as follows:

In chapter \ref{hyperquad} we shall start  by presenting background results about $L$-theory that will be widely used throughout the thesis. Here we present a brief overview of the symmetric construction and symmetric, quadratic and hyperquadratic $L$-theory.

In chapter \ref{Arf-BK} we recall the definitions of the Arf and Brown-Kervaire invariants and describe how they relate to each other when the Brown-Kervaire takes values $0$ or $4$ in $\zz_8.$

In chapter \ref{Pontryagin-squares chapter} we recall  the classical construction of Pontryagin squares and also the construction of the equivariant Pontryagin squares given in \cite{Korzen}. We shall extend the definition of equivariant Pontryagin squares given in \cite{Korzen} to the evaluation of the equivariant Pontryagin square on odd cohomology classes.

In chapter \ref{Morita-chain-cx} we give chain complex proofs of  the results in \cite{Morita} and discuss how divisibility of the signature by $4$ and by $8$ can be detected.

\chapter{Background} \label{hyperquad}

\section{Symmetric, quadratic and hyperquadratic structures} \label{symm}
In this section we shall present some background material which will be fundamental for all three parts of the thesis. The ideas presented in this chapter are a review of some fundamental aspects in \cite{atsI}, \cite{atsII}.

\subsection{The algebraic theory}


Let $R$ denote a ring with involution with unit $1$,
$$\overline{\color{White}{a}}: R \to R ; a \mapsto \overline{a}.$$
For $a, b \in R$, this satisfies
$$\overline{a + b} = \overline{a} + \overline{b}, ~~~ \overline{\overline{a}}= a, ~~~ \overline{ab}= \overline{b}.\overline{a}, ~~~ \overline{1} = 1.$$

$R$-modules are understood to be left $R$-modules, unless a right $R$-module action is specified.

\begin{definition}\label{definition1}

\begin{itemize}
\item[(i)] If $M$ is an $R$-module, then the \textbf{transpose} $R$-module is the right $R$-module denoted by $M^{t}$ with the same additive group and with
$$M^t \times R \to M^t ; (x, a) \mapsto \overline{a} x.$$
\item[(ii)] The {\bf dual} of the $R$-module is
$$M^* = \textnormal{Hom}_R(M, R)$$
with $R$ acting by
$$R \times M^*  \to M^* ;  (a, f) \mapsto (x \mapsto (f(x). \overline{a})).$$
\item[(iii)] The \textbf{slant map} is the $\zz$-module morphism defined for  $R$-modules $M$ and $N$ by
$$M^t \otimes_R N \to \textnormal{Hom}_R(M^*, N) ; ~~ x \otimes y \mapsto (f \mapsto \overline{f(x)}.y)$$
\end{itemize}
\end{definition}

\begin{proposition} The slant map is an isomorphism for f.g. projective $R$-modules $M$ and $N$.
In particular for $N=R$ the slant map is a natural $R$-module isomorphism
$$M \to M^{**} ; x \to (f \mapsto \overline{f(x)})$$
which identifies $M^{**}  \cong M.$
\end{proposition}

\begin{proposition}
\begin{itemize}
\item[(i)] For any $R$-modules $M$, $N$ $\textnormal{Hom}_R(M, N^*)$ is the abelian group of sesquilinear pairings
$$\lambda: M \times N \to R ~ ; ~ (x, y) \mapsto \lambda(x, y),$$
such that for all $x, x' \in M$, $y, y' \in N$, $r, s \in R$
\begin{itemize}
\item[(a)] $\lambda(x+x', y) ~=~ \lambda(x, y) + \lambda(x', y) \in R$,
\item[(b)] $\lambda(x, y+y') ~ = ~ \lambda(x, y)+ \lambda(x, y') \in R$,
\item[(c)] $\lambda(rx, sy) ~=~ s\lambda(x, y)\overline{r} \in R.$
\end{itemize}
\item[(ii)] \textbf{Transposition} defines an isomorphism
$$T: \textnormal{Hom}_{R}(M, N^*) \to \textnormal{Hom}_R(N, M^*); \lambda \mapsto T\lambda,$$
with $T \lambda(x, y) = \overline{\lambda(y, x)} \in R.$
\end{itemize}
\end{proposition}

\begin{definition}
A finite f.g. projective $R$-module chain complex
$$C: ~ \dots \to C_{r+1} \xrightarrow{d} C_{r} \xrightarrow{d}C_{r-1} \to \dots $$
is \textbf{$n$-dimensional} if each $C_r$ $(0 \leq r \leq n)$ is a finitely generated projective $R$-module and $C_r=0$  for $r < 0$ and $r > n.$
\end{definition}

For an $R$-module chain complex we write $C^r= (C_r)^* ~~ (r \in \zz).$

\begin{definition}
\begin{itemize}
\item[(i)] For an $R$-module chain complex $C$, the \textbf{dual} $R$-module chain complex $C^{-*}$ is
$$d_{C^{-*}} : (C^{-*})_r = C^{-r} \to (C^{-*})_{r-1} = C^{-r+1}. $$
\item[(ii)] The \textbf{$n$-dual} $R$-module chain complex $C^{n-*}$ is
$$d_{C^{n-*}} : (C^{n-*})_r = C^{n-r} \to (C^{n-*})_{r-1} = C^{n-r+1}. $$
\end{itemize}
\end{definition}

The $n$-fold suspension of the dual $S^nC^{-*}$ is isomorphic to the $n$-dual $C^{n-*}$.

If a chain complex $C$ is $n$-dimensional then its $n$-dual is also $n$-dimensional.

\begin{proposition} \label{slant-map} For finite f.g. projective $R$-module chain complexes $C$ and $D$, the \textbf{slant map} is an isomorphism
$$
\begin{array}{ccc}
C^t \otimes_R D & \to &\textnormal{Hom}_R(C^{-*}, D) \\
x \otimes y & \mapsto & (f \mapsto \overline{f(x)}y).
\end{array}
$$
\end{proposition}

 \begin{definition} \label{involution} Let $C$  be an $R$-module chain complex, where $R$ is a ring with involution. Let $\epsilon = \pm 1$. Define an involution on $C^t \otimes_{R} C$ by the \textbf{$\epsilon$-transposition} map
$$
\begin{array}{ccc}
T_{\epsilon} : C^t_p \otimes_R C_q & \to & C^t_q \otimes_R C_p \\
x \otimes y & \mapsto & (-1)^{pq} y \otimes \epsilon x,
\end{array}
$$
with $x \in C_p$ and $y \in C_q.$
 \end{definition}

$T_{\epsilon}:C^t\otimes_RC \to C^t\otimes_{R}C$ is an automorphism of the $\zz$-module chain complex such that $(T_{\epsilon})^2=1$, giving $C^t\otimes_RC$ a $\zz[\zz_2]$-module structure.

Using the identification via the slant map, we can see that for finite dimensional $C$ the transposition  map $T_{\epsilon}$ is an involution isomorphism
$$
\begin{array}{ccc}
T_{\epsilon} : \textnormal{Hom}_R(C^p, C_q) & \to & \textnormal{Hom}_R(C^q, C_p) \\
\phi & \mapsto & (-1)^{pq} \epsilon \phi^*.
\end{array}
$$

Mostly $\epsilon = 1$ for us, but we shall need $\epsilon = -1$ in Chapter \ref{Pontryagin-squares chapter}.

Note that the transposition involution induces a $\bb{Z}_2$-action on $C^t \otimes_R C.$



\begin{lemma}
A cycle $f \in \textnormal{Hom}_R(C, D)_n$ is a chain map (up to sign)
$$f : S^nC \to D. $$
\end{lemma}

Note that $H_n(\textnormal{Hom}_R(C, D)) = H_0(\textnormal{Hom}_R(S^nC, D))$ is the $\zz$-module of chain homotopy classes of $R$-module chain maps $S^nC \to D.$

$H_n(C^t \otimes_R D)$ is the $\bb{Z}$-module of chain homotopy classes of chain maps $C^{n-*} \to D$, since using the slant map we can identify
$$H_n(C^t \otimes_R C) = H_n( \textnormal{Hom}_R(C^{-*}, C)) = H_0(\textnormal{Hom}_R(C^{n-*}, C)).$$

A cycle $\phi \in \textnormal{Hom}_R(C^{-*}, C)_n = (C^t \otimes_R C)_n$ is a chain map $\phi : C^{n-*} \to C.$

\vspace{5pt}

Let $W$ be the standard free $\bb{Z}[\bb{Z}_2]$-module resolution of $\bb{Z}$,

$$W: \dots \bb{Z}[\bb{Z}_2] \xrightarrow{1+T} \bb{Z}[\bb{Z}_2] \xrightarrow{1-T}\bb{Z}[\bb{Z}_2] \xrightarrow{1+T} \bb{Z}[\bb{Z}_2]  \to 0$$
and let $\widehat{W}$ be the complete resolution,
$$\widehat{W}: \dots \bb{Z}[\bb{Z}_2] \xrightarrow{1+T} \bb{Z}[\bb{Z}_2] \xrightarrow{1-T}\bb{Z}[\bb{Z}_2] \xrightarrow{1+T} \bb{Z}[\bb{Z}_2]  \to \dots $$

We  know from Definition \ref{involution} that the transposition involution $T_{\epsilon}$ provides a $\bb{Z}_2$-action on $C^t \otimes_R C$, so we can define the $\bb{Z}$-module chain complexes
$$
\begin{array}{ccl}
W^\% C & =& \textnormal{Hom}_{\bb{Z}[\bb{Z}_2]}(W, C^t \otimes_R C) \\
W_\% C & = & W \otimes_{\zz[\zz_2]}(C^t \otimes_R C) \\
\widehat{W}^\% C & = &\textnormal{Hom}_{\bb{Z}[\bb{Z}_2]}(\widehat{W}, C^t \otimes_R C) \\
\end{array}
$$
In this context the chain complex $C$ is finite-dimensional, so we can use the slant map from Definition \ref{slant-map} to identify the tensor product $C^t \otimes_R C$ in the above expressions with $\textnormal{Hom}_R(C^{-*}, C)$.

\begin{definition} \label{structures}
\begin{itemize}
\item[(i)] An $n$-dimensional \textbf{$\epsilon$-symmetric structure} on a finite dimensional $R$-module chain complex $C$ is a cycle
$$\phi \in (W^{\%}C)_n, $$
that is, a collection of morphisms
$\left\{\phi_s \in \textnormal{Hom}_R(C^{n-r+s}, C_r) \vert r\in \bb{Z}, s \in \bb{Z} \right\}$
such that
$$
\begin{array}{r}
d \phi_s + (-1)^r \phi_s d^* + (-1)^{n+s-1}(\phi_{s-1}+ (-1)^sT_{\epsilon} \phi_{s-1})=0 :  \\
C^{n-r+s-1} \to C_r,
\end{array}
$$
with $s \geq 0$ and $\phi_{-1} =0$.

\item[(ii)] An $n$-dimensional \textbf{$\epsilon$-quadratic structure} of a finite dimensional $R$-module chain complex $C$ is a cycle
$$\psi \in (W_{\%}C)_n$$
that is, a collection of morphisms $\{ \psi_s : C^{n-r-s} \to C_{r} \vert \hspace{10pt}s \geq 0\}$
such that
$$
\begin{array}{r}
d \psi_s + (-1)^r\psi_s d^* + (-1)^{n-s-1}( \psi_{s+1} +(-1)^{s+1} T_{\epsilon}\psi_{s+1}) =0:  \\
C^{n-r-s-1} \to C_r
\end{array}
$$
with $s\geq 0$.
\item[(iii)]

An $n$-dimensional \textbf{$\epsilon$-hyperquadratic structure} on a finite dimensional $R$-module chain complex is a cycle
$$\theta \in (\widehat{W}^{\%} C)_n$$
that is, a collection of morphisms
$\{ \theta_s  \in  \textnormal{Hom}_R(C^{n-r+s}, C_r) \vert r \in \zz, s\in \zz \}$
such that,
$$
\begin{array}{r}
d \theta_s + (-1)^r \theta_s d^* + (-1)^{n+s-1}(\theta_{s-1} + (-1)^s T_{\epsilon}\theta_{s-1})=0 : \\
 C^{n-r+s-1} \to C_r
\end{array}$$
with $s \in \zz.$
\end{itemize}
\end{definition}

In definition \ref{structures} each $\phi_s$ is a chain homotopy between $\phi_{s-1}$ and $T_{\epsilon}(\phi_{s-1}),$ each $\psi_s$ is a chain homotopy between $\psi_{s-1}$ and $T_{\epsilon}(\psi_{s-1}),$ and each $\theta_s$ is a chain homotopy between $\theta_{s-1}$ and $T_{\epsilon}(\theta_{s-1}).$

\begin{definition} \begin{itemize}
\item[(i)] Two $n$-dimensional $\epsilon$-symmetric structures are equivalent if they differ by the boundary of a chain in $(W^{\%}C)_{n+1}.$
\item[(ii)] Two $n$-dimensional $\epsilon$-quadratic structures are equivalent if they differ by the boundary of a chain in $(W_{\%}C)_{n+1}.$
\item[(iii)] Two $n$-dimensional $\epsilon$-hyperquadratic structures are equivalent if they differ by the boundary of a chain in $(\widehat{W}^{\%}C)_{n+1}.$
\end{itemize}
\end{definition}

\begin{definition}
\begin{itemize}
\item[(i)] $Q^n(C, \epsilon)$ is the abelian group of equivalence classes of $n$-dimensional symmetric structures on $C$.
\item[(ii)] $Q_n(C, \epsilon)$ is the abelian group of equivalence classes of $n$-dimensional quadratic structures on $C$.
\item[(iii)] $\widehat{Q}^n(C,\epsilon)$ is the abelian group of equivalence classes of $n$-dimensional hyperquadratic structures on $C$
\end{itemize}
\end{definition}

Note that the following isomorphisms hold
$$
\begin{array}{ccl}
Q^n(C, \epsilon) & \cong & H_n(W^{\%} C) = H_n(\textnormal{Hom}_{\bb{Z}[\bb{Z}_2]}(W, C^t \otimes_R C)) = H^n(\bb{Z}_2 ; C^t \otimes_{R} C) \\
Q_n(C, \epsilon) & \cong & H_n(W_\% C) = H_n(W \otimes_{\zz[\zz_2]}(C^t \otimes_R C)) = H_n(\zz_2 ; C^t \otimes C)
\\
\widehat{Q}^n(C, \epsilon) & \cong & H_n(\widehat{W}^\% C) =  H_n(\textnormal{Hom}_{\bb{Z}[\bb{Z}_2]}(\widehat{W}, C^t \otimes_R C)) = \widehat{H}^n(\zz_2 ; C^t \otimes_R C).
\end{array}
$$

There is a long exact sequence of $Q$-groups,
$$\dots \to Q_n(C,\epsilon) \xrightarrow{1+T_{\epsilon}} Q^n(C, \epsilon) \xrightarrow{J} \widehat{Q}^{n}(C,\epsilon) \xrightarrow{H} Q_{n-1}(C, \epsilon) \to \dots $$
with
$$
\begin{array}{cl}
1+T_{\epsilon} : &(W_{\%} C)_n \to W^{\%}(C)_n \\
& \{\psi_s \in (C^t \otimes_R C)_{n-s} \vert s \geq 0 \} \mapsto \{((1+T_{\epsilon}) \psi)_s = \left\{
	\begin{array}{ll}
		(1+T_{\epsilon}) \psi_0  & \mbox{if } s = 0 \\
		0 & \mbox{if } s \geq 1
	\end{array}
\right.\} \\
J : &(W^{\%} C)_n \to  \widehat{W}^{\%}(C)_n \\
& \{\phi_s \in (C^t \otimes_R C)_{n+s} \vert s \geq 0 \} \mapsto \{(J \phi)_s = \left\{
	\begin{array}{ll}
		\phi_s  & \mbox{if } s \geq 0 \\
		0 & \mbox{if } s \leq -1
	\end{array}
\right.\} \\
H : &(\widehat{W}^{\%} C)_n \to  W^{\%}(C)_{n-1} \\
& \{\theta_s \in (C^t \otimes_R C)_{n+s} \vert s \in \zz \} \mapsto \{(H \theta)_s = \theta_{-s-1} \vert s \geq 0 \}. \\

\end{array}
$$




\begin{definition}
An $n$-dimensional  
$\epsilon$-symmetric  complex $(C, \phi)$ is \textbf{ Poincar\'e} if $\phi_0: C^{n
-*}\to C$ is a chain equivalence.
\end{definition}
The symmetrization chain map is
$$1+T_{\epsilon}: W_{\%}(C) \to W^{\%}(C)$$
with $\phi_s = (1+T_{\epsilon}) \psi_0$ if $s=0$ and $\phi_s=0$ if $s \geq 1.$
\begin{definition}
An $n$-dimensional $\epsilon$-quadratic complex $(C, \psi)$ is \textbf{Poincar\'e} if the symmetrization $(1+T)\psi_0$ is a chain equivalence.
\end{definition}


\begin{definition}\label{algebraic-Wu}
The \textit{Wu classes} of the $\epsilon$-symmetric structure of a chain complex over a ring with involution $R$ are defined by the function
$$v_r(\phi): H^{n-r}(C) =H_0(\textnormal{Hom}_R(C, S^{n-r}R)) \to  Q^n(S^{n-r}R, \epsilon) ; x \mapsto (x \otimes x)(\phi_{n-2r}).$$
 \end{definition}

The following are the definitions of $\epsilon$-symmetric, $\epsilon$-quadratic and $\epsilon$-(symmetric, quadratic) Poincar\'e pairs, which we shall need later on. The definitions that I include here are quoted from \cite{algpoinc}. A relevant concept in these definitions is that of algebraic mapping cone $\mc{C}(f)$ of $f: C \rightarrow D$, which is defined by
$$d_{\mc{C}(f)}=\left(\begin{array}{cc} d_D & (-1)^{r-1}f \\ 0 & d_C \end{array} \right): \mc{C}(f)_r = D_r \oplus C_{r-1} \to \mc{C}(f)_{r-1} = D_{r-1} \oplus C_{r-2}.$$

\begin{definition}
An $(n+1)$-dimensional \textbf{$\epsilon$-symmetric} Poincar\'e pair over $R$ $$(f : C \to D, (\delta \phi, \phi))$$ consists of
\begin{itemize}
\item an $n$-dimensional $R$-module chain complex $C$,
\item an $(n+1)$-dimensional $R$-module chain complex $D$,
\item a chain map $f : C \to D,$
\item a cycle $(\delta \phi, \phi) \in C(f^{\%} : W^{\%} C \to W^{\%}D)_{n+1} = (W^{\%}D)_{n+1} \oplus  (W^{\%}C)_{n}$
\end{itemize}
such that the $R$-module chain map $D^{n+1-*} \to \mc{C}(f)$ defined by
$$(\delta \phi, \phi)_0 = \left(\begin{array}{c} \delta \phi_0 \\ \phi_0 f^* \end{array} \right) : D^{n+1-r} \to \mc{C}(f)_r = D_r \oplus C_{r-1}$$
is a chain equivalence.
\end{definition}

\begin{definition}
An $(n+1)$-dimensional \textbf{$\epsilon$-quadratic} Poincar\'e pair over $R$
$$(f : C \to D, (\delta \psi, \psi))$$ consists of
\begin{itemize}
\item an $n$-dimensional $R$-module chain complex $C$,
\item an $(n+1)$-dimensional $R$-module chain complex $D$,
\item a chain map $f : C \to D,$
\item a cycle $(\delta \psi, \psi) \in C(f_{\%} : W_{\%} C \to W_{\%}D)_{n+1} = (W_{\%}D)_{n+1} \oplus  (W_{\%}C)_{n}$
\end{itemize}
such that the $R$-module chain map $D^{n+1-*} \to \mc{C}(f)$ defined by
$$(1+T_{\epsilon})(\delta \psi, \psi)_0 = \left(\begin{array}{c} (1+T_{\epsilon})\delta \psi_0 \\ (1+T_{\epsilon})\psi_0 f^* \end{array} \right) : D^{n+1-r} \to \mc{C}(f)_r = D_r \oplus C_{r-1}$$
is a chain equivalence.
\end{definition}

\begin{definition}
An $(n+1)$-dimensional \textbf{$\epsilon$-(symmetric, quadratic)} Poincar\'e pair over $R$
$$(f : C \to D, (\delta \phi, \psi))$$ consists of
\begin{itemize}
\item an $n$-dimensional $R$-module chain complex $C$,
\item an $(n+1)$-dimensional $R$-module chain complex $D$,
\item a chain map $f : C \to D,$
\item a cycle $(\delta \phi, \psi) \in C((1+T_{\epsilon})f_{\%} : W_{\%} C \to W^{\%}D)_{n+1} = (W^{\%}D)_{n+1} \oplus  (W_{\%}C)_{n}$
\end{itemize}
such that the $R$-module chain map $D^{n+1-*} \to \mc{C}(f)$ defined by
$$(\delta \phi, (1+T_{\epsilon})\psi)_0 = \left(\begin{array}{c} \delta \phi_0 \\ (1+T_{\epsilon})\psi_0 f^* \end{array} \right) : D^{n+1-r} \to \mc{C}(f)_r = D_r \oplus C_{r-1}$$
is a chain equivalence.
\end{definition}

\subsection{A geometric application: the symmetric construction}

In the geometric case $X$ will be a connected space with fundamental group $\pi_1(X) = \pi$ and $\wtX$ its universal cover.
The ring will be $R= \zz[\pi]$ and the involution is defined by
$$\overline{\color{White}{a}}~:~R \to R ; ~~ \sum_{g \in \pi} n_g g \mapsto \sum_{g \in \pi} n_g g^{-1}.$$
We set $\epsilon =1$, and write $T_{\epsilon}=T$ and $Q^{n}(C, \epsilon)= Q^n(C).$

The symmetric construction is described in detail in \cite{atsII}. Here we will give a brief summary of the construction together with the definition of symmetric $L$-theory, which is  also described in detail in \cite{atsII}. In the geometric version, the $\epsilon$ from the previous section is $1$ and is left out of the notation of the $Q$-groups.

The diagonal map $\Delta : X \to X \times X$ was used by Lefschetz on the level of homology to identify the Poincar\'e duality isomorphisms of an oriented manifold with the intersection numbers of submanifolds. The Alexander-Whitney-Steenrod symmetric construction $\Delta:C(X)\to W^{\%}C(X)$ captures cup products, the Steenrod squares and the Pontryagin squares, and is the key to the symmetry properties of Poincar\'e duality. (There are analogues for primes $p \neq 2$ but only $p=2$ concerns us here.)

Given a CW complex $X$, let $C(X)$ be its singular chain complex of free $\zz$-modules.

\begin{theorem} (Eilenberg-Zilber)
Let $X$ and $Y$ be topological spaces and let $X \times Y$ be the product space of both. There exists a natural chain homotopy equivalence
$$ EZ_0 : C(X \times Y) \simeq C(X) \otimes_{\zz} C(Y)$$
with natural higher chain homotopies
$$EZ_i : C(X \times Y)_r \to (C(X) \otimes_{\zz} C(Y))_{r+i} ~~~ (i \geq 1)$$
such that
$$  \partial EZ_{i+1} + (-1)^{i} EZ_{i+1}\partial ~= ~  EZ_i ~T + (-1)^{i+1}T~ EZ_i : C(X \times Y)_r \to (C(X) \otimes_{\zz} C(Y))_{r+i} ~~ (i \geq 0),$$
where the transposition isomorphisms are
$$T: X \times Y \to Y \times X; (x, y) \mapsto (y, x)$$
and
$$T= T_1: C(X) \otimes_{\zz} C(Y) \to C(Y) \otimes_{\zz} C(X) ; x \otimes y \mapsto \pm y \otimes x$$
\end{theorem}

Note that $EZ_{i+1} : EZ_{i} \simeq T(EZ_i)$ is a chain homotopy, and that  for $X=Y$ the collection of $\{ EZ_i \}$
defines a natural chain map
$$EZ ~ : ~ C(X \times X) \to \textnormal{Hom}_{\zz[\zz_2]}(W, C(X) \otimes_{\zz} C(X)).$$

\begin{proof} See \cite{Eilenberg-Zilber}.
\end{proof}
The diagonal maps on the chain complex level are induced by the diagonal map of a space
\begin{align*}
\Delta : X  \to X \times X ; x  \mapsto (x,x).
\end{align*}
This map has a $\pi_1(X)-$equivariant generalisation for the universal cover $\widetilde{X}$ of $X$.
Writing $\pi = \pi_1(X)$ and $\pi \times \widetilde{X} \to \widetilde{X}, (g, \widetilde{x}) \mapsto g \widetilde{x} $, the generalisation is given by
$$\widetilde{\Delta} / \pi : \widetilde{X} / \pi = X \to (\widetilde{X} \times \widetilde{X}) / \pi =(\widetilde{X} \times \widetilde{X}) / \left((x, y) \sim (gx, gy) \right) = (\widetilde{X} \times_{\pi} \widetilde{X}),   $$
so that
\begin{displaymath}
 \xymatrix{ & \widetilde{X} \times_{\pi} \widetilde{X} \ar[dd]  \\
X \ar[ur]^{\widetilde{\Delta}/ \pi  } \ar[dr]_{\Delta}& \\
 & X \times X}
\end{displaymath}






From the diagonal map $X \to \wtX \times_{\pi} \wtX$ we get a $\zz[\pi]$-module chain map
$$\Delta_* : C(X) \to C( \wtX \times_{\pi} \wtX). $$
Composing this map with the Eilenberg-Zilber map we obtain a chain equivalence
$$EZ \circ \Delta_* :C(X) \to C( \wtX \times_{\pi} \wtX) \to C(\wtX)^t \otimes_{\bb{Z}[\pi]} C(\wtX),$$
In general we take $\zz[\pi]$-modules to be left $\zz[\pi]$-modules. The superscript $t$ denotes the involution $\bar{g} = g^{-1}$ on $\zz[\pi]$
being used to define the transposition isomorphism which sends the left $\zz[\pi]$-module $C(\widetilde{X})$ to a right $\zz[\pi]$-module $C(\widetilde{X})^{t}$ in order to form the tensor product,
$$\Delta_0: C(\wtX) \to C(\wtX)^t \otimes_{\bb{Z}} C(\wtX). $$
Tensoring with $\bb{Z}$ on the left we get
$$\Delta_0 : \bb{Z} \otimes_{\bb{Z}[\pi]} C(\wtX) \to \bb{Z} \otimes_{\bb{Z}[\pi]}( C(\wtX)^t \otimes_{\bb{Z}} C(\wtX)).$$
This is equivalent to a natural chain map
$$ \Delta_0 : C(X) \to C(\wtX)^t \otimes_{\bb{Z}[\pi]} C(\wtX)$$
which is the chain complex version of the map $X \to \wtX \times_{\pi} \wtX$ from the diagram  above.
This also has natural higher chain homotopies
$$ \Delta_i : C(X) \to (C(\widetilde{X})^t \otimes_{\bb{Z}[\pi]} C(\widetilde{X}))_{r+i} ~~ (i \geq 1)$$
such that
$$\partial \Delta_{i+1} + (-1)^i \Delta_{i+1} \partial = (T+ (-1)^{i+1} ) \Delta_{i}: C(X)_r \to (C(\tilde{X})^t \otimes_{\zz[\pi]} C(\tilde{X}) )_{r+i}.$$

The map $\Delta_i : C(X) \to (C(\widetilde{X})^t \otimes_{\bb{Z}[\pi]} C(\widetilde{X}))_{r+i} ~~ (i \geq 1)$ is a chain homotopy between  $\Delta_{i-1}$ and $T \Delta_{i-1}$, constructed by acyclic model theory.
The symmetric construction is the natural chain map
$$
\phi_X= \{ \Delta_i \vert i \geq 0 \} : C(X) \to W^{\%}C(\wtX) = \textnormal{Hom}_{\bb{Z}[\bb{Z}_2]} (W, C(\wtX)^t \otimes_{\bb{Z}[\pi]} C(\wtX))
$$
which induces group morphisms in
\begin{equation}\label{symmetric construction}
\phi_X : H_n(X) \to H_n( W^{\%}C(\wtX)) = Q^n(C(\wtX)).
\end{equation}


\section{Symmetric, quadratic and hyperquadratic $L$-theory}
In the remainder of this chapter we shall give a summarized presentation of symmetric, quadratic and hyperquadratic $L$-theory. Here we only need to consider $\epsilon=1,$
although in the general theory $\epsilon  = \pm 1$ is also allowed.
The origin of  these theories is the classical Witt group of symmetric bilinear forms over a field $k$, which were first introduced in \cite{Witt}.  Quadratic Witt groups are the surgery obstruction groups introduced by Wall in \cite{Wall-book} (First published in 1970 and re-edited in 1999). Witt groups over a ring with involution were introduced in \cite{atsI}.

We shall be dealing with left $R$-modules, where $R$ is a ring with involution.
$\textnormal{Hom}_R(A, B)$ denotes the additive group of $R$-module morphisms $f: A \to B.$

As in Definition \ref{definition1}, the \textbf{dual $R$-module} of a module $V$ is $V^*  \cong \textnormal{Hom}_R(V, R)$ with $R$ acting by
$$R \times \textnormal{Hom}_R(V, R) \to \textnormal{Hom}_R (V, R) ; ~ (v, r) \mapsto (x \mapsto f(x). \overline{v})$$

$(V, \lambda)$ is a finitely generated free $R$-module $V$ with a morphism $\lambda: V \to V^*=\textnormal{Hom}_R(V, R)$.

The morphism $\lambda: V \to V^*$ can be expressed as pairing $\lambda: V \times V \to R$.

The form is $\epsilon$-symmetric if $\epsilon\overline{\lambda(x, y)} = \lambda(y, x) \in R$ for all $x, y \in V$. The form is symmetric if $\epsilon=1$ in this expression.

The form is \textbf{nonsingular} if $\lambda$ is an isomorphism and \textbf{nondegenerate} if $\lambda$ is injective.

 A \textbf{Lagrangian subspace} of a symmetric form $(V, \lambda)$ is a subspace $L$ such that $\lambda(L, L)=0$ and $\textnormal{dim}(L)= (\frac{1}{2})\textnormal{dim}(V)$.
A form which admits a Lagrangian is called \textbf{metabolic}.

Two nonsingular symmetric bilinear forms $(V,\lambda),(V',\lambda')$ are \textbf{Witt equivalent} if there exists an isomorphism
$(V,\lambda) \oplus (U,\mu) \cong (V',\lambda') \oplus (U',\mu')$ with $(U,\mu),(U',\mu')$ forms which admit Lagrangians.

The \textbf{Witt group $L^0(R)$} over a ring with involution $R$ is the abelian group of Witt equivalence classes of nonsingular symmetric bilinear forms  over $R$, with the group operation corresponding to the orthogonal direct sum of forms.
A nonsingular symmetric form $(V, \lambda)$ over a ring $R$, with $V$ a f.g free $R$-module and $\lambda: V \times V \to R$, represents $0$ in the Witt group $L^0(R)$ if it contains a Lagrangian subspace.

\vspace{5pt}

Witt groups can also be defined for skew-symmetric forms, and for quadratic forms over a ring with involution $R$.
$L^0(R)$ is the Witt group of non-singular symmetric forms over $R$. $L_0(R)$ is the Witt group of non-singular quadratic forms over $R$.

\subsection{Symmetric $L$-groups $L^*(R)$}

Take $R$ to be a ring with involution.
Then the symmetric $L$-group $L^n(R)$ is the abelian group of cobordism classes of $n$-dimensional symmetric complexes $(C, \phi)$ over $R$. Addition is given by direct sum.

The computation of the symmetric $L$-groups of $\bb{Z}$ which is given in the following definition will be particularly important for the rest of the thesis.
\begin{proposition} (\hspace{-1pt}\cite{atsI}) The symmetric $L$-groups of $\bb{Z}$ are $4$-periodic
$$
L^n(\bb{Z}) = \left\{
	\begin{array}{ll}
		\bb{Z} \textnormal{ (signature) } & \mbox{if } n \equiv 0 \pmod{4} \\
		\bb{Z}_2 \textnormal{ (de Rham invariant) } & \mbox{if } n \equiv 1 \pmod{4} \\
             0  & \mbox{if } n \equiv 2 \pmod{4} \\
		0 & \mbox{if } n \equiv 3 \pmod{4} \\
	\end{array}
\right.
 $$
for $n \geq 0.$
\end{proposition}

We shall also deal with the $L$-groups of chain complexes over $\bb{Z}[\pi_1(X)]$. The symmetric signature of Mishchenko \cite{Mishchenko} is defined for an $n$-dimensional geometric Poincar\'e complex $X$  by the symmetric Poincar\'e cobordism class
$$\sigma^*(X) = (C(\wtX), \phi_X) \in L^n(\bb{Z}[\pi_1(X)]).$$

\subsection{Quadratic $L$-groups $L_*(R)$}

The cobordism group of $n$-dimensional quadratic complexes over a ring with involution $R$ is denoted by $L_n(R).$ These groups are the surgery obstruction groups of Wall. We shall give no further details about them here apart from the computation over $\zz,$ and refer the reader to \cite{ Ker-Mil,   Wall-book, atsI, atsII} for an extensive treatment.
\begin{proposition}

It is proved in \cite{atsII} that the quadratic $L$-groups are $4$-periodic, that is, there is an isomorphism
$$L_n(R) \xrightarrow{\cong} L_{n+4}(R).$$
With the definitions in \cite{atsI, atsII},  the symmetric  $L$-groups are in general not $4$-periodic, although there are morphisms
$$\bar{S}^2 : L^n(R) \to L^{n+4}(R) \textnormal{~(double skew-suspension)~}$$
and
$$1+T : L_n(R) \to L^n(R) \textnormal{~(symmetrization)~} $$
which are isomorphisms modulo $8$-torsion for all $n \geq 0$.

 The quadratic $L$-groups over $\bb{Z}$ are given by
$$
L_n(\bb{Z}) =
\left\{
	\begin{array}{ll}
		\bb{Z} \textnormal{ (signature/8) } & \mbox{if } n \equiv 0 \pmod{4} \\
		0 & \mbox{if } n \equiv 1 \pmod{4} \\
             \bb{Z}_2  \textnormal{ (Arf invariant) }  & \mbox{if } n \equiv 2 \pmod{4} \\
		0& \mbox{if } n \equiv 3 \pmod{4} \\
	\end{array}
\right.
$$
\end{proposition}

\subsection{The hyperquadratic $L$-groups $\widehat{L}^*(R)$ }

The symmetric, quadratic and hyperquadratic $L$-groups are related by an exact sequence
$$ \dots \longrightarrow L_n(R) \overset{1+T}{\longrightarrow} L^n(R) \overset{J}{\longrightarrow} \widehat{L}^n(R) \overset{\partial}{\longrightarrow} L_{n-1}(R) \longrightarrow \dots,$$
where $T$ is the involution and $1+T$ is the symmetrization map.
The symmetrization maps $1+T: L_{*}(R) \to L^*(R)$ are isomorphisms modulo $8$ torsion, so that the hyperquadratic $L$-groups $\widehat{L}^*(R)$ are $8$-torsion groups.

Recall that
$L^0(Z)$ is the Witt group of nonsingular symmetric forms over $\zz$ and is isomorphic to $\zz$, the isomorphism given by the signature, and $\widehat{L}^0(Z) \cong \zz_8$ with the isomorphism given by the hyperquadratic signature.
The map $J: L^0(\bb{Z}) \longrightarrow \widehat{L}^0(\bb{Z})$ sends a form $(F, \phi)$ to $\phi (v, v)$, where $v \in F$ is a characteristic element, that is an element such that
for any $u \in F$ it holds that
$$\phi(u, u) = \phi(u, v) \in \bb{Z}_2.$$

The  map $\partial: \widehat{L}^{n}(R) \to L^{n}(R)$ is the boundary map. This boundary is described great detail for low dimensions in \cite[page 546]{BanRan}.

The hyperquadratic $L$-groups can be interpreted as the \textbf{relative group} in this exact sequence. This interpretation gives rise to definition \ref{hyper}.

\begin{definition}\label{hyper} The hyperquadratic $L$-groups $\widehat{L}^n(R)$ are the cobordism groups of $n$-dimensional (symmetric, quadratic) Poincar\'e pairs $(f: C \to D, (\delta \phi, \psi))$ over a ring with involution $R$ such that the $R$-module chain map
$$(\delta \phi_0, (1+T) \psi_0): D^{n-*} \rightarrow \mc{C}(f) $$
is a chain equivalence.
\end{definition}

\begin{remark} As we shall see, the $\widehat{L}$-groups can also be defined as cobordism groups of algebraic normal complexes $(C, \phi, \gamma, \theta)$. We will come back to this definition after introducing chain bundle structures.
\end{remark}

In general the symmetric and hyperquadratic $L$-groups $L^*(R)$ and $\widehat{L}^*(R)$ are not $4$-periodic. Nevertheless, they are $4$-periodic in their simply-connected version, that is when $R= \bb{Z}$.
The computation of the various simply-connected $L$-groups is given in Proposition 4.3.1 in \cite{exactseqRan}.
\begin{proposition}
 The hyperquadratic $L$-groups over $\bb{Z}$ are given by
$$
\widehat{L}^n(\bb{Z}) =
\left\{
	\begin{array}{ll}
		\bb{Z}_8 \textnormal{ (signature$\pmod{8}$) } & \mbox{if } n \equiv 0 \pmod{4} \\
		\bb{Z}_2 \textnormal{ (de Rham invariant) } & \mbox{if } n \equiv 1 \pmod{4} \\
             0  & \mbox{if } n \equiv 2 \pmod{4} \\
		\bb{Z}_2  \textnormal{ (Arf invariant) }& \mbox{if } n \equiv 3 \pmod{4} \\
	\end{array}
\right.
$$
\end{proposition}

Other computations of $L$-groups that will be relevant in chapter 2 are the algebraic $L$-groups of $R= \zz_2$,

\begin{align*}
L_n(\bb{Z}_2) = &
\left\{
	\begin{array}{ll}
		\bb{Z}_2 \textnormal{ (Arf invariant) } & \mbox{if } n \equiv 0 \pmod{2} \\
             0  & \mbox{if } n \equiv 1 \pmod{2} \\
			\end{array}
\right. \\
L^n(\bb{Z}_2) = &
\left\{
	\begin{array}{ll}
		\bb{Z}_2 \textnormal{ (rank$\pmod{2}$) } & \mbox{if } n \equiv 0 \pmod{2} \\
             0  & \mbox{if } n \equiv 1 \pmod{2} \\
			\end{array}
\right. \\
\widehat{L}^n(\zz_2) = & \zz_2.
\end{align*}
with $1+T =0 : L_n(\zz_2) \to L^n(\zz_2).$

\section{Normal complexes}

An $n$-dimensional geometric Poincar\'e complex $X$  is a space with the homotopy type of a finite $CW$ complex which satisfies Poincar\'e duality, that is, it has a fundamental class $[X] \in H_n(X)$ such there is an isomorphism given by the cap product,
$$[X] \cap - : H^* (\wtX) \to H_{n-*}(\wtX),$$
where $\wtX$ is the universal cover of $X.$
Geometric Poincar\'e complexes were defined by Wall in \cite{pcx-wall}.

Let $G(j)$ denote the set of homotopy equivalences of $S^{j-1}$ with the compact-open topology. $BG(j)$ classifies homotopy $S^{j-1}$-bundles and $BSG(j)$ classifies oriented homotopy $S^{j-1}$-bundles. (These bundles are also called ``spherical fibrations").

If $X$ is an $n$-dimensional geometric Poincar\'e complex, a regular neighbourhood $N$ of an embedding $X \subset R^{n+j}$ ($j$ large) determines the Spivak normal fibration
$$\nu_X ~:~ (D^j,S^{j-1}) \to (N,\partial N) \to X$$
which has a classifying map
$$\nu_X : X \to BSG(j)$$
See \cite{Spivak}.

The geometric Poincar\'e complex $X$ has a Spivak normal structure
$$(\nu_X: X \to BSG(j), \rho_X : S^{n+j} \to T(\nu_X)),$$
where $T(\nu_X)= N / \partial N$ denotes the Thom space of $\nu_X.$

An $n$-dimensional \textit{normal} complex $X$ is a space with the homotopy type of a finite $CW$ complex together with a spherical fibration $\nu_X : X \to BSG(j)$ and a map $\rho_X : S^{n+j} \to T(\nu_X)$.
The map $\rho_X$ determines a stable homotopy class
$$\rho_X \in \pi_{n+j}^s(T(\nu_X)) = \lim_{\underset{k}{\longrightarrow}} \pi_{n+j+k}(\Sigma^k T(\nu_X)).$$
The Hurewicz image of $\rho_X \in \pi_{n+j}^s(T(\nu_X))$ is called the fundamental class $[X] \in H_n(X)\cong H_{n+j}(T(\nu_X))$, but the cap product does not necessarily give an isomorphism.

Normal complexes were defined by Quinn in \cite{Quinn-normal}.

A $4k$-dimensional normal complex $X$ has both a signature $\sigma^*(X) \in \zz$,
which is the signature of the symmetric form
$$
\begin{array}{rl}
H^{2k}(X) / torsion \times H^{2k}(X) / torsion & \to \zz; \\
(x, y) & \mapsto \langle x \cup y, [X] \rangle,
\end{array}
$$
and a normal signature $\widehat{\sigma}^*(X) \in \widehat{L}^{4k}(\zz)=\zz_8$. Both these signatures  are homotopy invariants. If $X$ is Poincar\'e the normal signature is the mod $8$ reduction of the signature. In general this is not case, see \cite{Mod8}.

\subsection{Chain bundles and algebraic normal complexes}\label{chain-bundles}

As mentioned before, the $\widehat{L}$-groups can also be defined as the cobordism groups of normal complexes $(C, \phi, \gamma, \theta)$.
In the first place we need to define what is meant by a chain bundle structure on a symmetric complex. References giving a detailed exposition of the algebraic theory of chain complexes and chain bundle theory are \cite{atsI}, \cite{exactseqRan}, \cite{bluebook}, \cite{algpoinc}, \cite{Weiss-Ker}. Here we will summarize the main definitions.


\begin{definition}
A \textbf{chain bundle} over a finite dimensional chain complex $C$ is a $0$-dimensional hyperquadratic structure
$$\gamma \in (\widehat{W}^{\%} C^{0-*})_0. $$
\end{definition}

\begin{definition} An algebraic \textbf{normal structure}  $(\gamma, \theta)$ on an $n$-dimensional symmetric complex $(C, \phi)$ is a chain bundle $\gamma \in (\widehat{W}^{\%}C^{0-*})$ together with an equivalence of $n$-dimensional hyperquadratic structures on $C$,
$$\theta : J(\phi) \to (\widehat{\phi}_0)^{\%}(S^n \gamma)$$
as defined by a chain $\theta \in (\widehat{W}^{\%}C)_{n+1}$ such that
$$J(\phi) - (\widehat{\phi}_0)^{\%}(S^n \gamma) = d \theta \in (\widehat{W}^{\%}C)_{n}$$
\end{definition}
\begin{definition} An algebraic normal complex $(C, \phi, \gamma, \theta)$ is a symmetric complex $(C, \phi)$  together with a chain bundle structure $(\gamma, \theta),$ where $(\gamma, \theta)$ are as defined above.
\end{definition}

Each $n$-dimensional symmetric \emph{Poincar\'e} complex $(C, \phi)$ has a unique equivalence class of normal structures $(\gamma, \theta)$. The image of the equivalence class of symmetric structures $[\phi] \in Q^n(C)$ under the composition
$$Q^n(C) \xrightarrow{J} \widehat{Q}^n(C) \xrightarrow{((\widehat{\phi}_0)^{\%})^{-1}} \widehat{Q}^n(C^{n-*}) \xrightarrow{(S^n)^{-1}} \widehat{Q}^0(C^{0-*})$$
is the equivalence class of bundles $[\gamma] \in  \widehat{Q}^0(C^{0-*}).$

The Wu classes of the chain bundle $(C, \gamma)$ over a ring with involution $R$ are $R$-module morphisms,
\begin{align*}
v_r(\gamma) : H_r(C) & \rightarrow  \widehat{H}^r(\bb{Z}_2, R) \\
(x : S^rR \to C) & \mapsto \gamma_{-2r}(x)(x).
\end{align*}

\subsection{(Symmetric, quadratic) pairs and normal complexes}
The  one-one correspondence between $n$-dimensional algebraic normal complexes and the homotopy classes of $n$-dimensional (symmetric, quadratic) Poincar\'e  pairs is carefully described in \cite[section 7]{algpoinc}

An $n$-dimensional normal complex $(C, \phi, \gamma, \theta)$ determines an $n$-dimensional (symmetric, quadratic) Poincar\'e pair $(\partial C \to C^{n-*}, (\delta \phi, \psi))$ with $\partial C = \mathcal{C}(\phi_0 : C^{n-*} \to C)_{*+1}  $
Conversely, an $n$-dimensional (symmetric, quadratic) Poincar\'e pair $(f: C \to D, (\delta \phi, \psi))$ determines an $n$-dimensional algebraic normal complex $(\mathcal{C}(f), \phi, \gamma, \theta)$ with $\gamma \in \widehat{Q}^0(\mathcal{C}(f)^{-*})$. \\



From this one-one correspondence it can be deduced that there are two alternative definitions of the hyperquadratic $L$-groups over a ring with involution $R$.

\begin{definition} Let $R$ be a ring with involution. The hyperquadratic $L$-group $\widehat{L}^n(R)$ can be defined in one two equivalent ways,
\begin{itemize}
\item[(i)] the cobordism group of $n$-dimensional algebraic normal complexes $(C, \phi, \gamma, \theta)$ over $R$,
\item[(ii)] the cobordism group of $n$-dimensional (symmetric, quadratic) Poincar\'e pairs over $R$.
\end{itemize}
\end{definition}

\subsection{Classifying chain bundle over a ring $R$ with involution}
Chain bundles  $\gamma$ over a chain complex $(C, \gamma)$ over a ring with involution $R$ are classified by the $R$-module chain maps $f: C \rightarrow B(\infty)$ and  $\chi: \gamma \rightarrow \beta(\infty)$. Here $(B(\infty), \beta(\infty))$ is the universal chain bundle and thus has the property that the Wu classes give $R$-module isomorphisms,
$$v_r(\beta(\infty)): H_r(B(\infty)) \xrightarrow{\cong} \widehat{H}^r(\bb{Z}_2, R).$$

The universal chain bundle over a ring $R$ has the property that
$$H_0(\textnormal{Hom}_{\bb{Z}}(C, B(\infty))) \to \widehat{Q}^0(C^{-*}) $$
is an isomorphism for any finite dimensional chain complex $C$.

In particular if $R = \bb{Z}$, then the universal chain bundle $(B(\infty), \beta(\infty))$ is defined by

$$
d_{B(\infty)} =
\left\{
	\begin{array}{ll}
		2 & \mbox{if } r \textnormal{ is odd} \\
		0 & \mbox{if } r  \textnormal{ is even} \\
          	\end{array}
\right. : B(\infty)_r = \bb{Z} \to B(\infty)_{r-1} = \bb{Z},
$$
That is, the universal chain bundle over $\bb{Z}$ is given by
$$B(\infty): \dots \rightarrow B_{2k+2} = \bb{Z} \xrightarrow{0} B_{2k+2} = \bb{Z} \xrightarrow{2}B_{2k} = \bb{Z} \xrightarrow{0} B_{2k-1} = \bb{Z} \rightarrow \dots$$
and
$$
\beta(\infty)_s =
\left\{
	\begin{array}{ll}
		1 & \mbox{if } 2r = s \\
		0 &  \textnormal{otherwise} \\
          	\end{array}
\right. : B(\infty)_{r-s} = \bb{Z} \to B(\infty)^{-r} = \bb{Z}.
$$

\begin{definition} A \textbf{normal $(B, \beta)$-structure}, $(\gamma, \theta, f, \chi)$, on an $n$-dimensional symmetric complex $(C, \phi)$ is  a normal structure $(\gamma, \theta)$ together with a chain bundle map,
$$(f, \chi): (C, \gamma) \longrightarrow (B, \beta).$$
\end{definition}
Note that in the previous definition $(B, \beta)$ is not assumed to be the universal chain bundle.

\begin{definition} (\hspace{-1pt}\cite{Weiss-Ker}) The \textbf{$\langle B, \beta \rangle$-structure $L$-groups $L\left<B, \beta \right>^n (R)$} are the cobordism groups of $n$-dimensional $\langle B, \beta \rangle$-normal complexes $(C, \phi, \gamma, \theta, f, \chi)$ over $R$.
\end{definition}

The quadratic $L$-groups and  the $(B, \beta)$-structure $L$-groups fit into an exact sequence
$$ \dots \to L_n(R) \to L\left<B, \beta \right>^n(R) \to \widehat{L}\left<B, \beta \right>^n(R) \to L_{n-1}(R) \to \dots$$
\subsection{Twisted $Q$-groups}

\begin{definition}
An \textbf{$n$-dimensional symmetric structure $(\phi, \theta)$} on a chain bundle $(C, \gamma)$ is an $n$-dimensional symmetric structure $\phi \in (W^{\%}C)_n$ together with an equivalence of $n$-dimensional hyperquadratic structures on $C$,
$$\theta : J(\phi) \to (\widehat{\phi}_0)^{\%}(S^n \gamma)$$
as defined by a chain $\theta \in (\widehat{W}^{\%}C)_{n+1}$ such that
$$J(\phi) - (\widehat{\phi}_0)^{\%}(S^n \gamma) = d \theta \in (\widehat{W}^{\%}C)_{n}$$
\end{definition}

An equivalence of $n$-dimensional symmetric structures on $(C, \gamma)$ is defined by an equivalence of symmetric structures together with an equivalence of hyperquadratic structures on $C$.

\begin{definition} (\hspace{-1pt}\cite[page 24]{algpoinc})
The \textbf{twisted quadratic $Q$-group $Q_n(C, \gamma)$} is the abelian group of equivalence classes of $n$-dimensional symmetric structures on a chain bundle $(C, \gamma)$.
\end{definition}
The twisted quadratic $Q$-groups fit into an exact sequence
$$
\begin{array}{rcccccccc} \dots \widehat{Q}^{n+1}(C)  & \longrightarrow & Q_n(C, \gamma) & \overset{N_{\gamma}}{\longrightarrow} & Q^n(C) & \overset{J_{\gamma}}{\longrightarrow} & \widehat{Q}^n(C) & \overset{H_{\gamma}}{\longrightarrow} Q_{n-1}(C, \gamma) & \longrightarrow \dots \\
            &                 & (\phi, \theta) & \longmapsto & \phi & \longmapsto & {\tiny J(\phi)- (\phi_0)^{\%}(S^n \gamma)} & & \\
\theta & \longmapsto &(0, \theta) &   &&  & \\
\end{array}
$$
In the untwisted case, that is with $\gamma = 0$, we have the usual exact sequence of $Q$-groups,
$$ \dots \longrightarrow Q_n(C) \overset{1+T}{\longrightarrow} Q^n(C) \overset{J}{\longrightarrow} \widehat{Q}^n(C) \overset{H}{\longrightarrow} Q_{n-1}(C) \longrightarrow \dots $$

\begin{definition} The class $(\phi, \theta) \in Q_n(\mathcal{C}(f), \gamma)$ is the \textbf{algebraic normal invariant} of \\$(f: C \to D, (\delta \phi, \psi))$.
\end{definition}

\begin{example} An $n$-dimensional normal space $(X, \nu_X, \rho_X)$ determines a cycle for a chain bundle $\gamma(\nu)$ and a cycle for an element $(\phi, \theta)$, hence defining an algebraic normal complex $(C(X), \phi, \gamma(\nu_X), \theta(X))$.
An oriented spherical fibration $\nu_X : X \to BSG$ determines an equivalence class of chain bundles $(C(X), \gamma(\nu))$ over $\zz$.

The Wu classes of the chain bundle $v_{r}(\gamma (\nu)) \in H_{r}(X ; \mathbb{Z}_2) \to \mathbb{Z}_2$ correspond to the usual Wu classes $v_{r}(\nu) \in H_{r}(X ; \mathbb{Z}_2)$, and the symmetric construction extends to the following commutative diagram,
$$
\xymatrix{
& \rho_X \in \pi_{n+j}(M(\nu)) \ar[d] \ar[r]^{\hspace{-15pt}Hurewicz} &\dot{H}_{n+j}(M(\nu)) =  H_{n}(X)\ar[d]^{\textnormal{\textit{symmetric construction}}}  &  \\
& (\phi, \theta) = \rho_{\gamma} \in Q_{n}(C(X), \gamma(\nu)) \ar[r] & \hspace{5pt} Q^{n}(C(X)) \\
}
$$
\end{example}

\begin{proposition}\label{Proposition: Weiss} (\hspace{-1pt}\cite[Proposition 6.6]{Weiss-Ker})
The homomorphism
$$\widehat{L} \langle B(\infty), \beta(\infty) \rangle^n(R) \overset{\cong}{\longrightarrow} Q_n(B(\infty), \beta(\infty)) $$
is an isomorphism for $n \geq 0$.
\end{proposition}
\begin{proof}

An element $(C, \phi, \gamma, \theta, g, \chi) \in \widehat{L}\langle B, \beta \rangle^n (R)$ is an $n$-dimensional algebraic normal complex, i.e. an $n$-dimensional symmetric complex $(C, \phi)$ with a normal $(B, \beta)$-structure $(\gamma, \theta)$, where
$$\gamma \in (\widehat{W}^{\%} C^{0-*})_0 $$ is a chain bundle, and
$$\theta \in (\widehat{W}^{\%} C)_{n+1}  \textnormal{}$$
is and equivalence of $n$-dimensional hyperquadratic structures on $C$,
$$\theta: J(\phi) \rightarrow (\widehat{\phi}_0)^{\%}(S^n \gamma), $$
and $(g, \chi)$ is a chain bundle map, such that $g: C \rightarrow B$ and $\chi : \gamma \rightarrow \beta$. \\
This chain bundle map $$(g, \chi) : (C, \gamma) \rightarrow (B, \beta)$$
induces a map in the twisted $Q$-groups
$$(g, \chi)_{\%} : Q_n(C, \gamma) \rightarrow Q_n (B, \beta).$$

The map $\widehat{L}\langle B, \beta \rangle^n (R) \rightarrow Q_n (B, \beta) $ sends $(C, \phi, \gamma, \theta, g, \chi)$ to $(g, \chi)_{\%} (\phi, \theta)$.\\
\\
Now suppose that $[(\phi, \theta)] \in Q_n (B, \beta) $ so that,
\begin{align*}
Q_n (B, \beta) \longrightarrow & \widehat{L}\langle B, \beta \rangle^n (R) \\
[(\phi, \theta)] \longmapsto  & (B, \phi, \gamma, \theta, Id, Id).
\end{align*}

Consequently the following composition of maps gives the identity,
$$
\xymatrix{
Q_n (B, \beta) \ar[r] \ar@/^2pc/[rr]^{Id} &
\widehat{L}\langle B, \beta \rangle^n (R) \ar[r] &
Q_n (B, \beta)  &
}
$$
\vspace{-15pt}
$$\hspace{80pt}[(\phi, \theta)] \longmapsto (B, \phi, \gamma, \theta, Id, Id) \mapsto (Id, Id)_{\%}(\phi, \theta) = (\phi, \theta). \hspace{20pt}
$$

Furthermore we can also consider the composite,
$$\widehat{L}\langle B, \beta \rangle^n (R) \longrightarrow
Q_n (B, \beta) \longrightarrow \widehat{L}\langle B, \beta \rangle^n (R). $$
Start by considering an element in $\widehat{L}\langle B, \beta \rangle^n (R) $, say $(C, \phi, \gamma, \theta, g, \chi) \in \widehat{L}\langle B, \beta \rangle^n (R)$, so that,
\begin{align*}
\widehat{L}\langle B, \beta \rangle^n (R) & \longrightarrow Q_n (B, \beta) \\
 (C, \phi, \gamma, \theta, g, \chi) & \longmapsto (g, \chi)_{\%}(\phi, \theta).
\end{align*}
\
Composition of maps gives,
\begin{align*}
\widehat{L}\langle B, \beta \rangle^n (R) & \longrightarrow
Q_n (B, \beta)  \longrightarrow \widehat{L}\langle B, \beta \rangle^n (R) \\
(C, \phi, \gamma, \theta, g, \chi) & \mapsto (g, \chi)_{\%}(\phi, \theta) \mapsto (B, g(\phi), \gamma, g(\theta), Id, Id).
\end{align*}

Here the map $g: C \rightarrow B$ can be suitably chosen so that $(C, \phi, \gamma, \theta, g, \chi)$ and $(B, g(\phi), \gamma, g(\theta), Id, Id)$ are cobordant, because $(g, \chi)$ is a normal map between them. The fact that they belong to the same cobordism class means that they represent the same element in the group $\widehat{L}\langle B, \beta \rangle^n (R)$,
and hence the composition $\widehat{L}\langle B, \beta \rangle^n (R) \longrightarrow Q_n (B, \beta)  \longrightarrow \widehat{L}\langle B, \beta \rangle^n (R)$ is again the identity.
\end{proof}

The symmetric $(B, \beta)$-structure $\widehat{L}$-groups $\widehat{L}\langle B, \beta \rangle^n (\bb{Z})$ are the cobordism groups of $n$-dimensional $(B, \beta)$-normal complexes over $\bb{Z}$.
For the universal chain bundle $(B(\infty), \beta(\infty))$ over $\bb{Z}$,
$$\widehat{L}\langle B, \beta \rangle^n (\bb{Z}) = \widehat{L}^n (\bb{Z}).$$
Hence there is an isomorphism,
$$ \widehat{L}^n (\bb{Z}) \cong Q_n(B, \beta),$$
where $(B, \beta)$ is the universal chain bundle over $\zz$.

For $n \geq 0$, there is \textbf{a $(v_{n+1}=0)-$universal chain bundle} over a ring with involution $R$ $(B\langle n+1 \rangle, \beta \langle n+1 \rangle)$ which is defined by the following properties,
\begin{itemize}
\item[(i)] The map $v_r(\beta \langle n+1 \rangle) : H_r(B\langle n+1 \rangle) \to \widehat{H}^r(\zz_2; R)$ is an isomorphism for $r \neq n+1,$
\item[(ii)] $H_{n+1}(B \langle n+1 \rangle)= 0.$
\end{itemize}

\begin{definition}\label{v_n+1}(\hspace{-1pt}\cite{algpoinc})
The $(v_{n+1}=0)$-symmetric $L$-groups of ring with involution $R$ are defined as
$$L\langle v_{n+1}\rangle^m (R)= L\langle B\langle n+1 \rangle, \beta \langle n+1 \rangle \rangle^m(R) \textnormal{ with } m \geq 0.$$
\end{definition}

\subsection{Twisted $Q$-groups and the mod $8$ signature of a normal complex}\label{Twisted-Q-groups}

In this section we discuss an expression for the mod $8$ signature in terms of the symmetric and hyperquadratic structures (See Theorem \ref{sign-theta}).
We shall need this theorem in the proof of Proposition \ref{signature-Psq}. First we shall give some definitions that will be relevant for the proof of \ref{sign-theta}.

\begin{definition} (\hspace{-1pt}\cite{Mod8}) \label{chain bundle} With $k \geq 0$ and $m \geq 1$, the chain bundle over $\bb{Z}$
$$(B(k, m), \beta(k, m)) = \mc{C}(d, \chi)$$
is the algebraic mapping cone $\mc{C}(d, \chi)$ of the chain bundle map
$$(d, \chi): (B(k, m)_{2k+1}, 0) \to (B(k,m)_{2k}, \delta),$$
where both $B(k, m)_{2k+1}$ and $B(k, m)_{2k}$ are regarded as chain complexes concentrated in degree $2k$, and
$$
\begin{array}{cc}
 d = 2m :B(k, m)_{2k+1}= \zz & \to B(k, m)_{2k}=\zz \\
\chi= 2m^2: B(k, m)_{2k+1} = \zz  & \to B(k, m)^{2k}= \zz
\end{array}
$$

\end{definition}

\begin{definition} The universal bundle over $\bb{Z}$ can be expressed as
$$(B(\infty), \beta(\infty)) = \sum^{\infty}_ {k=0} (B(k, 1), \beta(k, 1)),$$
with $k \geq 0$.
This is a consequence of the definitions of  chain bundle and universal chain bundle.
\end{definition}

In the proof of Proposition \ref{signature-Psq} we shall also need the following result which is proved in \cite{BanRan}.

\begin{theorem} \label{theta-and-Z8} (\hspace{-1pt}\cite[Proposition 52, Corollary 61]{BanRan})
The $4k$-dimensional $Q$-group of $(B(k, 1), \beta(k, 1))$ is given by the isomorphism
\begin{align*}
Q_{4k}(B(k, 1), \beta(k, 1))& \xrightarrow{\cong}  \bb{Z}_8 \\
(\phi, \theta) & \mapsto \phi_0(1,1) + 2 \phi_1(1, 1) + 4 \theta_{-2}(1,1).
\end{align*}
\end{theorem}

The proof of Proposition \ref{signature-Psq} will be a direct consequence of the following theorem.

\begin{theorem}  \label{sign-theta} (\hspace{-1pt}\cite{Mod8})
The mod $8$ signature of a $4k$-dimensional normal complex over $\bb{Z}$, $(C, \phi, \gamma, \theta)$, is given by
$$\widehat{\sigma}(C, \phi, \gamma, \theta) = \phi_0(v, v) + 2 \phi_1(v, u) + 4 \theta_{-2}(u, u) \in Q_{4k}(B(k, 1), \beta(k, 1)) =  \bb{Z}_8 $$
with $(u,v) \in H_0(\textnormal{Hom}_{\bb{Z}}(C, B(k,1)))$
$$
\xymatrix{
C_{2k+1} \ar[r]^{\hspace{-25pt}u} \ar[d]^{d} & B(k, 1)_{2k+1}= \bb{Z}  \ar[d]^2\\
C_{2k} \ar[r]^{\hspace{-25pt}v}  & B(k, 1)_{2k}= \bb{Z}}
$$
 a chain map representing the $2k$-th Wu class $v_{2k} \in H^{2k}(C ; \bb{Z}_2).$
\end{theorem}

\begin{proof}

The chain bundle $(C, \gamma)$ is classified by a chain bundle map
$$(f, \chi): (C, \gamma) \to (B(\infty), \beta(\infty)) $$
such that
$${\small
\xymatrix{
\txt{$C$} \ar[dr]_f  \ar@{->}[rr]^{v_{2k}(\gamma)}
&
&\txt{$B(k, 1)$}\\
&\txt { $B(\infty) = \sum\limits_{k=0}^{\infty} B(k, 1)$}\ar[ur]}}
$$

The inclusion$$(B(k, 1), \beta(k, 1)) \to (B(\infty), \beta(\infty)) $$ is a chain equivalence so it induces an isomorphism in the twisted $Q$-groups,
\begin{equation}\label{iso-Q}
Q_{4k}(B(k, 1), \beta(k, 1)) \xrightarrow{\cong} Q_{4k}(B(\infty), \beta(\infty)). \end{equation}
In \cite[Proposition 6.6]{Weiss-Ker} and \cite[Proposition 46]{BanRan} there is described an isomorphism
$$Q_{4k}(B, \beta) \xrightarrow{\cong} \widehat{L}^{4k}\langle B, \beta \rangle(\bb{Z}),$$
and for the universal chain bundle $(B(\infty), \beta(\infty))$ over $\bb{Z}$, the symmetric $(B, \beta)$ structure $\widehat{L}$-groups $\widehat{L}^{4k}\langle B(\infty), \beta(\infty) \rangle (\bb{Z})$ are just the hyperquadratic $L$-groups,
$$\widehat{L}^{4k}\langle B(\infty), \beta(\infty) \rangle (\bb{Z}) = \widehat{L}^{4k}(\bb{Z}),$$
so using the isomorphism in \eqref{iso-Q} we have that
$$Q_{4k}(B(k, 1), \beta(k, 1)) \xrightarrow{\cong}  Q_{4k}(B(\infty), \beta(\infty)) \xrightarrow{\cong} \widehat{L}^{4k}(\bb{Z})= \bb{Z}_8.$$
Hence an element $\widehat{\sigma}(C, \phi, \gamma, \theta) \in \widehat{L}^{4k}(\bb{Z})$ also represents an element in $Q_{4k}(B(k, 1), \beta(k, 1)) \xrightarrow{\cong}  \bb{Z}_8$,
$$\widehat{\sigma}(C, \phi, \gamma, \theta) \in Q_{4k}(B(k, 1), \beta(k, 1)) \xrightarrow{\cong}  \bb{Z}_8.$$


By Theorem \ref{theta-and-Z8} we know that the $4k$-dimensional $Q$-group of $(B(k, 1), \beta(k, 1))$ is given by the isomorphism
\begin{align*}
Q_{4k}(B(k, 1), \beta(k, 1))& \xrightarrow{\cong}  \bb{Z}_8 \\
(\phi, \theta) & \mapsto \phi_0(1,1) + 2 \phi_1(1, 1) + 4 \theta_{-2}(1,1)
\end{align*}
and the Wu class $v_{2k}(\gamma) : C \to B(k, 1)$ can be represented as $(u, v) \in H^{2k}(C, \bb{Z}_2)$ with $v \in C^{2k}$ and $u \in \textnormal{Ker}(d^* : C^{2k} \to C^{2k+1})$ such that
$$d^*(v) = 2u \in C^{2k+1}.$$
So there is an induced map,
$$\begin{array}{ccl}
Q_{4k}(C, \gamma ) & \xrightarrow{\left((u, v), \chi \right)^\%} &  Q_{4k}(B(k, 1), \beta(k, 1))  \\
(\phi, \theta) & \longmapsto & \left((u, v), \chi \right)^\%(\phi, \theta) = \phi_0 (v, v) + 2 \phi_1(v, u) + 4 \theta_{-2}(u, u),
\end{array} $$
so that,
$$\widehat{\sigma}(C, \phi, \gamma, \theta) = \phi_0(v, v) + 2 \phi_1(v, u) + 4 \theta_{-2}(u, u) \in Q_{4k}(B(k, 1), \beta(k, 1)) =  \bb{Z}_8. $$

\end{proof}

\chapter{Arf and Brown-Kervaire invariants in algebra}\label{Arf-BK}

\section{Quadratic enhancements}

In this chapter $(V, \lambda)$ will denote a nonsingular symmetric bilinear form over $\zz_2.$ $(V, \lambda)$ is \textbf{isotropic} if  $\lambda(x,x) = 0 \in \zz_2$ for all $x \in V,$ and \textbf{anisotropic} otherwise.
There are two indecomposable forms over $\zz_2$,
$$P=(\zz_2, 1) \textnormal{\hspace{3pt} and \hspace{3pt} } H =\left(\zz_2 \oplus \zz_2, \left(
\begin{matrix}
  0& 1 \\
  1 & 0 \\
   \end{matrix}
\right) \right),$$
\begin{itemize}
\item [(i)] $P$ is anisotropic, $H$ is isotropic, so $P \oplus P \ncong H.$
\item[(ii)] $H$ has complementary lagrangians $\zz_2 \oplus 0$ and $0 \oplus \zz_2.$
\item[(iii)]  The form $P \oplus P \cong \left(\zz_2 \oplus \zz_2, \left( \begin{matrix}
  0& 1 \\
  1 & 1 \\
   \end{matrix}
\right) \right)$ has a lagrangian $\{(x,x ) \in \zz_2 \oplus \zz_2 \vert x \in \zz_2\}$ without a lagrangian complement.
\end{itemize}

The following result is well known.
\begin{proposition} \label{form-decomposition} 
Every $(V, \lambda)$ decomposes as an orthogonal direct sum of copies of
$$P=(\zz_2, 1) \textnormal{\hspace{3pt} and \hspace{3pt} } H =\left(\zz_2 \oplus \zz_2, \left(
\begin{matrix}
  0& 1 \\
  1 & 0 \\
   \end{matrix}
\right) \right),$$
that is,
$$(V, \lambda) = \bigoplus_{p} P \oplus \bigoplus_{k}H.$$

 \end{proposition}

\begin{proof}
For every non-zero element $x \in V$,  either $\lambda(x,x)=0 \in \zz_2$ or $\lambda(x,x)= 1 \in \zz_2$.
If $\lambda(x,x)=1 \in \zz_2$ have $(V,\lambda) = (\langle x \rangle,1) \oplus (\langle x \rangle^{\perp}/\langle x \rangle,[\lambda])$. We can then repeat the process of splitting off anisotropic subspaces (i.e with $\lambda(a, a)=1\in \zz_2$) until we are left with an isotropic subspace, that is, a subspace with $\lambda(b,b)=0 \in \zz_2$.
\end{proof}

\begin{proposition}
With $(V, \lambda)$ as in the previous proposition, there is an isomorphism in symmetric $L$-theory given by,
\begin{align*}
L^0(\zz_2) \longrightarrow & \zz_2 \\
(V, \lambda) \longmapsto & \hspace{2pt} p \hspace{-5pt}\pmod{2}.
\end{align*}
In particular the symmetric forms $P \oplus P$  and $H$ are both Witt equivalent to $0$.
\end{proposition}

\begin{definition} \label{quadratic enhancement}
Let $V$ be a $\bb{Z}_2$-vector space and $\lambda : V \times V \to \bb{Z}_2$ a non-singular symmetric bilinear form.
\begin{itemize}
\item[(i)] A \textbf{$\bb{Z}_2$-valued quadratic enhancement} of $\lambda$ is a function  $h: V \to \bb{Z}_2$ such that for all $x, y \in V$
\begin{equation}h(x + y) = h(x) + h (y) +  \lambda (x , y)  \in \bb{Z}_2.\end{equation}

\item[(ii)] A \textbf{$\bb{Z}_4$-valued quadratic enhancement} of $\lambda$ is a function $q : V \to \bb{Z}_4$ such that for all $x, y \in V$
\begin{equation}\label{Z4-enhancement}
q(x + y) = q (x) + q (y) + i \lambda (x, y) \in \bb{Z}_4
\end{equation}
where $i = 2 : \bb{Z}_2  \to \bb{Z}_4$ and $x, y \in V$.
\end{itemize}
\end{definition}

\begin{remark}
The quadratic enhancements $h$ and $q$ cannot be recovered uniquely from the  $\bb{Z}_2$-valued symmetric bilinear pairing $\lambda : V \times V \to \bb{Z}_2$. Let $j: \zz_4 \to \zz_2$ denote mod $2$ reduction and let $q : V \to \bb{Z}_4$ be a $\zz_4$-valued enhancement of the bilinear pairing $\lambda$. Then the composition $jq$ is the mod $2$ reduction of $q$ and is given by
$$j q(x) = \lambda(x,x) \in \zz_2.$$
(This is a consequence of Equation \ref{Z4-enhancement}).

For every $\lambda$ there exists a non-unique $\zz_4$-enhancement $q$, but there may not exist a $\zz_2$-enhancement $h$. Furthermore every $h$ determines a $q$ by the identity $q=2h$.

If $q(V) \subseteq 2 \zz_2 \subset \zz_4$, we can define the $\zz_2$-enhancement $h:V\to \zz_2; ~x \mapsto q(x)/2$ such that $2h=q.$
\end{remark}

We now discuss what are the possible $\zz_2$ and $\zz_4$-valued 	quadratic enhancements of such a nonsingular symmetric bilinear form $(V, \lambda)$.

\section{$\zz_2$-enhancements $h$ of $(V, \lambda)$ and the Arf invariant}
The following result is well-known.
\begin{proposition}\label{iff-lambda-0}
Let $(V, \lambda)$ be a non-singular symmetric form over $\zz_2$; then the following conditions are equivalent
\begin{itemize}
\item[(i)] $(V, \lambda)$ admits a $\zz_2$-valued quadratic enhancement $h: V \to \zz_2,$
\item[(ii)] $\lambda(x, x)= 0\in \zz_2$ for all $x \in V,$ (i.e. $(V, \lambda)$ is isotropic)
\item[(iii)]  $\lambda$ can be split into an orthogonal sum of hyperbolics,
$ \left(
\begin{matrix}
  0& 1 \\
  1 & 0 \\
   \end{matrix}
\right),
$ that is,  $(V, \lambda) = \bigoplus \limits_{k} H.$
\end{itemize}
\end{proposition}
\begin{proof}
$(i) \Rightarrow (ii)$ If $(V, \lambda)$ admits a quadratic enhancement $h$, this means that $h(x+y) = h(x)+h(y) + \lambda(x, y) \in \zz_2$, setting $y=x$, we have that $\lambda(x,x)= 2h(x)=0 \in \zz_2.$

$(ii) \Rightarrow (iii)$ For each $x \in V$ there exists a $y \in V$ with $\lambda(x,y)=1$. On the pair $\{x, y\}$, the form is hyperbolic so  the pair has an orthogonal complement in $V$. That is, $V = \langle x \rangle^{\perp} / \langle x \rangle \oplus \left(\zz_2 \oplus \zz_2, \left(
\begin{matrix}
  0& 1 \\
  1 & 0 \\
   \end{matrix}
\right) \right)$,
We can then repeat the process of splitting off hyperbolics until the whole of $V$ has been expressed as a direct sum of hyperbolics.
$(iii) \Rightarrow (i)$ Clear from the definition of $h: V \to \zz_2.$
\end{proof}

A $\zz_2$-quadratic enhancement $h$ of $H$ will take value $0$ when evaluated on 0,  $h(0)=0,$ and can take values either $1$ or $0$ when evaluated on $x$ and $y$, where $\{x, y \}$ form a basis for $ \zz_2 \oplus \zz_2$. So we can form four $\zz_2$-enhancements on $H$, which we shall denote as $h^{0,0}$, $h^{0,1}$, $h^{1, 0}$ and $h^{1,1}$ according to their values on $x$ and $y$.

From the definition of a quadratic enhancement (Definition \ref{quadratic enhancement}) we know that the equation $h(x+y)-h(x)-h(y) = \lambda(x, y) \in \zz_2$ must be satisfied and  $\lambda(x, y)=1$.

A change of basis from $\{x, y \}$ to $\{x, x+y\}$ establishes an isomorphism from $h^{0,0}$ to $h^{0,1}$. Similarly a change of basis from $\{x, y \}$ to $\{x+y, y \}$ establishes an isomorphism from $h^{0,0}$ to $h^{1, 0}$. So there are two isomorphism classes of $\zz_2$-quadratic enhancements on $H,$

\begin{minipage}[c]{0.4\linewidth}
\vspace{-10pt}
\begin{equation*} \label{values-of-h}
h^{0,0} = \left\{ \begin{array}{rr}x & \mapsto 0 \\
y & \mapsto 0 \\
x+y & \mapsto 1
\end{array} \right.,
\end{equation*}

\end{minipage}
\begin{minipage}[c]{0.30\linewidth}

$$
h^{1, 1} =\left\{ \begin{array}{rr}x & \mapsto 1 \\
y & \mapsto 1 \\
x+y & \mapsto 1
\end{array} \right..
$$
\vspace{10pt}
\end{minipage}

\begin{proposition} (\hspace{-1pt}\cite{Browder})
Every nonsingular $\zz_2$-quadratic form $(V, \lambda, h)$ can be expressed \textbf{non-uniquely} as
$$(V, \lambda, h) =  \left(\bigoplus_{n} h^{0,0} \oplus \bigoplus_{k} h^{1, 1}\right).$$
\end{proposition}
\begin{proof}
The decomposition is \textbf{non-unique} because there is an isomorphism
$$h^{1,1} \oplus h^{1,1} \cong h^{0,0} \oplus h^{0,0}.$$
This isomorphism is achieved by a change of basis.
Note that the number of copies of $h^{1,1}$ counted modulo $2$ is unique.
\end{proof}

\subsection{The Arf invariant for $(V, \lambda, h)$}
Let $V$ be a $\bb{Z}_2$-vector space and $\lambda$ a non-singular symmetric bilinear form
$$\lambda: V \otimes V \to \bb{Z}_2,$$
and let $h: V \to \bb{Z}_2$ be a $\bb{Z}_2$-valued quadratic enhancement of this bilinear form which satisfies the following property,
$$h(x+y) = h(x) + h(y) +\lambda (x, y) \in \bb{Z}_2, $$
as in the previous section.
The Arf invariant was first defined in \cite{Arf} as follows,
\begin{definition} With a symplectic basis $\left\{ e_1, \dots, e_k, \bar{e}_1, \dots, \bar{e}_k   \right\}$ for $V$, the Arf invariant is defined as
$$\textnormal{Arf}(h) = \sum\limits_{j=1}^{k} h(e_j)h(\bar{e}_j) \in \bb{Z}_2.$$
\end{definition}

\begin{theorem} (\hspace{-1pt}\cite{Arf}) \label{Arf-L}
\begin{itemize}
\item[(i)] Two non-singular $\zz_2$-valued quadratic forms on a $\zz_2$-vector space V of finite dimension are Witt equivalent if and only if they have the same Arf invariant.
\item[(ii)] There is an isomorphism in quadratic $L$-theory,
\begin{align*}
L_0(\zz_2) \xrightarrow{\textnormal{Arf}} & \zz_2 \\
(V, \lambda, h) \longmapsto & \hspace{2pt} n \hspace{-5pt}\pmod{2},
\end{align*}
where $n$ is the number of copies of $h^{1,1}$  in $(V, \lambda, h) = \bigoplus \limits_n h^{1,1}\oplus \bigoplus \limits_k h^{0,0}.$
\end{itemize}
\end{theorem}

The Arf invariant counts the number of copies of $h^{1,1}$ modulo $2$, and $\textnormal{Arf}(h) =1$ if $h$ sends a majority of elements to $1$.

Later on we will make use of the Brown-Kervaire invariant, which we now define. The definition of the Brown-Kervaire invariant was formulated in \cite{Brown} and was  used in \cite{BanRan}. This invariant has an important relationship with the signature \cite{Morita}, which we shall discuss in section \ref{BK-and-L} and in chapter \ref{Morita-chain-cx}.

\section{$\zz_4$-enhancements $q$ of $(V, \lambda)$ and the Brown-Kervaire invariant}
\begin{proposition}\label{always-q}
Every symmetric bilinear form $(V, \lambda)$ over $\zz_2$ admits a $\zz_4$-valued quadratic enhancement $q: V \to \zz_4.$
\end{proposition}
\begin{proof} Every $q$ is constructed by lifting the $\zz_2$-valued symmetric bilinear form $\lambda$ to a symmetric form $\Lambda$ over $\zz$ and setting $q(x)= \Lambda(x,x) / 4 \in \zz_4$. Clearly such a lift is non-unique since adding even integers to the diagonal entries in $\Lambda$ does not change $\lambda$ but does change $q$.
\end{proof}

We shall now discuss what  the possible $\zz_4$-enhancements of both $H$ and $P$ are.

$H$ has four enhancements depending on the values that $q$ takes on $(x, y) \in \zz_2 \oplus \zz_2$. We shall denote these as $q^{0,0}$, $q^{0,2}$, $q^{2,0}$ and $q^{2,2}$.

From the definition of a $\zz_4$-quadratic enhancement (Definition \ref{quadratic enhancement}) we know that the equation $q(x+y)-q(x)-q(y) = 2\lambda(x, y) \in \zz_4$ must be satisfied and  $\lambda(x, y)=1.$ A change of basis from $\{x, y \}$ to $\{x, x+y\}$ establishes an isomorphism from $q^{0,0}$ to $q^{0,2}$. Similarly a change of basis from $\{x, y \}$ to $\{x+y, y \}$ establishes an isomorphism from $q^{0,0}$ to $q^{2, 0}$. So there are two isomorphism classes of $\zz_4$-quadratic enhancements on $H,$ (see \cite[Lemma 7]{Pink}),

\begin{minipage}[c]{0.4\linewidth}
\vspace{-10pt}
\begin{equation*} \label{values-of-h}
q^{0,0} =\left\{ \begin{array}{rr}x & \mapsto 0 \\
y & \mapsto 0 \\
x+y & \mapsto 2
\end{array} \right. ,
\end{equation*}

\end{minipage}
\begin{minipage}[c]{0.30\linewidth}

$$
q^{2,2} =\left\{ \begin{array}{rr}x & \mapsto 2 \\
y & \mapsto 2 \\
x+y & \mapsto 2
\end{array} \right..
$$
\vspace{10pt}
\end{minipage}

There are two isomorphism classes of quadratic forms on $P= (\zz_2, 1)$. These are given by $q(x)=1$ or by $q(x)=-1$ and are denoted in \cite{Kirby} as $P_{1}$ and $P_{-1}$ respectively

\begin{proposition}\label{decomposition}
Every $\zz_4$-valued quadratic form over $\zz_2$ $(V, \lambda, q)$ can be expressed non-uniquely as
\begin{equation*}
  (V, \lambda, q) =  \left(\bigoplus_{p_+} P_{1} \oplus  \bigoplus_{p_-} P_{-1}\right) \oplus    \left(\bigoplus_{m} q^{0,0} \oplus \bigoplus_{n} q^{2, 2} \right).
\end{equation*}
\end{proposition}

\begin{proof}
The proof of this proposition can be found in \cite[pages 428, 429]{Pink}. We stress here that the decomposition is non-unique as there are certain isomorphism relations:
\begin{itemize}
\item $q^{0,0} \oplus q^{0,0} \cong q^{2,2} \oplus q^{2,2},$
\item $P_1 \oplus P_1 \oplus P_1 \oplus P_1 \cong  P_{-1} \oplus P_{-1} \oplus P_{-1} \oplus P_{-1},$
\item $q^{0,0} \oplus P_1 \cong P_1 \oplus P_1 \oplus P_{-1},$
\item $q^{2,2} \oplus P_1 \cong  P_{-1} \oplus P_{-1} \oplus P_{-1},$
\item $q^{2,2} \oplus P_{-1} \cong P_1 \oplus P_1 \oplus P_1. $
\end{itemize}
These isomorphisms are achieved by choosing appropriate changes of basis in each case. They are proved in \cite{Pink}.
\end{proof}

We shall also be interested in the classification of $\zz_4$-enhancements up to Witt equivalence.

The forms which admit a lagrangian are referred to in \cite{Matsumoto} as \textit{split} and denoted by $S_i$.

\begin{remark}  (\hspace{-1pt}\cite[page 130]{Matsumoto})
Note that by Definition \ref{Witt-equivalence} two $\zz_4$-enhanced forms $(V, \lambda, q)$ and $(V', \lambda', q')$ belong to the same Witt class if
$$(V, \lambda,q) \oplus S_1 \cong (V', \lambda', q') \oplus S_2,$$
where $S_1$ and $S_2$ are split.
 \end{remark}

\begin{definition} \label{Witt-equivalence}(\hspace{-1pt}\cite{algpoinc}) The group $L\langle v_{1} \rangle^0(\zz_2)$ is the abelian group of Witt equivalence classes of nonsingular $\zz_4$-valued enhancements of nonsingular symmetric forms over $\zz_2.$
\end{definition}
\begin{remark} In the literature, the Witt group of nonsingular  $\zz_4$-valued enhancements of nonsingular symmetric forms over $\zz_2$ is frequently denoted by $W_{\zz_4}(\zz_2).$
\end{remark}

A non-singular quadratic form $(V,\lambda, q)$ represents $0$ in the Witt group of $\zz_4$-enhancements of a nonsingular $\zz_2$-symmetric bilinear form $L\langle v_1 \rangle^0(\zz_2)$, if it contains a subspace $L$ such that $q(L)=0$, $\lambda(L, L)=0$ and $\textnormal{dim}(L)= (\frac{1}{2})\textnormal{dim}(V)$. We call such a subspace $L$ a lagrangian.

\begin{example} \label{split} (\hspace{-1pt}\cite[page 130]{Matsumoto}) $q^{0,0}$ admits a lagrangian hence  $\left(\zz_2\oplus \zz_2, \left(
\begin{matrix}
  0& 1 \\
  1 & 0 \\
   \end{matrix}
\right), q^{0,0}\right)$ is $0$ in the Witt group $L\langle v_1 \rangle^0(\zz_2)$.

Similarly $P_{1} \oplus P_{-1}$  admits a lagrangian, $L= (1, 1)$, so $(\zz_2\oplus \zz_2, 1 \oplus 1, P_1 \oplus P_{-1} )$ is also $0$ in the Witt group.

\end{example}

We shall denote the Witt class of a $\zz_4$-enhanced form $(V, \lambda, q)$ by a bracket notation $[V, \lambda, q].$
\begin{proposition}
The Witt group of $\zz_4$-valued enhancements of a nonsingular $\zz_2$-valued symmetric form $\lambda$ is generated by the Witt class of $[P_{1}]$.
\end{proposition}
\begin{proof}
See \cite[page 130]{Matsumoto}. This will be discussed further in section \ref{Arf&Brown}.

\end{proof}

\begin{example} \label{4P1} (\hspace{-1pt}\cite{Matsumoto})
There is a Witt equivalence
$$[q^{2,2}] =  [4P_{1}] \in L\langle v_1 \rangle^0(\zz_2)$$
These two forms are Witt equivalent because there is an isomorphism of forms
$$q^{2,2} \oplus P_{-1} \cong P_1 \oplus P_1 \oplus P_1  $$
and therefore
$$q^{2,2} \oplus (P_{-1} \oplus P_1) \cong P_1 \oplus P_1 \oplus P_1  \oplus P_1.$$
 $(P_{-1} \oplus P_1)$ admits a lagrangian as was explained in example \ref{split}, hence by Definition \ref{Witt-equivalence} the forms $q^{2,2}$ and $P_1 \oplus P_1 \oplus P_1  \oplus P_1$ are Witt equivalent.
\end{example}

\subsection{The Brown Kervaire invariant for $(V, \lambda, q)$}

Let $V$ be a $\bb{Z}_2$ vector space, $\lambda : V \otimes V \to \bb{Z}_2$ a non-singular symmetric pairing and let $q : V \to \bb{Z}_4$ be a quadratic enhancement of the symmetric form so that
$$q(x +y) = q(x) + q(y)+ i \lambda (x, y) \in \bb{Z}_4,$$
where $i = 2 : \bb{Z}_2  \to \bb{Z}_4$ and $x, y \in V$.
\begin{definition} \label{Gauss-sum-formula} (\hspace{-1pt}\cite{Brown})
The Brown-Kervaire $\textnormal{BK}(V, \lambda, q)$ invariant is defined using a Gauss sum,
$$\sum_{x \in V} i ^{q(x)} = \sqrt{2}^{\textnormal{dim}V} e^{2 \pi i \textnormal{BK}(V,  \lambda, q) / 8},$$
with $i^2=-1$ and $x \in V$.
\end{definition}

\begin{theorem}\label{iso-BK} (\hspace{-1pt}\cite{Brown, Pink, GM}) Two non-singular $\zz_4$-valued quadratic forms on a $\zz_2$-vector space V of finite dimension are Witt equivalent if and only if they have the same Brown-Kervaire invariant. There is an isomorphism in $L$-theory given by
\begin{align*}
L\langle v_1 \rangle^0(\zz_2) & \xrightarrow{\textnormal{BK}} \zz_8 \\
(V, \lambda, q) \mapsto & 4n + p_+ - p_-\pmod{8},
\end{align*}
where $n$, $p_+$ and $p_-$ are the numbers of summands of $q^{2,2}$, $P_1$ and $P_{-1}$ respectively in the decomposition in Proposition \ref{decomposition}.
\end{theorem}

\begin{proof} It was proved in \cite{Brown} that  two non-singular $\zz_4$-valued quadratic forms on a $\zz_2$-vector space $V$ of finite dimension are Witt equivalent if and only if they have the same Brown-Kervaire invariant. Brown also identified the group with $\zz_8$. The $L$-theory interpretation is discussed in detail in \cite{algpoinc}. The formula to compute the Brown-Kervaire invariant of a form in terms of the numbers of summands of $q^{2,2}$, $P_1$ and $P_{-1}$ is proved in \cite{GM}.
\end{proof}

The Brown-Kervaire invariant as an $L$-theory invariant will be discussed further in section \ref{BK-and-L} and in chapter \ref{Morita-chain-cx} together with its relation with the signature and hyperquadratic signature.

\section{Relation between the Arf and Brown-Kervaire invariants} \label{Arf&Brown}

\begin{theorem} (\hspace{-1pt}\cite[Theorem 1.20 (vii)]{Brown})
Let $h$ be a $\zz_2$-valued nonsingular quadratic enhancement of a nonsingular symmetric bilinear $\zz_2$-valued form $\lambda$.
If $q = 2h$, then
$$BK(V, \lambda, q) = 4(\textnormal{Arf}(V, \lambda,  h)) \in \zz_8.$$
\end{theorem}

A Brown-Kervaire invariant that is $0 \in\zz_4$ can always be expressed as an Arf invariant. This will be proved in Proposition \ref{BK-and-4Arf}. Before we state Proposition \ref{BK-and-4Arf}, we shall discuss the relationship between $\zz_2$- and $\zz_4$-valued enhancements of a symmetric form $(V, \lambda)$ over $\zz_2.$
\begin{proposition} \label{relations-q-h}
Let $V$ be a $\bb{Z}_2$ vector space and $\lambda : V \otimes V\to \bb{Z}_2$ a nonsingular symmetric bilinear form over $\zz_2$.
\begin{itemize}
\item[(i)] There exists a $\zz_2$-enhancement $h$ if and only if $\lambda(x,x)=0 \in \zz_2.$
\item[(ii)] There always exists a $\zz_4$-enhancement $q$.
\item[(iii)] For every $\zz_2$-enhancement $h$ there exists a $\zz_4$-enhancement $q$ given by $q = 2h.$
\item[(iv)] \label{Witt-q-4-(iv)} Every $(V, \lambda, q)$ is given by $(V, \lambda, q)= \bigoplus \limits_{m} q^{0,0} \oplus \bigoplus \limits_{n} q^{2, 2}\oplus \bigoplus \limits_{p_+} P_{1} \oplus  \bigoplus \limits_{p_-} P_{-1}.$
If $$4n+p_+-p_- = 0 \textnormal{ or } 4 \in \zz_8,$$
 then $(V, \lambda, q)$ is Witt equivalent to $(V', \lambda', 2h')$ in $L\langle v_1 \rangle^0 (\zz_2),$ where
\begin{align*}
(V', \lambda', 2h') & = \bigoplus_m 2h^{0,0} \oplus \bigoplus_n 2h^{1,1} \\
& =\bigoplus_m q^{0,0} \oplus \bigoplus_n q^{2,2}
\end{align*}
\end{itemize}


\end{proposition}

\begin{proof}
\begin{itemize}
\item[(i)] See Proposition \ref{iff-lambda-0}.
\item[(ii)] See Proposition \ref{always-q}.
\item[(iii)] Any $\zz_2$-enhancement $h$ is isomorphic to
$$h = \bigoplus_{m} h^{0,0} \oplus \bigoplus_{n} h^{1, 1},$$
\vspace{-15pt}
so that
\vspace{-1pt}
\begin{align*}2h &= \bigoplus_{m} 2h^{0,0} \oplus \bigoplus_{n}2h^{1, 1} \\
&= \bigoplus_{m} q^{0,0} \oplus \bigoplus_{n} q^{2, 2}  = q.
\end{align*}

\item[(iv)] If $4n+p_+ - p_- \equiv 0 \textnormal{ or } 4 \pmod8$, then $(V, \lambda, q)$ is isomorphic to
$$(V, \lambda, q) = \bigoplus_{m} q^{0,0} \oplus \bigoplus_{n} q^{2, 2} \oplus \bigoplus_k(P_1 \oplus P_{-1}).$$

If $p_+$ and $p_-$ are both $0$, then $(V, \lambda, q) =\bigoplus_{m} q^{0,0} \oplus \bigoplus_{n} q^{2, 2}$, we know by the previous part that this is isomorphic to $(V, \lambda, 2h')$ with $(V, \lambda,h') = \bigoplus \limits_{m} h^{0,0} \oplus \bigoplus \limits_{n} h^{1, 1}.$

If $p_+$ or $p_-$ (or both) are nonzero and $4n+p_+ -p_- = 0 \textnormal{ or } 4 \in \zz_8,$ then $q$ is isomorphic to
$$ q \cong \bigoplus_{m} q^{0,0} \oplus \bigoplus_{n} q^{2, 2} \oplus \bigoplus_k(P_1 \oplus P_{-1})$$
with $m \geq 0$, $n \geq 0$ and $k = 0 \textnormal{ or } 1.$
That is, the  form can be written in a way such that all $P_1$ and $P_{-1}$ can be paired. Since $(P_1 \oplus P_{-1})$ is a split form admitting a lagrangian, then it is zero in the Witt group $L\langle v_1 \rangle^0(\zz_2),$ so by Definition \ref{Witt-equivalence}
 $$[(V, \lambda, q)] = \left[\bigoplus_{m} q^{0,0} \oplus \bigoplus_{n} q^{2, 2}\right] \in L\langle v_1 \rangle^0(\zz_2).$$
\end{itemize}

\end{proof}

\begin{example}  The Witt equivalence of the forms $4P_1$ and $q^{2,2}$ was discussed in example \ref{4P1}. They both have Brown-Kervaire invariant $4 \in \zz_8.$ In this example we see how they can be expressed in terms of a $\zz_2$-valued enhancement,
\begin{align*}
[4P_1] & = [q^{2,2}]  = [2h^{1,1}] \in  L\langle v_1 \rangle^0(\zz_2).
\end{align*}
\end{example}

\begin{example}
$[3P_1] \in L\langle v_1 \rangle^0(\zz_2)$ is not in the subgroup of $L\langle v_1 \rangle^0(\zz_2)$ generated by $[q^{2,2}]$.
Therefore the $\zz_4$-enhancement $3P_1$ is not Witt equivalent in $L\langle v_1 \rangle^0(\zz_2)$ to a $\zz_4$-enhancement that can be written as $(V, \lambda, 2h)$, where $h$ is a $\zz_2$-valued enhancement.
\end{example}

With the following proposition we describe how a Brown-Kervaire invariant that takes values $0$ and $4$ in $\zz_8$ can always be expressed as a classical $\zz_2$-valued Arf invariant.

\begin{proposition} \label{BK-and-4Arf}
If the Brown-Kervaire invariant of a nonsingular symmetric form over $\zz_2$ with a $\zz_4$-valued quadratic enhancement $(V, \lambda, q)$ is divisible by $4$, then this Brown-Kervaire invariant can be expressed as  the classical Arf invariant of a $\zz_2$-valued form,
$$\textnormal{BK}(V, \lambda, q) =4 \textnormal{Arf}\left(L^{\perp}/L, [\lambda], \frac{\left[q\right] }{2} \right) \in 4\zz_2 \subset \zz_8 $$
where $L^{\perp}= \left \{x \in V \vert  \lambda(x,x) = 0 \in \zz_2 \right\}$ and $L= \langle v \rangle$ with $v$ the characteristic element of $\lambda$ in $V$.
\end{proposition}

\begin{proof}
Let $(V,\lambda,q)$ be a nonsingular symmetric form over $\zz_2$ with a $\zz_4$-valued quadratic enhancement, and denote its dimension by  $n = \textnormal{dim}(V)$.
For any $\zz_4$-enhanced form $(V,\lambda,q)$ there is a unique Wu class $v \in V$ such that
$$\lambda(x,x) = \lambda(x,v) \in \zz_2 \textnormal{ with } x\in V.$$
The following function is linear
$$f : V \longrightarrow \zz_2 ; x \mapsto \lambda(x,x) = jq(x) = \lambda(x,v),$$
where $jq(x)\in \zz_2$ is the mod $2$ reduction of $q(x) \in \zz_4.$

The following identity \eqref{identity} relating the $\zz_4$-quadratic enhancement and the Brown-Kervaire invariant reduced modulo $4$ is a consequence of \cite[theorem 1.1]{Morita} and \cite[Proposition 2.3]{Morita}. (We will discuss these results in more detail in chapter \ref{Morita-chain-cx}),
\begin{equation}\label{identity}[\textnormal{BK}(V,\lambda,q)] = q(v) \in \zz_4.\end{equation}
So $\textnormal{BK}(V,\lambda,q) \in \zz_8$ is divisible by $4$ if and only if $q(v)=0 \in \zz_4$.
If $q(v)=0 \in \zz_4$ then $\lambda(v,v) = 0 \in \zz_2$ and the \textit{ \textbf{Wu sublagrangian}} $L=\langle v \rangle \subset (V,\lambda,q)$
is defined, with $L \subseteq L^\perp = \left\{x \in V|\lambda(x,x)=0 \in \zz_2\right\}$.
The \textit{ \textbf{maximal isotropic subquotient}}
$$(L^{\perp}/L, [\lambda], [q])$$
 has a canonical $\zz_2$-valued quadratic enhancement $[h]: x \mapsto [q(x)]/2$ and
$$\textnormal{BK}(V,\lambda,q) = 4\textnormal{Arf}(L^{\perp}/L,[\lambda],[h]) \subset 4\zz_2 \subset \zz_8.$$
For the dimension of $L^{\perp}/L$ there are two cases, according as to whether $v=0$ or $v\neq 0:$
\begin{itemize}
\item[(i)] If the Wu class is $v=0$ then $(V, \lambda)$ is already isotropic as $L^{\perp}/L = V,$ and $\textnormal{dim}(L^{\perp}/L) = n.$
\item[(ii)] If the Wu class is $v \neq 0$ then $(V, \lambda)$ is anisotropic and $\textnormal{dim}(L^{\perp}/L) = n-2.$

\end{itemize}

\end{proof}

\begin{remark}In \cite[page 42]{algpoinc} there is computed the inclusion of the quadratic $L$-group $L_0(\zz_2) = \zz_2$ defined by the Arf invariant into the Witt group of nonsingular  $\zz_4$-valued quadratic forms over $\zz_2,$ $L\langle v_{1} \rangle^0(\zz_2),$

$$
\xymatrix{ L_0(\zz_2) \ar[d]^{\textnormal{Arf}} \ar[r] &  L\langle v_{1} \rangle^0(\zz_2) \ar[d]^{\textnormal{BK}} \ar[r] & \widehat{L}\langle v_{1} \rangle^0(\zz_2) \ar[d]  \\
\zz_2 \ar[r]^4 & \zz_8 \ar[r] & \zz_4
}
$$

\raggedbottom

\end{remark}




\chapter{Classical and Equivariant Pontryagin squares}\label{Pontryagin-squares chapter}

In this chapter we will review the definition of the classical Pontryagin squares and their relationship with cup-$i$ products and Steenrod squares.
Pontryagin squares provide a quadratic enhancement of the intersection form of a symmetric bilinear form.

The notion of an equivariant Pontryagin square was introduced in \cite{Korzen}. Here we extend this notion to prove that the equivariant Pontryagin squares can also be defined on odd cohomology classes. With this definition we can reconstruct the Pontryagin square of the total space $E$ of a fibration of the form $F^{4m+2} \to E \to B^{4n+2}$ from the equivariant Pontryagin square on the base.
The main tool for the definition of the equivariant Pontryagin square is the exact sequence in Proposition 19 of \cite{BanRan}. We shall start by giving some background on cup-$i$ products and Steenrod squares necessary for the definition of the Pontryagin squares.

\section{Cup-$i$ products and Steenrod squares} \label{cupi}

Recall from Chapter \ref{Foundations} that we defined the chain approximation map
$$\Delta_0 : C(X) \to C(X) \otimes C(X).$$
Writing $[X]$ for the fundamental class of $X$, and using the slant map defined in \ref{slant-map}
we defined
$$\phi_0 := \backslash \Delta_0([X]) : C^{n-r} \to C_r.$$
Recall that  a chain homotopy $\Delta_{i+1}: C \to C \otimes C$ between $\Delta_i$ and $T_{\epsilon}\Delta_{i}$ can be constructed inductively, such that the following relation is satisfied
$$\partial \Delta_{i+1} + (-1)^i \Delta_{i+1} \partial = (T_{\epsilon} + (-1)^{i+1} ) \Delta_{i}.$$
The higher chain homotopies give information about the failure of $\Delta_0$ to be symmetric, and they are used to define the entire symmetric structure on a symmetric complex.
Composing these with the slant map induces,

$$\phi_s := \backslash \Delta_s([X]) : C^{n-r+s} \to C_r.$$

\begin{definition}
The cup product of two cohomology classes $x \in H^p(X)$ and $y \in H^q(X)$ is
$$x \cup y = \Delta_0^*([X])(x \otimes y) \in H^{i+j}(C).$$
\end{definition}

Cup products are signed-commutative  in cohomology $x \cup y = (-1)^{pq}(y \cup x)$,  but they do not commute on the chain level. This is a consequence of the fact that the chain approximation $\Delta_0$ is only defined up to chain homotopy. The transpose $T\Delta_0 $ produces different maps $\phi_0$ on the chain level, which then have the same effect on homology level, as they are homotopic. The non-trivial chain homotopy $\Delta_1: T\Delta_0  \simeq \Delta_0$ gives the failure of the cup product to commute on the chain level.

\begin{definition} (\hspace{-1pt}\cite{Mosher-Tangora})
For each integer $i \geq 0$ define the \textbf{cup-$i$ product} on the chain level as
$$C^p(X) \otimes C^q(X) \longrightarrow C^{p+q-i}(X) : (x, y) \longrightarrow x \cup_i y $$
by the formula $$x \cup_i y = \Delta_i(x \otimes y).$$

On cohomology, the cup-$i$ product of $x \in H^p(X)$ and $y \in H^q(X)$ is given by
$$x \cup_i y = \Delta_i^*([X])(x \otimes y) \in H^{p+q-i}(X).$$
\end{definition}

The higher chain homotopies $\Delta_i$ are strongly related to the Steenrod squares.
\begin{definition} (\hspace{-1pt}\cite{Steenrod, Mosher-Tangora})
The Steenrod square $Sq^i$ is a cohomology operation
$$Sq^i : H^r(X, \bb{Z}_2) \longrightarrow H^{r+i}(X; \bb{Z}_2). $$
On the cochain level, the Steenrod squares are defined as
$$Sq^i  : C^r(X;\mathbb{Z}_2) \longrightarrow  C^{r+i}(X;\mathbb{Z}_2), $$
such that  $Sq^i(x) := x \cup_{r-i} x$, with $x \in C^r(X;\mathbb{Z}_2). $

In particular for $X$ a $2n$-dimensional Poincar\'e space,
$$Sq^{n} : H^{n}(X ; \mathbb{Z}_2) \longrightarrow H^{2n}(X ; \mathbb{Z}_2)= \zz_2$$
is the evaluation of the cup product,
$$x \mapsto \langle x \cup_0 x, [X] \rangle $$
\end{definition}

In chapter \ref{Foundations} we presented the symmetric construction as being given by the natural chain map
$$\phi_X : C(X) \to \textnormal{Hom}_{\bb{Z}[\bb{Z}_2]}(W, C(\wtX)^t \otimes_{\bb{Z}[\pi]} C(\wtX))$$
induced by the diagonal chain approximation $\Delta$, for any space $X.$
The mod $2$ reduction of the composite
$$H_{2n}(X) \xrightarrow{\phi_X} Q^{2n}(C(X)) \xrightarrow{v_i} \textnormal{Hom}_{\bb{Z}}(H^{2n-i}(X), Q^{2n}(S^{2n-i}\bb{Z})) $$
is given by
$$v_i(\phi_X(x))(y) = \langle v_i \cup y, x \rangle= \langle Sq^i(y), x\rangle \in \bb{Z}_2.$$

Steenrod squares are related to Wu classes in the following way.

\begin{definition}
With $X$ a $2n$-dimensional Poincar\'e space, the $i$-th Wu class
$$v_i \in H^{i}(X; \bb{Z}_2)$$
is given by the Steenrod square, so that for all $x \in H^{2n-i}(X ; \bb{Z}_2)$ it holds that,
$$Sq^i(x) =  v_i \cup x \in H^{2n}(X; \zz_2)=\zz_2.$$
(See \cite{Mil-Sta}).
\end{definition}

\section{Classical Pontryagin squares}

Pontryagin first defined the cohomology operation known as Pontryagin square in \cite{Pontryagin}. Pontryagin squares were also carefully studied by Whitehead in \cite{Whitehead2} and \cite{Whitehead}.

Thus with $X$ a space, the Pontryagin square is an unstable cohomology operation
$$\mc{P}_2: H^n(X; \bb{Z}_2) \to H^{2n}(X; \bb{Z}_4).$$
Although the Pontryagin square is defined on modulo $2$ cohomology classes, it cannot be constructed solely from the modulo $2$ cup product structure.

Let  $d^*: C^n(X; \bb{Z}) \to C^{n+1}(X; \bb{Z})$  be the singular cohomology coboundary operator.
We shall represent an element $x \in H^n(X ; \zz_2)$ as a cycle
$$x=(y, z) \in \textnormal{Ker} \left(\left(\begin{array}{cc} d^* & 0 \\   2 & d^* \end{array}    \right) : C^n(X) \oplus C^{n+1}(X) \to C^{n+1}(X) \oplus C^{n+2}(X)\right) $$
Using this notation we define the Pontryagin square as follows.

\begin{definition}
  The Pontryagin square is defined on the cochain level by
\begin{align*}
\mc{P}_2(x) = \cP_2(y, z) &= y \cup_0 y + y \cup_1 d^*y \in H^{2n}(X; \bb{Z}_4) \\
                                       & = y \cup_0 y + 2 y \cup_1 z \in H^{2n}(X; \bb{Z}_4).
\end{align*}
\end{definition}
The construction of the cup-$i$ products was introduced in section \ref{cupi},
$$\cup_i : C^p(X) \otimes C^q(X) \longrightarrow C^{p+q-i}(X) : (u, v) \longrightarrow u \cup_i v. $$

The coboundary formula for the cup-$1$ product, with $u \in C^p(X)$ and $v \in C^q(X)$ is given by
$$d(u \cup_i v) = (-1)^i du \cup_i v + (-1)^{i+p}u \cup_i dv + (-1)^{i+1} u \cup_{i-1} v + (-1)^{pq+1} v \cup_{i-1} u.$$
This formula can be applied to check that $y \cup_0 y + y \cup_1 d^*y \textnormal{ mod} 4 $ is a cocycle mod $4$ and that its cohomology class only depends on that of $x = (y, z) \in H^n(X; \bb{Z}_2)$ (see \cite{Mosher-Tangora})
\begin{definition}\label{Pont-square-definition}
 The Pontryagin square is defined on cohomology by
\begin{align*}\mc{P}_2: H^n(X ; \bb{Z}_2) & \to H^{2n}(X; \bb{Z}_4) \\
                                                 x = (y, z) & \mapsto y \cup_0 y + y \cup_1 d^*y,
\end{align*}
where (y, z) are as defined above.
\end{definition}

The maps in the exact sequence $0 \longrightarrow \mathbb{Z}_2 \overset{i}{\longrightarrow} \mathbb{Z}_4 \overset{r}{\longrightarrow} \mathbb{Z}_2 \longrightarrow  0 $ induce maps in cohomology,
$$\dots \longrightarrow H^{2n}(X; \mathbb{Z}_2)\overset{i^*}{ \longrightarrow} H^{2n}(X; \mathbb{Z}_4) \overset{r^*}{\longrightarrow} H^{2n}(X; \mathbb{Z}_2) \overset{\delta}{\longrightarrow} \dots$$
so
$$r_* \mathcal{P}_2(x) = x \cup x \in H^{n} (X ; \mathbb{Z}_2),$$ where $r_* : H^{n}(X ; \mathbb{Z}_4) \longrightarrow H^{n}(X ; \mathbb{Z}_2)$ is the map induced by the non-trivial map \\ $\mathbb{Z}_4 \longrightarrow \mathbb{Z}_2$.

\begin{remark}

The Steenrod square $Sq^{n}$ is a mod $2$ reduction of the Pontryagin square:
\begin{displaymath}
\xymatrix{ &  H^{2n}(X; \mathbb{Z}_4) \ar[d]^{r^*}  \\
 H^{n}(X; \mathbb{Z}_2) \ar[r]^{Sq^{n}} \ar[ur]^{\mathcal{P}_2} &  H^{2n}(X; \mathbb{Z}_2)
}
\end{displaymath}
\end{remark}
\begin{proposition} \label{on a sum} (\hspace{-1pt}\cite{Mosher-Tangora})
\begin{itemize}
\item[(i)]Let $x$ and $x'$ be cocycles in $H^n(X; \bb{Z}_2)$, where $X$ is a  Poincar\'e space.  The Pontryagin square evaluated on a sum is given by,
$$\mc{P}_2(x+x') = \mc{P}_2(x) + \mc{P}_2(x')+(x \cup x')+(-1)^n(x \cup x') \in H^{2n}(X; \bb{Z}_4).$$
So that in particular for $n$ even,
$$\mc{P}_2(x+x') = \mc{P}_2(x) + \mc{P}_2(x')+j (x \cup x')\in \bb{Z}_4,$$
where $j$  is the non-trivial homomorphism $j: \bb{Z}_2 \to \bb{Z}_4.$

\item[(ii)] The Pontryagin square is a quadratic function with respect to cup product, that is,
$$\mc{P}_2(x+ x') = \mc{P}_2(x) + \mc{P}_2(x') + j(x \cup x') \in H^{2n}(X; \bb{Z}_4)$$
for any $x, x' \in H^{n}(X ; \bb{Z}_2)$, where $j : \bb{Z}_2 \to \bb{Z}_4$ is the non-trivial homomorphism.
\end{itemize}
\end{proposition}
\begin{proof}
This follows directly from the definition of a quadratic enhancement \ref{quadratic enhancement} and the evaluation of the Pontryagin square on a sum of cocycles \ref{on a sum}.
\end{proof}

\section{Algebraic Pontryagin squares}

In this section we present the algebraic analog of the Pontryagin square. The image of the algebraic Pontryagin square lies in the $4k$-dimensional symmetric $Q$-group of a finitely generated free $\bb{Z}$-module concentrated in degrees $2k+1$ and $2k$
$$B(2k, s): \dots \to S^{2k+1}\bb{Z} \xrightarrow{d= s} S^{2k}\bb{Z} \to \dots$$
Following the computations of  $Q$-groups presented in \cite{BanRan}, it can be shown that this $Q$-group is $\zz_4$, when $s=2$.

\begin{definition} (\hspace{-5pt} \cite{Mod8})\label{algebraic-Pont}
The \textbf{$\zz_2$-coefficient Pontryagin square} of a $4k$-dimensional symmetric complex $(C, \phi)$ over $\zz$ is the function
$$
\begin{array}{rcl}
\mc{P}_2(\phi) : H^{2k}(C ; \zz_2) = H_0(\textnormal{Hom}_{\zz}(C, B(2k,s)))& \to &Q^{4k}(B(2k, 2)) = \zz_4\\
(u, v) & \mapsto & (u, v)^{\%}(\phi) = \phi_0(u, v) + 2\phi_1(v, u).
\end{array}
$$
\end{definition}


In definition \ref{algebraic-Pont},  $C$ is a $4k$-dimensional chain complex with a $4k$-dimensional symmetric structure, $B(2k, s)$ is a finitely generated free $\bb{Z}$-module concentrated in degrees $2k+1$ and $2k$ and let $g$ be a chain map $g = (u, v) : C \longrightarrow B(2k, d)$.
Note that there is defined a $\bb{Z}$-module chain map as follows,
$$
\xymatrix{
C: \ar[d]_{g = (u, v)} &\dots \ar[r]^{d} &C_{2k+1} \ar[d]_{u} \ar[r]^{d} & C_{2k} \ar[d]_{v} \ar[r]^{d} &C_{2k-1}\ar[d] \ar[r]^{d} & \dots \\
B(2k, s): &\dots  \ar[r] & \bb{Z} \ar[r]^{s=2} & \hspace{5pt} \bb{Z}  \ar[r] & 0 \ar[r] & \dots \\
}
$$

Here $u \in \textnormal{Hom}_{\bb{Z}}(C_{2k+1}, \bb{Z})$, in fact $u$ is a cocycle such that
$$u \in \textnormal{Ker}(d^*: C^{2k+1} \rightarrow C^{2k+2}).$$
Similarly $v \in \textnormal{Hom}_{\bb{Z}}(C_{2k},\bb{Z})$, $v$ is a cochain such that
$$ d^*(v) \in \textnormal{Im}(2: C^{2k}\to C^{2k+1}).$$
Moreover $d^*v = 2u$.

Since $v$ is a morphism $v : C_{2k} \to \bb{Z}$, we can compose $v$ with a map $\bb{Z} \to \bb{Z}_2$,
$$[v]: C_{2k} \to \bb{Z} \to \bb{Z}_2,$$
which represents a cocycle
$$[v] \in \textnormal{Ker}(d^*: \textnormal{Hom}_{\bb{Z}}(C_{2k}, \bb{Z}_2) \to  \textnormal{Hom}_{\bb{Z}}(C_{2k+1}, \bb{Z}_2)). $$

There is an isomorphism
$$H_0(\textnormal{Hom}_{\bb{Z}}(C, B(2k, s))) \to H^{2k}(C, \bb{Z}_2) ; (u, v) \to [v].$$

Hence,
$$
\begin{array}{ccc}
\mc{P}_2(\phi): H^{2k}(C, \bb{Z}_2)=H_0(\textnormal{Hom}_{\bb{Z}}(C, B(2k, s)))  & \to & \bb{Z}_4 \\
(u, v) &\mapsto &\phi_0(v, v) + 2\phi_1(v, u).
\end{array}
$$

The computation of $Q^{4k}(B(2k, s))$ is given in \cite{Mod8} as
$$Q^{4k}(B(2k, s)) \to \bb{Z}_4 ; \phi \to \phi_0 + d\phi_1.$$



The $\bb{Z}$-module chain map $(u, v) : C \longrightarrow B(2k,s)$ induces a map in the $Q$-groups,
$$(u, v)^{\%}: Q^{4k}(C) \longrightarrow Q^{4k}(B(2k,s)) =  Q^{4k}\left(\begin{tabular}{c} $S^{2k+1}\bb{Z}$ \\ $\downarrow^2$ \\ $S^{2k}\bb{Z}$ \end{tabular} \right),$$
which sends the symmetric structure of $(C, \phi)$ to the Pontryagin square evaluated on $(u, v).$ That is, with $(u, v) : C \longrightarrow B(2k,s)$ we have,

\vspace{8pt}

\begin{tabular}{cccccccccc}
\hspace{-30pt}{\small $ Q^{4k}(C)$} & {\small $\xrightarrow{(u, v)^{\%}}$ } & {\small $Q^{4k}(B)$} & {\small $\xrightarrow{\cong}$} &  {\small  $\bb{Z}_4$} & $\xleftarrow{\mc{P}_2(\phi)}$ & {\small $H^{2k}(C ; \bb{Z}_2)$} & {\small $\overset{\cong}\longleftarrow$} & {\small$H_0(\textnormal{Hom}_{\bb{Z}}(C, B))$ }\\
 $\phi$ & $\longmapsto$ & $(u, v)^{\%}(\phi)$& $=$ &  {\small $\phi_0(v,v)+2 \phi_1(u,v)$} & $\longleftarrow$ &$[v]$ & $\longleftarrow$ & $(v, u)$
\end{tabular}

\begin{remark}

If $C=C(X)$ is the chain complex of a space $X$ and  $\phi=\phi_X[X] \in Q^{4k}(C)$ is the image of a homology class $[X] \in H_{4k}(X)$ under the symmetric construction $\phi_X$ (\ref{symmetric construction}) then the evaluation of the Pontryagin square $\mathcal{P}_2:H^{2k}(X;\zz_2) \to H^{4k}(X;\zz_4)$ on the mod $4$ reduction $[X]_4 \in H_{4k}(X;\zz_4)$ is the algebraic Pontryagin square
$$\mathcal{P}_2(\phi) : H^{2k}(C; \zz_2)=H^{2k}(X; \zz_2) \to \zz_4 ; x \mapsto \langle \mathcal{P}_2(x), [X]_4 \rangle.$$
There is a commutative diagram ,
$$
\xymatrix{ H^{2k}(C(X); \zz_2) \ar[r]^{\mc{P}_2(\phi)} &  Q^{4k}(B(2k,2))= \zz_4 \\
H^{2k}(X ; \zz_2) \ar[u]^{=} \ar[r]^{\hspace{-55pt} \mc{P}_2} & H^{4k}(X; \zz_4)= H^{4k}(X ; Q^{4k}(B(2k,2))) \ar[u]_{\langle [X]_4, - \rangle}
}
$$
We are mainly interested in the case when $X$ is a $4k$-dimensional geometric Poincar\'e space and $[X]$ is the fundamental class.
\end{remark}

\section{Equivariant Pontryagin squares}
\subsection{The image of the equivariant Pontryagin square}

In the definition of the equivariant Pontryagin square we will need an explicit description for the $2k$-dimensional $Q$-group of a chain complex $B(k, d)$ concentrated in degrees $k$ and $k+1$.  We then use this computation to prove the isomorphism  $$Q^{2k}(B(k, d))= Q^{2k}  \left(\begin{tabular}{c} $B_{k+1}=I$ \\ $\downarrow^d$ \\ $B_{k}=R$ \end{tabular} \right) \overset{\cong}{\longrightarrow} \bb{Z}[\pi]/(I^2+2I),$$ which is described in Proposition \ref{iso-ring-Q}. The image of the equivariant Pontryagin square lies in this group. The equivariant Pontryagin square was defined on an even cohomology class in \cite{Korzen}. Here we will extend this definition to an odd class.

The $n$-dimensional $\epsilon$-symmetric $Q$-group $Q^n(C, \epsilon)$ is the abelian group of equivalence classes of $n$-dimensional $\epsilon$-symmetric structures on $C$,
$$Q^n(C, \epsilon) = H_n(W^{\%}C)_n.$$

\begin{example} \label{Q-group-1-dimension}
The simplest example for a $Q$-group is  the computation of the $\epsilon$-symmetric $Q$-groups of a f.g. projective $R$-module chain complex concentrated in degree $k$. Proposition 16 in \cite{BanRan} gives the full computation of these $\epsilon$-symmetric $Q$-groups. Following this proposition,
\begin{align*}
Q^{2k}(B_k, \epsilon) & = H^0(\bb{Z}_2; S(B^k), (-1)^kT_{\epsilon}) \\
                       &=H^0(\bb{Z}_2; \textnormal{Hom}_R(B^k,B_k), (-1)^kT_{\epsilon})\\
                       &= \{ \gamma \in \textnormal{Hom}_R(B^k, B_k)\vert T_{\epsilon}\gamma = (-)^k \gamma \}. \\
\end{align*}
\raggedbottom
That is, $Q^{2k}(B_k, \epsilon)$ is the group of $(-)^k \epsilon$-symmetric forms on $B^k.$
\end{example}

We will now focus on the computation of an explicit formula for the $2k$-dimensional $Q$-group of a f.g projective $R$-module chain complex which is concentrated in degrees $k, k+1$:
$$B(k,d): \dots \to 0 \to B_{k+1} \overset{d}{\rightarrow} B_k \to \dots$$
This group is given as part of the long exact sequence presented in Proposition 19 of \cite{BanRan}. Example 20 in \cite{BanRan} gives the explicit computation for the case $k=0$.

We recall now the notation for the $\epsilon$-duality involution $T_{\epsilon}$ which we will use in the next proposition. Let $R$ be a ring with involution,
$$T_{\epsilon} : \textnormal{Hom}_R(B^q, B_q) \to \textnormal{Hom}_{R}(B^q, B_p) ; \gamma \mapsto (-1)^{pq} \epsilon \gamma^*.$$

\begin{proposition}\label{Q-2k} Let $R= \bb{Z}[\pi]$ be a ring with involution.
The map
$$
\begin{array}{l}
Q^{2k}(B(k,d), \epsilon) \overset{\cong}{\longrightarrow} \\
 \dfrac{ \{ \gamma \in \textnormal{Hom}_{\bb{Z}[\pi]} (B^k, B_{k}) \vert \gamma= (-1)^k \epsilon \gamma^*  \}}{\{d\alpha d^* + (1+(-1)^kT)d  ~ \textnormal{ for }~ \alpha \in \textnormal{Hom}_{\bb{Z}[\pi]} (B^{k+1}, B_{k+1}), \beta \in \textnormal{Hom}_{\bb{Z}[\pi]} (B^k, B_{k+1}) \beta \vert \alpha= (-1)^k \epsilon \alpha^*\} }; \\
\phi \mapsto \phi_0  + d\phi_1 .
\end{array}
$$
is an isomorphism.
\end{proposition}

\begin{proof}
We shall write $B$ for $B(k,d)$ where there is no confusion from the context.
The exact sequence given in Proposition 19 of \cite{BanRan} is as follows
\begin{equation}\label{exact-seq-Q}
\dots \to Q^{2k+1}(B, \epsilon) \to B_{k+1} \otimes_A H_{k+1}(B) \overset{F}{\rightarrow} Q^{2k}(d, \epsilon) \overset{t}{\rightarrow} Q^{2k}(B, \epsilon)\to 0.
\end{equation}
Since the sequence is exact, the $Q$-group $Q^{2k}(B, \epsilon)$ is the cokernel of the map $$B_{k+1} \otimes_A H_{k+1}(B) \overset{F}{\rightarrow} Q^{2k}(d, \epsilon).$$
So in the first place we shall compute the relative group $Q^{2k}(d, \epsilon).$
The chain map $d: B_{k+1} \to B_{k}$ induces a map in the $Q$-groups
$$d^{\%} : Q^{2k}(B_{k+1}, \epsilon) \to Q^{2k}(B_{k}, \epsilon) ; \alpha \mapsto d \alpha d^*.$$
The relative group  $Q^{2k}(d, \epsilon)$ fits into the long exact sequence of $Q$-groups,
$$Q^{2k+1}(B_k) \to Q^{2k+1}(d) \to Q^{2k}(B_{k+1}) \xrightarrow{d^{\%}}  Q^{2k}(B_{k}) \to Q^{2k}(d) \to Q^{2k-1}(B_{k+1})=0.$$

Regarding the f.g. free $R$-module chain complex
$$B(k,d): \dots \to 0 \to B_{k+1} \overset{d}{\rightarrow} B_k \to \dots$$  as a chain map of chain complexes both concentrated in degree $k$,
\begin{displaymath}
\xymatrix{
\dots \ar[r] & 0 \ar[r] &*+[F]{B_{k+1}} \ar[d]^d \ar[r] &0 \ar[r]& \dots \\
\dots \ar[r] & 0 \ar[r] &*+[F]{B_{k}} \ar[r] &0 \ar[r]& \dots
}
\end{displaymath}
we have that $B(k, d)$ is the algebraic mapping cone $\mc{C}(d)$ of the chain map $d: B_{k+1} \to B_k.$

We have seen in example \ref{Q-group-1-dimension} that the  $\epsilon$-symmetric $Q$-groups of a f.g. projective $R$-module chain complex concentrated in degree $k$ is given by,
\begin{align*}
Q^{2k}(B_k, \epsilon) & = \{ \gamma \in \textnormal{Hom}_A(B^k, B_k)\vert T_{\epsilon}\gamma = (-)^k \gamma \}.
\end{align*}
Similarly for $B_{k+1}$. Note that we have defined $d : B_{k+1} \to B_k$ to be a chain map of two chain complexes both concentrated in degree $k$, so that we are considering $B_{k+1}$ to be a chain complex concentrated in degree $k$, thus the computation following Proposition 16 of \cite{BanRan} is the same as the one given in example \ref{Q-group-1-dimension},
\begin{align*}
Q^{2k}(B_{k+1}, \epsilon) & = H^0(\bb{Z}_2; S(B^{k+1}), (-1)^kT_{\epsilon}) \\
                       &=H^0(\bb{Z}_2; \textnormal{Hom}_A(B^{k+1},B_{k+1}), (-1)^kT_{\epsilon})\\
                     &= \{ \alpha \in \textnormal{Hom}_A(B^{k+1}, B_{k+1})\vert T_{\epsilon}\alpha =(-)^k\alpha \}.
\end{align*}

Hence the relative group in the exact sequence is
\begin{align*}
Q^{2k}(d, \epsilon) &= \frac{\textnormal{$(-)^k\epsilon$-Sym}(B^{k})}{d^{\%} \textnormal{$(-)^k\epsilon$-Sym}(B^{k+1})}\\
                    &= \frac{ \{\gamma=(-)^k\gamma^* \in \textnormal{Hom}_A(B^k, B_k)\}}{\{ d\alpha d^* \vert \alpha=(-)^k\alpha^* \in \textnormal{Hom}_A(B^{k+1}, B_{k+1})\}}
\end{align*}
and $Q^{2k}(B(k,d),\epsilon)$ is the cokernel of the map
$$F ~ : ~ B_{k+1} \otimes_A H_{k+1}(B) \rightarrow \frac{\textnormal{$(-)^k\epsilon$-Sym}(B^{k})}{d^{\%} \textnormal{$(-)^k\epsilon$-Sym}(B^{k+1})}$$ in the long exact sequence \eqref{exact-seq-Q} on page \pageref{exact-seq-Q}, so that,
$$Q^{2k}(B(k,d),\epsilon)= \textnormal{coker}\left( F: B_{k+1} \otimes_A H_{k+1}(B) \rightarrow \frac{\textnormal{$(-)^k\epsilon$-Sym}(B^{k})}{d^{\%} \textnormal{$(-)^k\epsilon$-Sym}(B^{k+1})}\right).$$

We now refer to Proposition 18 in \cite{BanRan} for the definition of the map $F$.
First note that $B(k,d) = \mc{C}(d)$, as we have explained above. An element in $B_{k+1} \otimes H_{k+1}(B(k, d)) = B_{k+1} \otimes H_{k+1}(\mc{C}(d))$ is a pair $(\beta, h)$ where $\beta: B^{k} \to B_{k+1}$ is a chain map and $h$ is a chain homotopy $h: d \beta \simeq 0: B^k \to B_k$.

Hence,
$$
\begin{array}{l}
Q^{2k}(B(k,d), \epsilon) \overset{\cong}{\longrightarrow} \\
 \textnormal{coker}\left( (1+(-1)^kT)d : \textnormal{Hom}_A(B^k, B_{k+1}) \rightarrow \frac{\textnormal{$(-)^k\epsilon$-Sym}(B^{k})}{d^{\%} \textnormal{$(-)^k\epsilon$-Sym}(B^{k+1})}\right) \\
 = \dfrac{ \{ \gamma \in \textnormal{Hom}_{\bb{Z}[\pi]} (B^k, B_{k}) \vert \gamma= (-1)^k \epsilon \gamma^*  \}}{\{d\alpha d^* + (1+(-1)^kT)d  ~ \textnormal{ for }~ \alpha \in \textnormal{Hom}_{\bb{Z}[\pi]} (B^{k+1}, B_{k+1}), \beta \in \textnormal{Hom}_{\bb{Z}[\pi]} (B^k, B_{k+1}) \beta \vert \alpha= (-1)^k \epsilon \alpha^*\} } . \\
\end{array}
$$
\end{proof}
\raggedbottom

\begin{proposition}\label{iso-ring-Q}
Let $R$ be the ring $R= \bb{Z}[\pi]$ and let $I$ be the ideal given by $I= \textnormal{Ker}(\bb{Z}[\pi] \to \bb{Z}_2)$ and $B(k, d)$ be the f.g free $\bb{Z}[\pi]$-module chain complex concentrated in degrees $k$ and $k+1$, $B(k, d): \dots \to 0 \to B_{k+1}=I \overset{d}{\rightarrow} B_k=R \to 0 \to \dots$, then
$$Q^{2k}(B(k, d))= Q^{2k}  \left(\begin{tabular}{c} $B_{k+1}=I$ \\ $\downarrow^d$ \\ $B_{k}=R$ \end{tabular} \right) \overset{\cong}{\longrightarrow} \bb{Z}[\pi]/(I^2+2I).$$
\end{proposition}

\begin{proof}

The ring $R=\bb{Z}[\pi]$ is a ring with involution. This involution extends to the involution on $R/I$. The ideal $I$ is a two sided ideal which is invariant under the involution, so that $I^*=I$.
Note that by the definition of the ideal $I=\textnormal{Ker}(\bb{Z}[\pi] \to \bb{Z}_2)$ we can form an exact sequence of $\bb{Z}[\pi]$ modules,
$$0 \to  I=\textnormal{Ker}(\bb{Z}[\pi] \to \bb{Z}_2) \xrightarrow{d} R=\bb{Z}[\pi] \to \bb{Z}_2 \to 0,$$
so that $R/I= \bb{Z}_2$.
Also note that the differential in the chain complex $$B(k, d): \dots \to 0 \to B_{k+1}=I \overset{d}{\rightarrow} B_k=R \to 0 \to \dots$$ is now given by the inclusion of the ideal $I$ into the ring $R$.
Using the isomorphism described in Proposition \ref{Q-2k}, we know that
$$
\begin{array}{l}
Q^{2k}  \left(\begin{tabular}{c} $B_{k+1}=I$ \\ $\downarrow^d$ \\ $B_{k}=R$ \end{tabular} \right) \overset{\cong}{\longrightarrow} \\
\dfrac{\gamma= (-)^k\gamma^* \in \textnormal{Hom}_{\bb{Z}[\pi]} (\bb{Z}[\pi], \bb{Z}[\pi])}{\{d\alpha d^* + (1+(-1)^kT)d \beta \vert \alpha= (-)^k\alpha^* \in \textnormal{Hom}_{\bb{Z}[\pi]} (I^*, I), \beta \in \textnormal{Hom}_{\bb{Z}[\pi]} (\bb{Z}[\pi], I)\}} \\
\phi \mapsto \phi_0  + d\phi_1 .
\end{array}
$$

In this expression we have
$$\gamma= (-)^k\gamma^* \in \textnormal{Hom}_{\bb{Z}[\pi]} (\bb{Z}[\pi], \bb{Z}[\pi]) = \bb{Z}[\pi] \otimes_{\bb{Z}[\pi]} \bb{Z}[\pi] = \bb{Z}[\pi]=R.$$
Also,
$$\alpha = (-)^k\alpha^* \in \textnormal{Hom}_{\bb{Z}[\pi]} (I^*, I) = I \otimes_{\bb{Z}[\pi]} I =(I^2).$$

Since $d$ is the inclusion of the ideal $I$ into the ring $R$, then an element in the image of $d$ is an element in $I$,
so $d \beta \in I$ and $(1+(-1)^kT)d \beta \in (2I)$.
Hence,
$$Q^{2k}  \left(\begin{tabular}{c} $B_{k+1}=I$ \\ $\downarrow^d$ \\ $B_{k}=R$ \end{tabular} \right) \overset{\cong}{\longrightarrow} R / (I^2+2I). $$
\end{proof}

\raggedbottom


\subsection{Equivariant Pontryagin squares: definition}

The notion of equivariant Pontryagin square was introduced in \cite{Korzen} for even cohomology classes. Here we extend this notion to prove that the equivariant Pontryagin squares can also be defined on odd cohomology classes.

Let $R=  \bb{Z}[\pi]$ be a ring with involution and $I$ the $\zz_2$ augmentation ideal $I=\textnormal{Ker}(\bb{Z}[\pi] \to \bb{Z}_2)$ as in the previous section. We now define a $R$-module chain map with C a $2k$-dimensional chain complex $\bb{Z}[\pi] = R$-module, and $B(k, d_R) : \dots 0 \to B_{k+1} = I \to B_k = R \to 0 \to \dots$
$$
\xymatrix{
C: \ar[d]_{g = (u, v)} &\dots \ar[r]^{d_C} &C_{k+1} \ar[d]_{u} \ar[r]^{d_C} & C_{k} \ar[d]_{v} \ar[r]^{d_C} &C_{k-1}\ar[d] \ar[r]^{d_C} & \dots \\
B(k, d_R): &\dots  \ar[r] & I \ar[r]^{d_R=i} & \hspace{5pt} R  \ar[r] & 0 \ar[r] & \dots \\
}
$$
This chain complex map is defined by a cocycle $u \in \textnormal{Ker}(d^*_C: C^{k+1} \rightarrow C^{k+2})$ and  a cochain $ d^*_C(v) \in \textnormal{Im}(d^*_C: C^{k}\to C^{k+1})$ such that $d_C^*v = d_Ru$.

An element $[x] \in H^k(C; \bb{Z}_2)$ is given by the composite morphism
$$[x] : C_k \xrightarrow{v} \bb{Z} \to \bb{Z}_2 = R/I$$
so it is a cocycle of $\textnormal{Hom}_{\bb{Z}[\pi]}(C_k, \bb{Z}_2),$
$$[x] \in \textnormal{Ker}(d^* : \textnormal{Hom}_{\bb{Z}[\pi]}(C_k, \bb{Z}_2) \to \textnormal{Hom}_{\bb{Z}[\pi]}(C_{k+1}, \bb{Z}_2)). $$
Then there is an isomorphism,
\begin{align*}
H^k(C; \bb{Z}_2) &\to H_0(\textnormal{Hom}_{\bb{Z[\pi]}}(C, B(k, d_R))) \\
[x] & \to (u, v).
\end{align*}

\begin{definition} \label{equivariant-Pont} Let $R$ be the ring $R = \bb{Z}[\pi]$ and $I$ be the ideal $I = \textnormal{Ker}(\bb{Z}[\pi] \to \bb{Z}_2)$.
The $R/I$-coefficient Pontryagin square of a $2k$-dimensional symmetric complex $(C, \phi)$ over $R=\bb{Z}[\pi]$ is the function
\begin{align*}
\mc{P}_{d_R} : H^{k}(C; \bb{Z}_2)\to Q^{2k}(B(k, d_R))= R/(I^2+2I) ;\\
(u, v) \mapsto (u, v)^\%(\phi) = \phi_0(v, v) + d_R \phi_1(v, u)
\end{align*}
with $(v, u) \in C^k \oplus C^{k+1}$ such that $d_C^*(v) = d_Ru$, $d_C^*(u)=0 \in C^{k+2}$.
\end{definition}

\begin{remark}
Note that in definition \ref{equivariant-Pont} $k$ can be odd.
\end{remark}

\chapter{The Signature, the Pontryagin square, the Brown-Kervaire and Arf invariants in topology} \label{Morita-chain-cx}

\begin{flushright}

\textit{``Any good theorem should have several proofs, the more the better. For two reasons: usually, different proofs have different strengths and weaknesses, and they generalise in different directions, they are not just repetitions of each other."}
\end{flushright}

\begin{flushright}
\textit{Sir Michael Atiyah}
\end{flushright}

\vspace{10pt}

In this chapter we will give chain complex proofs of  the main results in \cite{Morita}. We shall prove:
\begin{itemize}
\item Morita's Theorem (\cite[theorem~1.1]{Morita}) which gives the mod $8$ relation between the signature of a $4k$-dimensional Poincar\'e space $X$ and the Brown-Kervaire invariant,
$$\sigma (X)  = \textnormal{BK}(H^{2k}(X; \mathbb{Z}_2), \lambda, \mc{P}_2) \in \zz_8.$$
(See theorem \ref{Morita}).
\item the mod $4$ relation between the signature and Pontryagin square,
$$\sigma(X) = \langle \mc{P}_2(v_{2k}), [X]\rangle \in \zz_4. $$
 This will be proved in the first section in this chapter. This proof is just the mod $4$ reduction of the expression for the signature given in \cite[Proposition 52, Corollary 61]{BanRan}, which we discussed in section \ref{Twisted-Q-groups}.
\end{itemize}
The contents of section \ref{obstructions} in this chapter were not considered in \cite{Morita}. In that section we shall describe certain invariants which detect when the signature is divisible either by $4$ or by $8$, and in particular in subsection \ref{BK-Arf-in-topology} we state the relation between the Brown-Kervaire and Arf invariants in topology.

\section{The signature modulo $4$ and the Pontryagin square} \label{first-section}

Morita \cite{Morita} proved that there exists a mod $4$ relation between the signature of a $4k$-dimensional Poincar\'e space and the Pontryagin square.  Here we give a new proof of this result by proving its algebraic analogue. In \cite{Morita} this result is stated geometrically as follows,

\begin{proposition} (\hspace{-1pt}\cite[Proposition 2.3]{Morita})
Let $X^{4k}$ be an oriented Poincar\'e complex and $\mc{P}_2: H^{2k}(X ; \bb{Z}_2) \to \bb{Z}_4$ be the Pontryagin square, then
$$\sigma(X) = \langle \mc{P}_2 (v_{2k}), [X] \rangle \in \zz_4, $$
where $v_{2k} \in H^{2k}(X ; \bb{Z}_2)$ is the $2k$-th Wu class of $X.$
\end{proposition}

The algebraic analogue of this result is stated as follows, using the algebraic  Wu class (defined in \ref{algebraic-Wu}):
\begin{proposition} \label{signature-Psq}
Let $(C, \phi)$ be a $4k$-dimensional symmetric Poincar\'e complex over $\mathbb{Z}$, then
$$\sigma(C, \phi) = \mc{P}_2(v_{2k})\in \zz_4,$$
where  $v_{2k} \in H^{2k}(C ; \bb{Z}_2)$ is the $2k$-th algebraic Wu class of  $(C, \phi).$
\end{proposition}

\begin{proof}
The proof of this proposition is a direct consequence of Theorem \ref{sign-theta}.
The chain complexes in this proposition are Poincar\'e so the hyperquadratic signature (also referred to in the literature as mod $8$ signature)  coincides with the signature reduced modulo $8$ and we can simply write the relation in Theorem \ref{sign-theta}
$$\widehat{\sigma}(C, \phi, \gamma, \theta) = \phi_0(v, v) + 2 \phi_1(v, u) + 4 \theta_{-2}(u, u) \in Q_{4k}(B(k, 1), \beta(k, 1)) =  \bb{Z}_8 $$
as
$$\sigma(C, \phi) = \phi_0(v, v) + 2 \phi_1(v, u) + 4 \theta_{-2}(u, u) \in \bb{Z}_8. $$
Reducing modulo $4$ we obtain,
$$
\xymatrix{
\sigma (C, \phi)=  \phi_0(v, v) + 2 \phi_1(v, u) + 4 \theta_{-2}(u, u) \ar@{|->}[d] & \in & Q_{4k}(B(k, 1), \beta(k, 1)) = \bb{Z}_8\ar[d]^{\textnormal{mod} 4} \\
\mc{P}_2(\nu_{2k}) = \phi_0(v, v) + 2 \phi_1(v, u) & \in & Q_{4k}(B(k, 1)) = \bb{Z}_4.}
$$

Hence the result follows.

\end{proof}

\subsection{A condition for divisibility by of the signature by $4$}

In \cite{Morita} the following corollary is proved.
\begin{corollary}\label{sign-pont-zero} (\hspace{-1pt}\cite[corollary 1.2]{Morita}) Let $X$ be a $4k$-dimensional oriented Poincar\'e space, then $\sigma(X) = 0 \in \zz_4$ if and only if
$$\mc{P}_2(v_{2k}) = 0 \in \bb{Z}_4.$$
Here $v_{2k}$ is the $2k$-th Wu class of $M$.
\end{corollary}

As before we give the algebraic analogue of this corollary.

\begin{corollary} \label{Pontryagin-zero}
Let $(C, \phi)$ be a $4k$-dimensional symmetric Poincar\'e complex, then $\sigma(C, \phi) = 0  \in \zz_4$ if and only if,
$$\mc{P}_2(v_{2k}) = 0 \in \bb{Z}_4$$
where $v_{2k}$ is the  $2k$-th algebraic Wu class of the symmetric chain complex $(C, \phi)$.
\end{corollary}

\begin{proof}
The proof of this corollary is now a direct consequence of Proposition \ref{signature-Psq}.
\end{proof}
This corollary is specially interesting in the case of a fibration:
From \cite[Theorem A]{modfour} we know that for a fibration $F \to E \to B$, the signature is multiplicative modulo $4$,
$$\sigma(E) - \sigma(F) \sigma(B) =0 \in \zz_4.$$
In the case of a fibration with dimensions of the base and fibre congruent to $2$ modulo $4$,
\begin{equation}\label{sqsym}
F^{4i+2} \to E \to B^{4j+2},\end{equation} we know that for dimension reasons the signatures of base and fibre is $0$, so that the signature of the total space is always divisible by $4$,
$$\sigma(E) = 0 \in \zz_4.$$
Thus by Corollary \ref{sign-pont-zero}, a property of bundles of the form $F^{4i+2} \to E \to B^{4j+2}$ is that
$$\langle \mc{P}_2(v(\nu_E)), [E] \rangle=0 \in \zz_4.$$

When the fibration has dimensions $F^{4j} \to E^{4k} \to B^{4i}$ we have that $\sigma(E) - \sigma(F) \sigma(B) =0 \in \zz_4,$ and hence
$$\langle \mc{P}_2(v), [E \sqcup -F \times B] \rangle =0 \in \zz_4,$$
where $v$ is the Wu class $v \in H^{2k}(E;\zz_2) \oplus H^{2k}(F \times B).$

\section{The signature modulo $8$ and the Brown-Kervaire invariant}

\subsection{The Brown-Kervaire invariant and quadratic linking forms}\label{BK-and-L}

In section \ref{Morita-theorem} we shall give a proof of  \cite[theorem 1.1]{Morita} using chain complex methods. This theorem gives a relation between the signature mod $8$ and the Brown-Kervaire invariant of a linking form over $(\zz, (2)^{\infty})$.

\begin{definition}\label{Brown-linking}
The Brown-Kervaire invariant $\textnormal{BK}(T, b, q) \in \zz_8$ of a nonsingular quadratic linking form $(T, b, q)$ is defined using a Gauss sum
$$\sum_{x \in T} i^{q(x)} = \sqrt{\vert T \vert} e^{2 \pi i \textnormal{BK}(T, b, q)/8}.$$
\end{definition}
 \begin{remark} Note that this essentially is the same expression as that of Definition  \ref{Gauss-sum-formula}, which gives the Brown-Kervaire invariant of \cite{Brown}. In \cite{Brown} $T$ is required to be a $\zz_2$-vector space.  The more general application of this Gauss sum is given  \cite[page 109]{Brum-Morg}.
\end{remark}

Here we will discuss the relation between the signature modulo $8$, the hyperquadratic signature and the Brown-Kervaire invariant. Important references here are \cite{exactseqRan},  \cite[section 4]{BanRan} and \cite{Mod8}.

We start by giving some general definitions about linking forms over $(R, S)$ where $R$ is a ring with involution and  $S\subset R$ is a multiplicative subset of central nonzero divisors.

The following are standard definitions from \cite{exactseqRan}.
\begin{definition}
\begin{itemize}
\item[(i)] \textbf{An $(R, S)$-module} is an $R$-module $T$ with a one-dimensional f.g. projective $R$-module resolution
$$0 \to P \xrightarrow{d} Q \to T \to 0$$
with $S^{-1}d : S^{-1}P \to S^{-1}Q$ is an $S^{-1}R$-module isomorphism, and $$S^{-1}T = 0. $$

\item[(ii)] The \textbf{dual} of an $(R, S)$-module $T$ is given by
\begin{align*}
T\widehat{{\color{White}a}} & = \textnormal{Ext}^1_R(T, R) \\
                                           &= \textnormal{Hom}_R(T, S^{-1}R/R) \\
                                          &=  \textnormal{coker}(d^*: Q^* \to P^*).
\end{align*}

\end{itemize}
\end{definition}
\begin{definition}
\begin{itemize}
\item[(i)]
A \textbf{ $\epsilon$-symmetric linking form $(T, b)$ over $(R, S)$ } is an $(R, S)$-module $T$ together with $b \in \textnormal{Hom}_R(T, \T)$, such that $b \widehat{{\color{White}a}} = \epsilon b$, so that
$$\overline{b(x, y)} = \epsilon b(y, x) \in S^{-1}R/R  ~~~ (x, y \in T)$$

\item[(ii)] A \textbf{lagrangian for $(T, b)$} is an $(R, S)$ module $U \subset T$ such that
$$U^{\perp} = \left\{x \in T \vert b(x)(U) = 0 \in S^{-1}R/R \right\}=U, $$
and the sequence
$$0 \to U \to T \to \U  \to 0$$
is exact, where $\U$ is the dual of $U$.
\end{itemize}
\end{definition}

\begin{definition}
\begin{itemize}
\item[(i)] An \textbf{  $\epsilon$-quadratic linking form $(T, b, q)$ over $(R, S)$}  is an $\epsilon$-symmetric linking form $(T, b)$ together with a function
$$q: T \to Q_{\epsilon}(R, S) = \frac{\{ c \in S^{-1}R \vert \epsilon \overline{c}= c\}}{\{a+ \epsilon \overline{a} \vert a \in R \}}$$
such that
$$
\begin{array}{l}
q(ax) = a q(x) \overline{a} \\
q(x+y) = q(x)+ q(y) + b(x, y) + b(y, x) \\
q(x) = b(x,x) \in S^{-1}R/R.
\end{array}
$$

\item[(ii)] A \textbf{lagrangian} $U$ for the nonsingular quadratic linking form $(T, b, q)$ is a lagrangian of the nonsingular form $(T, b)$, such that $q$ restricts to $0$ on $U$.
\end{itemize}
 \end{definition}

 $L_0(R, S)$ is the cobordism group of quadratic $(R, S)$-module linking forms, and $L^0(R,S)$ is the cobordism group of symmetric $(R, S)$-module linking forms. These groups fit into the localization exact sequences

\begin{equation}\label{exact2}
\dots \to L^n(R) \to L^n(S^{-1}R) \to L^n(R, S) \to L^{n-1}(R) \to \dots
\end{equation}
 \begin{equation}\label{exact3}\dots \to L_n(R) \to L_n(S^{-1}R) \to L_n(R, S) \to L_{n-1}(R) \to \dots
\end{equation}
where $S^{-1}R$ is the localized ring with involution obtained from $R$ by inverting $S$. These exact sequences are defined in detail in \cite{exactseqRan}.

We shall only be concerned with the case $1/2 \in S^{-1}R$
\begin{example} $R=\zz$ and $S=(2)^{\infty}$, $1/2 \in S^{-1}R$.
\end{example}
With $1/2 \in S^{-1}R$ there is an isomorphism $ L_*(S^{-1}R) \cong L^*(S^{-1}R)$, so there is a braid of exact sequences,

\noindent\hspace{2cm}\textcolor{white}{W}\\[0ex]
\[ \xymatrix@C-11pt@R-15pt{
 &&L_0(R) \ar@/^1.7pc/@[Blue]@[thicker][rr] \ar@[Red]@[thicker][dr] &&
L_0(S^{-1}R)\cong L^0(S^{-1}R) \ar@[Blue]@[thicker][dr] \ar@[Green]@/^1.7pc/@[thicker][rr]&& L^0(R, S)
&& \\
&
&&L^0(R) \ar@[Green]@[thicker][ur] \ar@[Red]@[thicker][dr] &&
L_0(R, S)  \ar@[Blue]@[thicker][dr]   \ar@[Black]@[thicker][ur]  && \\
&&
L^1(R, S) \ar@/_1.7pc/@[thicker][rr] \ar@[Green]@[thicker][ur] &&
\widehat{L}^0(R) \ar@[thicker][ur]     \ar@[Red]@/_1.7pc/@[thicker][rr]&& L_{-1}(R)
  &&} \]

\medskip

\noindent\hspace{2cm}\textcolor{white}{W}

\begin{example} Let $R = \bb{Z}$, $S= (2)^{\infty}= \left\{2^i \vert i \geq 0 \right\} \subset \bb{Z}$ and $S^{-1}R = \bb{Z}[\frac{1}{2}]$. Here is the braid in this case.

\vspace{-60pt}
\begin{center}
\noindent%
\begin{minipage}[c]{0.05\linewidth}
{\color{White} a}
\end{minipage}
\begin{minipage}[c]{0.8\linewidth}
\begin{figure}[H]
\rotatebox{90}{
\begin{minipage}{\textheight}
\labellist
\small\hair 2pt
\pinlabel ${L_{4k}(\bb{Z})  \xrightarrow{\cong} \bb{Z}}$ at 80 205
\pinlabel {\tiny${(C, \phi) \mapsto \sigma(F^{2k}(C), \phi_0)/8}$} at 79 190
\pinlabel ${L^{4k}(\bb{Z})  \xrightarrow{\cong} \bb{Z}}$ at 180 125
\pinlabel {\tiny${(C, \phi) \mapsto \sigma(F^{2k}(C), \phi_0)}$} at 185 110
\pinlabel {\tiny{\tiny${\widehat{L}^{4k}(\bb{Z}) \xrightarrow{\cong} Q_{4k}(B, \beta)=\bb{Z}_8}$}} at 286 49
\pinlabel {\tiny${(C, \phi, \gamma, \theta) \mapsto (g, \chi)_{\%}(\phi, \theta)}$} at 286 35
\pinlabel {\tiny ${L_{4k}(\bb{Z}, (2)^{\infty})  \xrightarrow{\cong} \bb{Z}_2 \oplus \bb{Z}_8}$} at 398 125
\pinlabel \tiny {${(T, b, q) \mapsto (|T|_2, \textnormal{BK}(T, b, q))}$} at 401 110
\pinlabel \tiny {${\bb{Z}\left[\frac{1}{2}\right]}$} at 260 220
\pinlabel \tiny{${L^{4k}\left(\bb{Z}\left[\frac{1}{2}\right]\right)=L_{4k}\left(\bb{Z}\left[\frac{1}{2}\right]\right)}$} at 285 190
\pinlabel {\tiny$\zz$} at 521 220
\pinlabel \tiny {${L^{4k}(\bb{Z}, (2)^{\infty})  \xrightarrow{\cong} \bb{Z}_2}$} at 505 209
\pinlabel \tiny {${(T, b) \mapsto |T|_2}$} at 505 200
\pinlabel \tiny {${|T|_2=}$} at 475 192
\endlabellist
\centering
\includegraphics[width=\linewidth]{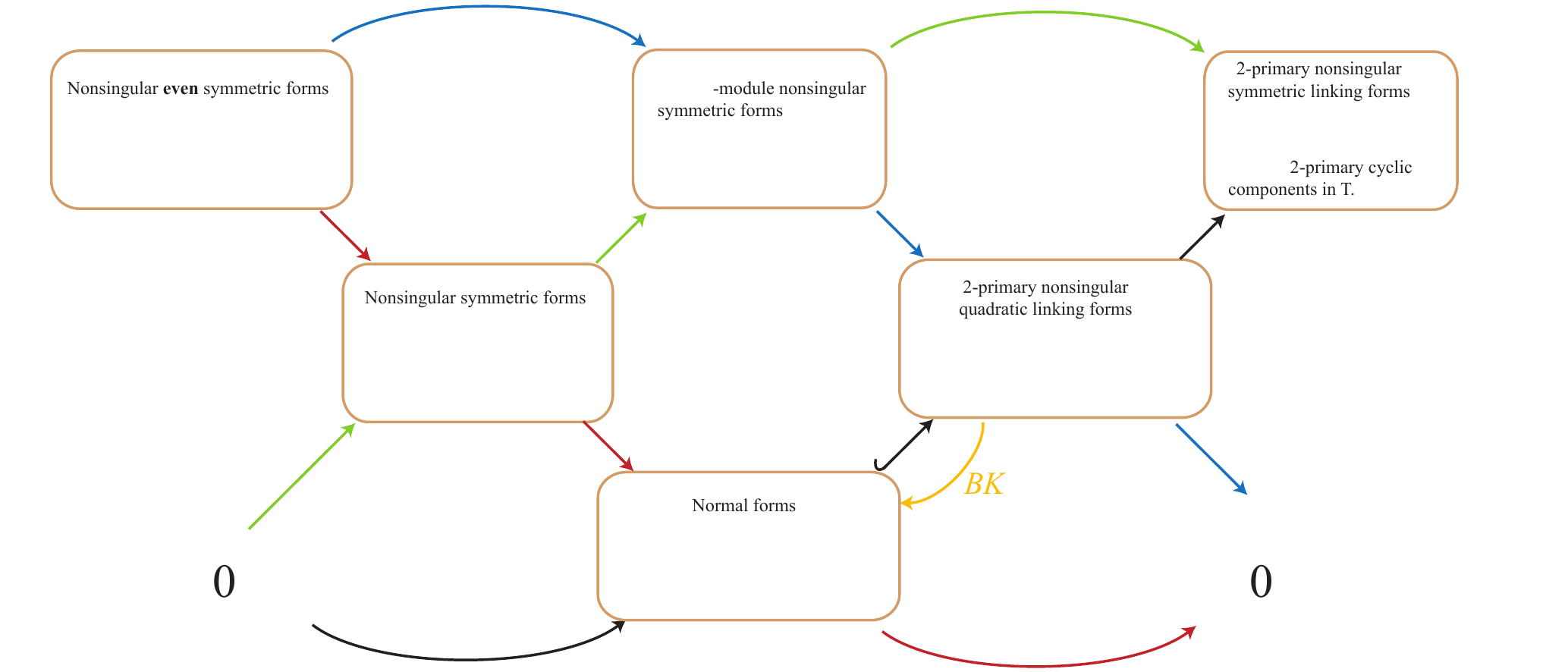}
\caption{close up of the localization braid}
\label{fig:braid}
\end{minipage}}
\end{figure}
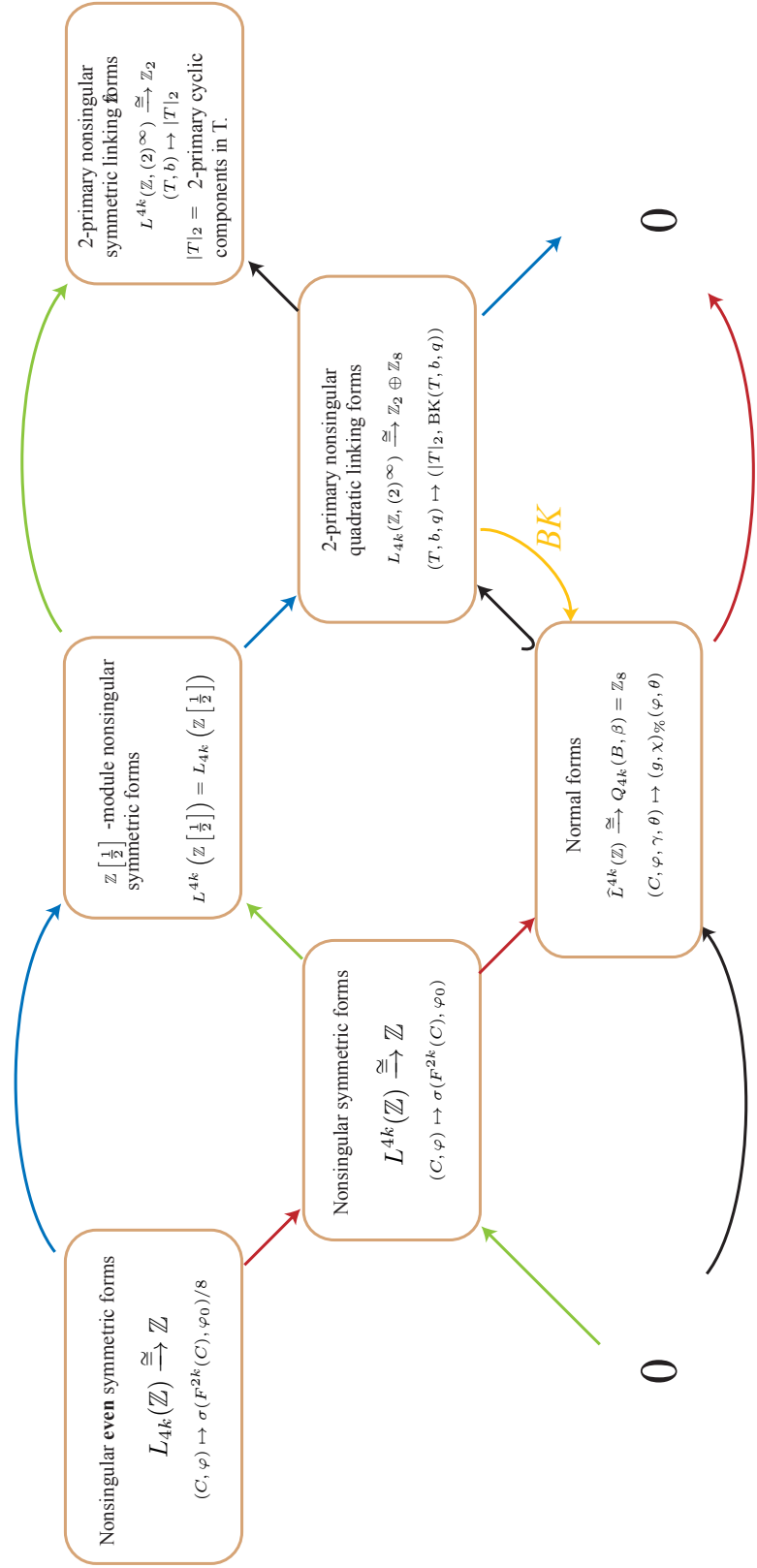
\end{minipage}
\end{center}

\end{example}

The group $L_{4k}(\zz, (2^{\infty}))$ is the cobordism group of $(4k-1)$-dimensional $\zz[\frac{1}{2}]$-acyclic quadratic Poincar\'e complexes over $\zz$.
A quadratic $(4k-1)$-dimensional Poincar\'e complex determines a nonsingular quadratic linking form over $(\zz, (2)^{\infty})$, hence the group $L_{4k}(\zz, (2^{\infty}))$ is also the Witt group of nonsingular quadratic linking forms over $(\zz, (2^{\infty}))$.

Similarly the $L^{4k}(\zz, (2)^{\infty})$ is the cobordism group of $(4k-1)$-dimensional $\zz[\frac{1}{2}]$-acyclic symmetric Poincar\'e complexes over $\zz$. These have a one-one correspondence with symmetric linking forms over $(\zz, (2)^{\infty})$. So that  the group $L^{4k}(\zz, (2^{\infty}))$ is also the Witt group of nonsingular symmetric linking forms over $(\zz, (2^{\infty}))$.

An element $(T,b,q) \in \textnormal{ker}(L_0(\zz,(2)^{\infty}) \to L^0(\zz,(2)^{\infty}))$ is represented by a quadratic linking form $(T,b,q)$ such that $(T,b) \oplus (T',b') \cong (T'',b'')$ with $(T',b')$, $(T'',b'')$ symmetric linking forms which admit lagrangians $U' \subset T'$, $U'' \subset T''$. The Brown-Kervaire invariant determines the Witt class of $(T,b,q).$

If $(T,b,q)$ has a lagrangian $U$ (i.e. a lagrangian for $(T,b)$ such that $q$ is $0$ on $U$) then $(T,b,q)=0 \in L_{4k}(Z,(2)^{\infty})$ and $BK(T,b,q)=0 \in Z_8$.

As indicated in the braid above, the Brown-Kervaire invariant splits the sequence
$$
\xymatrix{\dots  \ar[r]& L^1(\zz, (2^{\infty}))=0 \ar[r] & \widehat{L}^0(\zz) \ar[r] & L_0(\zz, (2^{\infty})) \ar@/^1.1pc/[l]^{BK} \ar[r] & L^0(\zz, (2^{\infty})) \ar[r]& \dots
}
$$
so the Brown-Kervaire invariant of a nonsingular linking form determines the hyperquadratic signature of a normal complex. (See \cite{BanRan}).

A $4k$-dimensional normal complex can be expressed as a (symmetric, quadratic) pair which has as boundary a $(4k-1)$-dimensional quadratic Poincar\'e complex.

A $(4k-1)$-dimensional quadratic Poincar\'e complex $(C, \psi)$ over $\zz$ is always null-cobordant by a quadratic Poincar\'e pair $(C \to \delta C, (\delta \psi, \psi))$ since $L_{4k-1}(\zz) =0$.
From the $4k$-dimensional quadratic Poincar\'e pair $(C \to \delta C, (\delta \psi, \psi))$ we can construct a $4k$-dimensional {\bf quadratic} complex,
$$(\delta C/ C,  \delta \psi / \psi),$$
by the algebraic Thom construction (see \cite[proposition 1.15]{bluebook})

The symmetrization map $Q_{4k}(\delta C/C) \xrightarrow{1+T} Q^{4k}(\delta C/C)$  allows us to construct a $4k$-dimensional symmetric complex $(\delta C/ C,  (1+T)\delta \psi / \psi)$.

\begin{definition}\label{hyper-sign} (\hspace{-1pt}\cite{Mod8})
The hyperquadratic signature of a $(4k-1)$-dimensional quadratic Poincar\'e complex $(C, \psi)$ over $\zz$ is defined by the mod $8$ reduction of the  signature of the $4k$-dimensional symmetric complex $(\delta C/ C, (1+T) \delta \psi / \psi),$
$$\widehat{\sigma}^*(C, \psi) = \sigma^*(\delta C/ C, (1+T) \delta \psi / \psi) \in \zz_8. $$
\end{definition}
\begin{remark}
Note that the hyperquadratic signature of a $(4k-1)$-dimensional quadratic Poincar\'e complex $(C, \psi)$ over $\zz$ is a well defined homotopy invariant but not a cobordism invariant. This is the case because two different nullcobordisms $\delta C$ and $\delta C'$ need not have the same signature. That is, if we consider the $4k$-dimensional quadratic Poincar\'e complex over $\zz$ given by $(\delta C \cup_C \delta C')$, then by Novikov additivity,
$$ \sigma(\delta C \cup_C \delta C')= \sigma(\delta C) - \sigma(\delta C'). $$
Since $(\delta C \cup_C \delta C')$ is quadratic then
$$\sigma(\delta C \cup_C \delta C')= \sigma(\delta C) - \sigma(\delta C') =0 \in \zz_8,$$
so the signatures are only congruent mod $8.$
\end{remark}

\begin{definition} \label{sign-pair} (\hspace{-1pt}\cite{Mod8})
The hyperquadratic signature of a (symmetric, quadratic) Poincar\'e pair $(C \to D, (\delta \phi, \psi))$ over $\zz$ is
\begin{align*}
\widehat{\sigma}^*(C \to D, (\delta \phi, \psi)) & = \sigma^*(D \cup_{C} \delta C, \delta \phi \cup -(1+T) \delta \psi)  \\
 & = \sigma (D/C, \delta \phi/ (1+T) \psi) - \widehat{\sigma}^*(C, \psi) \in \zz_8,
\textnormal{~with~} L_{4k-1}(\zz)=0.
\end{align*}
\end{definition}
\begin{remark}Note that the hyperquadratic signature of a (symmetric, quadratic) Poincar\'e pair  is a cobordism invariant. If $(C \to D, (\delta \phi, \psi))$ is nullcobordant then $D \cup_C \delta C$ is the boundary of a symmetric complex $E$. A boundary has signature $0$, so $\sigma^*(D \cup_C \delta C)=0$. This is an algebraic interpretation of Lemma C in \cite{brumfield-Morgan}.
\end{remark}

We are interested in relating the signature modulo $8$ of a $4k$-dimensional symmetric Poincar\'e complex over $\zz$ with the Brown-Kervaire invariant of a linking form. For this purpose, the following is the most relevant part of the braid in figure \ref{fig:braid},
\begin{equation} \label{Braid-essential}
 \xymatrix{& L_{4k}(\zz[\frac{1}{2}]) =  L^{4k}(\zz[\frac{1}{2}])  \ar[dr]^{\partial_*} & \\
L^{4k}(\zz) \ar[ur] \ar[dr]& & L_{4k}(\zz, (2)^{\infty})  \ar@/^1.7pc/[dl]^{BK} \\
& \widehat{L}^{4k}(\zz) \ar[ur] & }
\end{equation}
We shall first discuss this diagram (\ref{Braid-essential}) for the case when $k=0$, i.e. with forms.

A nonsingular symmetric  form $(K, \lambda) \in L^{0}(\zz)$ has signature $\sigma(K, \lambda) \in \zz$, which reduced modulo $8$ is sent to the hyperquadratic signature of the normal complex $\widehat{\sigma}^*(K, \lambda, \gamma, \theta)$ via the map $L^{0}(\zz) \to \widehat{L}^{0}(\zz)$, where $(\gamma, \theta)$ is a representative of the unique equivalence class of normal structures on the symmetric form $(K, \lambda)$ as defined in \ref{chain-bundles}.

We can  trace this diagram to see how a nonsingular symmetric form in $L^{0}(\zz)$ corresponds to a nonsingular quadratic linking form in $L_{0}(\zz, (2)^{\infty}).$
The map $L^{0}(\zz) \to L_{0}(\zz \left[\frac{1}{2}\right]) =  L^{0}(\zz \left[\frac{1}{2}\right])$ sends the nonsingular symmetric form $(K, \lambda)$ over $\zz$ to a  nonsingular symmetric form $\zz[\frac{1}{2}] \otimes_{\zz}(K , \lambda)$ over $\zz \left[\frac{1}{2}\right]$. Since $2$ is a unit in $\zz \left[\frac{1}{2}\right]$, the morphism of forms over $\zz$
$$ 2: (K,  \lambda) \to (K, 4 \lambda)$$
becomes an isomorphism in $\zz \left[\frac{1}{2}\right]$,
$$\zz \left[\frac{1}{2}\right] \otimes_{\zz}(K , \lambda) = \zz \left[\frac{1}{2}\right] \otimes_{\zz}(K , 4 \lambda) \in L_0 \left(\zz \left[\frac{1}{2}\right] \right).$$
Note that $\zz \left[\frac{1}{2}\right] \otimes_{\zz}(K , 4 \lambda) = \zz \left[\frac{1}{2}\right] \otimes_{\zz}(K , 2 \lambda + 2 \lambda^*)$, so this is an even symmetric, i.e. quadratic form.
Following \cite{Mod8},  the form $(K, 2\lambda + 2 \lambda^*)$ is mapped to a quadratic linking form $(T, b, q)$ under the boundary map $L_{0}(\zz \left[ \frac{1}{2} \right]) \xrightarrow{\partial_*} L_{0}(\zz, (2)^{\infty})$, described in detail below.

Recall that a symmetric form $(F, \alpha)$ over $\zz$ is nondegenerate if $\alpha: F \to F^*$ is injective. A symmetric form $(F, \alpha)$ over $\zz$ determines a nondegenerate form
$$(F', \alpha') = (F / \textnormal{Ker}(\alpha: F \to F^*), [\alpha])$$
which has the same signature,
$$\sigma(F', \alpha') = \sigma(F, \alpha) \in \zz.$$

\begin{proposition}  (\hspace{-1pt}\cite{Mod8, Brum-Morg}) The boundary of a {\bf nondegenerate}  symmetric form $(F, \alpha)$ over $\zz$ is the nonsingular symmetric linking form over $(\zz, (2)^{\infty})$
$$T= \partial^*(F, \alpha) = (\partial F, b),$$
with
$$\begin{array}{l}
\partial F = \textnormal{coker}(\alpha: F \to F^*) \\
b : \partial F \times \partial F \to \zz\left[\frac{1}{2} \right]/ \zz; (x,y) \mapsto \alpha^{-1}(x, y) \textnormal{ with } x, y \in T. \\
\end{array}
 $$
If the nonsingular symmetric form is even then it admits a quadratic enhancement,
$$
\begin{array}{l}
q : \partial F \to \zz_4; x \mapsto \frac{\alpha^{-1}(x,x)}{2}.
\end{array}
 $$
\end{proposition}

Applying this proposition to our nonsingular form $\left[\frac{1}{2}\right] \otimes_{\zz}(K , 4 \lambda) \in L_0 \left(\zz \left[\frac{1}{2}\right] \right)$, we can construct a nonsingular quadratic linking form given by
$$
\begin{array}{l}
T = \textnormal{coker}(4: K \to K)_{*+1}= K / 4K, \\
b : T \times T \to \zz_2 ; (x,y) \mapsto (4\lambda)^{-1}(x, y) \textnormal{ with } x, y \in T, \\
q : T \to \zz_4; x \mapsto \frac{(4\lambda)^{-1}(x,x)}{2}
\end{array}
$$
where $\lambda : K \to K^*$ is an isomorphism and $\lambda^{-1}: K^* \to K$ defines a symmetric form on $K^*$.
There is an exact sequence
$$0 \to U =K/ 2K \to  T= K/ 4K \to  \U = K/ 2K \to 0.$$
The form $(T, b, q )= \left(K/ 4K, (4 \lambda)^{-1}, \frac{(4 \lambda)^{-1}}{2}\right) \in L_{4k}(\zz, (2)^{\infty})$ is quadratic and the nonsingular symmetric form $(T, b)$ has a lagrangian $U= K/ 2K$.  Hence we can form a $0$-dimensional (symmetric, quadratic) complex given by
$$\left(\mc{C}(4:K \to K)_{*+1} \to \mc{C}(2:K \to K)_{*+1}, \left(0, \frac{(4 \lambda)^{-1}}{2}\right)\right) \in \widehat{L}^{0}(\zz).$$
In the braid in figure \ref{fig:braid}, we see that the map $\widehat{L}^{0}(\zz)\to L_{4k}(\zz, (2)^{\infty})$ is an injection. The Brown-Kervaire invariant of the quadratic linking form $(T, b, q)= \left(K/ 4K, (4 \lambda)^{-1}, \frac{(4 \lambda)^{-1}}{2}\right) \in L_{4k}(\zz, (2)^{\infty})$ determines the same Witt class as the hyperquadratic signature of the (symmetric, quadratic) pair  $\left(\mc{C}(4:K \to K)_{*+1} \to \mc{C}(2:K \to K)_{*+1}, \left(0, \frac{(4 \lambda)^{-1}}{2}\right)\right)$, that is,
\begin{multline*}
  \widehat{\sigma}^*\left(\mc{C}(4:K \to K)_{*+1} \to \mc{C}(2:K \to K)_{*+1}, \left(0, \frac{(4 \lambda)^{-1}}{2}\right)\right) = \\
  \textnormal{BK}\left(K/ 4K, (4 \lambda)^{-1}, \frac{(4 \lambda)^{-1}}{2}\right)  \in \zz_8.
\end{multline*}
(See \cite[section 4]{BanRan}).

We need to check that $\widehat{\sigma}^*\left(\mc{C}(4:K \to K)_{*+1} \to \mc{C}(2: K \to K)_{*+1}, \left(0, \frac{(4 \lambda)^{-1}}{2}\right)\right) $ coincides with $\sigma(K, \lambda) \pmod{8}.$

The (symmetric, quadratic) pair $\left(\mc{C}(4:K \to K)_{*+1} \to \mc{C}(2:K \to K)_{*+1}, \left(0, \frac{(4 \lambda)^{-1}}{2}\right)\right) $ can be expressed as a normal complex,
\begin{multline*}
  \left(\mc{C}(4: K \to K)_{*+1} \to \mc{C}(2: K \to K)_{*+1}, \left(0, \frac{(4 \lambda)^{-1}}{2}\right)\right)  =  \\
  \left(\mc{C}(2: K \to K)_{*+1}, [\lambda], [\gamma], [\theta]\right) \in \widehat{L}^{4k}(\zz).
\end{multline*}

\begin{proposition} \label{cobordism-normal-cx}
The normal form $(K, \lambda, \gamma, \theta)$ is cobordant to $ \left(\mc{C}(2: K \to K)_{*+1}, [\lambda], [\gamma], [\theta]\right).$
\end{proposition}
\begin{proof}

The normal structure on $(K, \lambda, \gamma, \theta)$ restricts to $0$ on $2K \subset K$, so there is a quotient normal structure on
$\left(\mc{C}(2: K \to K)_{*+1}, [\lambda], [\gamma], [\theta]\right).$ That is, there is an algebraic normal map $K \to \mc{C}(2)$ as defined in \cite[pages 48 and 49]{bluebook}.
The quotient normal structure on $\left(\mc{C}(2: K \to K)_{*+1}, [\lambda], [\gamma], [\theta]\right)$ is the pullback from the universal chain bundle $(B, \beta)$ by the Wu class $K \xrightarrow{v} \zz.$
$$
\xymatrix{
K \ar[r]^{\hspace{0pt}v} \ar[d]^{2} &  \bb{Z}  \ar[d]^2\\
K \ar[r]^{\hspace{0pt}v}  & \bb{Z}}
$$
(See section \ref{Twisted-Q-groups} for the definition of the universal chain bundle.)
The normal structure is determined by this chain map to the universal chain bundle, which identifies its class in the universal twisted $Q$-group $Q_{0}(B, \beta)$ and consequently its cobordism class  in $\widehat{L}^{0}(\zz)$
as a normal complex.

The classifying map of $(K, \lambda, \gamma, \theta)$ factors through that of the map for $\left(\mc{C}(2: K \to K)_{*+1}, [\lambda], [\gamma], [\theta]\right),$
$$
\xymatrix{
 & K \ar[r]^{\hspace{0pt}v} \ar[d]^{2} &  \bb{Z}  \ar[d]^2\\
K \ar[r]^1 & K \ar[r]^{\hspace{0pt}v}  & \bb{Z}}
$$
so the class of the normal structure of $(K, \lambda, \gamma, \theta)$ is identified with the same class as $\left(\mc{C}(2: K \to K)_{*+1}, [\lambda], [\gamma], [\theta]\right)$ in the universal twisted $Q$-group $Q_{0}(B, \beta)$. Consequently both complexes, $(K, \lambda, \gamma, \theta)$  and $\left(\mc{C}(2: K \to K)_{*+1}, [\lambda], [\gamma], [\theta]\right)$ are cobordant in $\widehat{L}^{0}(\zz).$

\end{proof}

If $(K, \lambda, \gamma, \theta)$ and $\left(\mc{C}(2: K \to K)_{*+1}, [\lambda], [\gamma], [\theta]\right)$ are cobordant,  then $(K, \lambda, \gamma, \theta)$ is also cobordant to $\left(\mc{C}(4: K \to K)_{*+1} \to \mc{C}(2:K \to K)_{*+1}, \left(0, \frac{(4\lambda)^{-1}}{2}\right)\right).$

Hence,
\begin{multline*}
\widehat{\sigma}^*\left(\mc{C}(4: K \to K)_{*+1} \to \mc{C}(2:K \to K)_{*+1}, \left(0, \frac{(4\lambda)^{-1}}{2}\right)\right) =  \\ \widehat{\sigma}^*(K, \lambda, \gamma, \theta)\equiv  \sigma(K, \lambda) \pmod{8}.
\end{multline*}
This means that the diagram \eqref{Braid-essential} on page \pageref{Braid-essential} coming from the relevant part of the braid in figure \ref{fig:braid}   is commutative, so that,
$$\textnormal{BK} \left(\mc{C}(4\lambda)_{*+1}, (4 \lambda)^{-1}, \frac{(4 \lambda)^{-1}}{2} \right) = \sigma(K, \lambda)\in \zz_8.$$

\begin{example}
The symmetric form $(\zz, 4,0)$ has signature $1$. The boundary map gives a quadratic linking form
$$\partial_*(\zz, 4, 0) = (\zz_4, 1/4, 1/8 ),$$
We can compute the Brown-Kervaire invariant of this quadratic linking form by using the Gauss sum formula from Definition \ref{Brown-linking},
\begin{align*}
\sum_{x \in T} e^{2 \pi i q(x)} & =  e^{2 \pi i / 4} + e^{2\pi i} \\
& = i + 1
\end{align*}
This has argument $\pi/4$, hence
$$BK(\zz_4, 1/4, 1/8) = \sigma(\zz, 4, 0) = 1 \in \zz_8. $$
\end{example}

\subsection{Morita's theorem} \label{Morita-theorem}
The original statement of  \cite[theorem~1.1]{Morita} relating the Brown-Kervaire invariant and the signature modulo $8$ is formulated geometrically and it relates the signature of a $4k$-dimensional Poincar\'e space $X$ and the Brown-Kervaire invariant of a Pontryagin square, which is a quadratic enhancement of the cup product structure on the $\bb{Z}_2$-vector space $H^{2k}(X; \mathbb{Z}_2)$.

Before Morita proved his theorem, Brown had already given a proof in a special case:

\begin{theorem}\label{Brown's theorem} (\hspace{-1pt}\cite[Theorem 1.20 (viii)]{Brown})
Let $U$ be a finitely generated  free abelian group. If $\Phi: U \otimes U \to \mathbb{Z}$ is a symmetric bilinear form over $\mathbb{Z}$ with determinant $1$ or $-1$, and  $\eta: U/2U \to \mathbb{Z}_4$ is defined by $\eta(u) = \Phi(u, u) \in \zz_4$, then $\eta$ is quadratic and
$$\textnormal{BK}(U/2U, \Phi, \eta) = \sigma(U, \Phi)\in \zz_8.$$
\end{theorem}

\cite[Theorem 1.20 (viii)]{Brown} is purely algebraic and has a straightforward proof.

Morita's theorem is given in \cite{Morita} as follows:

\begin{theorem}\label{Morita} (\hspace{-1pt}\cite[theorem~1.1]{Morita})
Let $X$ be a $4k$-dimensional Poincar\'e space, then
$$\sigma (X)  \equiv \textnormal{BK}(H^{2k}(X; \mathbb{Z}_2), \lambda, \mc{P}_2) \in \zz_8.$$
\end{theorem}

The proof given in \cite{Morita} involves spectral sequences, and a further proof given in \cite{Taylor} involves Gauss sums. In our proof we will reformulate the theorem in terms of symmetric Poincar\'e complexes $(C, \phi)$. The Pontryagin square which geometrically depends on the cup and cup-$1$ products, depends algebraically on the symmetric structure $\phi$ as was explained in Chapter \ref{Pontryagin-squares chapter}. This is denoted by $\mc{P}_2(\phi)$, although for simplicity we will  write $\mc{P}_2$ in what follows.
\begin{theorem} \label{Morita-theorem-cx}
Let $(C, \phi)$ be a $4k$-dimensional symmetric Poincar\'e complex over $\mathbb{Z}$, then
$$\sigma (C, \phi) \equiv \textnormal{BK}(H^{2k}(C; \mathbb{Z}_2), \phi, \mathcal{P}_2)  \in \zz_8.$$
\end{theorem}

The proof of Morita's theorem (\ref{Morita}) using the chain complex approach relies on the following key points:

\begin{itemize}
\item[(i)] There is an isomorphism given by the signature $L^{4k}(\bb{Z}) \overset{\sigma}{\longrightarrow} \bb{Z},$
\item[(ii)] There is an injective map $\widehat{L}^{4k}(\bb{Z}) \to  L_{4k}(\zz, (2)^{\infty}),$
\item[(iii)]   $\textnormal{BK}(H^{2k}(C, \bb{Z}_2), \phi, \mathcal{P}_2) \in L_{4k}(\zz, (2)^{\infty}),$
\item[(iv)] $\sigma(C,\phi)$ maps to  $\textnormal{BK}(H^{2k}(C, \bb{Z}_2), \phi, \mathcal{P}_2)$ under the composition of maps
$$
\begin{array}{ccccc}
L^{4k}(\bb{Z}) &\to &\widehat{L}^{4k}(\bb{Z}) & \to & L_{4k}(\zz, (2)^{\infty})  \\
\sigma(C, \phi) &\mapsto &\widehat{\sigma}^*(C, \phi, \gamma, \theta)& \mapsto & \textnormal{BK}(H^{2k}(C, \bb{Z}_2), \phi,  \mathcal{P}_2).
\end{array}
$$

\end{itemize}

It is not possible to pass directly from Theorem \ref{Brown's theorem} (Brown) to Theorem \ref{Morita} (Morita) because in general
$$H^{2k}(X; \mathbb{Z}_2) \neq H^{2k}(X)/ 2H^{2k}(X),$$
the difference being given by the exact sequence,
$$0 \longrightarrow H^{2k}(X)/ 2H^{2k}(X)  \longrightarrow H^{2k}(X; \mathbb{Z}_2) \longrightarrow \textnormal{ker}(2)\longrightarrow 0$$
with $2: H^{2k}(X) \to H^{2k}(X)$.

In the proof of Theorem \ref{Morita} Morita uses spectral sequences to prove that the torsion in the middle dimension $H^{2k}(X)$ does not contribute to the Brown-Kervaire invariant.

The proof of Morita's theorem (as in \ref{Morita-theorem-cx}) simplifies considerably by using algebraic symmetric Poincar\'e complexes $(C, \phi)$. Every $4k$-dimensional symmetric Poincar\'e complex over $\zz$ is cobordant to one that is connected below the middle dimension.
So it can be concentrated in the middle dimension
$$C: \dots \longrightarrow 0 \longrightarrow C_{2k} \longrightarrow 0 \longrightarrow \dots$$
Such a symmetric Poincar\'e complex is homotopy equivalent to a symmetric form. The essential idea is that by using symmetric Poincar\'e complexes instead of Poincar\'e spaces we can reduce the proof of Morita's theorem to Brown's simpler case by chain complex methods.

\begin{proof} (of Theorem \ref{Morita-theorem-cx})
A $4k$-dimensional symmetric Poincar\'e complex $(C, \phi)$ over $\zz$ is cobordant to the symmetric form $(F^{2k}(C), \phi_0)$, where $F^{2k}(C)$ is the finitely generated free abelian group given by $F^{2k}(C)= H^{2k}(C, \zz) / torsion$ and $(F^{2k}(C), \phi_0)$ is a nonsingular symmetric form over $\zz$ with $\phi_0: F^{2k}(C) \times F^{2k}(C) \to \zz.$
 It is defined in \cite{atsI} that the signature of a nonsingular symmetric Poincar\'e complex $(C, \phi)$ over $\zz$ is given by,
$$\sigma(C, \phi) = \sigma(F^{2k}(C), \phi_0) \in \zz,$$

Every symmetric Poincar\'e complex gives rise to a normal complex,

\begin{minipage}[c]{0.4\linewidth}
\vspace{-33pt}
\begin{align*}
L^{4k}(\zz) &\to \widehat{L}^{4k}(\zz) \\
(C, \phi) & \mapsto (C, \phi, \gamma, \theta )
\end{align*}

\end{minipage}
\begin{minipage}[c]{0.40\linewidth}

\begin{align*}
\vspace{-50pt}
L^{4k}(\zz) &\to \widehat{L}^{4k}(\zz) \\
(F^{2k}(C), \phi_0) & \mapsto (F^{2k}(C), \phi_0, [\gamma], [\theta] )
\end{align*}
\vspace{10pt}
\end{minipage}

\vspace{-20pt}
Clearly, since $(C, \phi)$ and $ (F^{2k}(C), \phi_0)$ are cobordant, then the normal complexes $(C, \phi, \gamma, \theta )$ and $(F^{2k}(C), \phi_0, [\gamma], [\theta] )$ are also cobordant. (\cite{atsII})

From Proposition \ref{cobordism-normal-cx} we know that $(F^{2k}(C), \phi, [\gamma], [\theta] )$ is cobordant to the normal complex  $(\mc{C}(2:F^{2k}(C) \to F^{2k}(C) ), [\phi]', [\gamma]', [\theta]')$. Similarly $(C, \phi, \gamma, \theta)$ is cobordant to $(\mc{C}(2: C \to C), \phi', \gamma', \theta')$.

Since $F^{2k}(C)$ is a finitely generated free abelian group, we can apply \cite[theorem 1.20 (viii)]{Brown}. We define a $\zz_4$-valued quadratic form
$$ \Phi : F^{2k}(C) / 2 F^{2k}(C) \to \zz_4$$ by
$\Phi(u) = \phi_0(u, u) \in \zz_4.$
Hence applying Brown's theorem, we find that
$$\sigma(F^{2k}(C),\phi_0)  = \textnormal{BK}\left(F^{2k}(C)/ 2 F^{2k}(C), \phi_0, \Phi \right) \in \zz_8,$$
where $\phi_0: F^{2k}(C)/ 2 F^{2k}(C) \times F^{2k}(C)/ 2 F^{2k}(C) \to \zz_2$ is the mod $2$ reduction of $\phi_0.$
Note that
$$ F^{2k}(C)/ 2 F^{2k}(C) = H^{2k}(\mc{C}(2:F^{2k}(C) \to F^{2k}(C) ).$$
Similarly, the normal complex $(\mc{C}(2: C \to C), \phi', \gamma', \theta') \in \widehat{L}^{4k}(\zz)$ is mapped to $\textnormal{BK}(H^{2k}(C;\zz_2), \phi, q) \in L_{4k}(\zz, (2)^{\infty}).$

The description of the quadratic enhancement $q$ in the quadratic linking form $(H^{2k}(C;\zz_2), \phi, q)$ as the Pontryagin square $\mc{P}_2$ follows directly from the argument in chapter \ref{Pontryagin-squares chapter}, where we described how the symmetric structure $\phi$ of the symmetric Poincar\'e complex $(C, \phi)$ over $\zz$  is sent to the Pontryagin square,
\begin{align*}
(u,v)^{\%}: Q^{4k}(C) &\longrightarrow   Q^{4k}\left(S^{2k+1}\bb{Z} \xrightarrow{2}S^{2k}\bb{Z} \right) = \zz_4 \\
\phi & \longmapsto \mc{P}_2(u, v) = \phi_0(v, v) + 2 \phi_1(v, u),
\end{align*}
\vspace{-2pt}
that is,
\vspace{-5pt}
\begin{align*}
\mc{P}_2 : H^{2k}(C ; \zz_2) & \longrightarrow \zz_4 \\
    (u, v) & \longmapsto \phi_0(v, v) + 2 \phi_1(v, u).
\end{align*}
\raggedbottom
Hence,
$$\sigma(C, \phi) = \textnormal{BK}(H^{2k}(C;\zz_2), \phi, \mc{P}_2) \in \zz_8,$$
and the result follows.

Note also that since $\sigma(C, \phi) = \sigma(F^{2k}(C), \phi_0) \in \zz,$ the corresponding Brown-Kervaire invariants are equal,
$$\textnormal{BK}\left(H^{2k}(C;\zz_2), \phi, \mc{P}_2\right) = \textnormal{BK}\left(F^{2k}(C)/ 2 F^{2k}(C), \phi_0, \Phi \right) \in \zz_8. $$

\end{proof}

\section{Obstructions to divisibility of the signature by $8$}\label{obstructions}
\subsection{The normal structure}
Both Theorem \ref{sign-theta} which gives the mod $8$ signature of a $4k$-dimensional normal complex over $\bb{Z}$ by
$$\widehat{\sigma}(C, \phi, \gamma, \theta) = \phi_0(v, v) + 2 \phi_1(v, u) + 4 \theta_{-2}(u, u) \in Q_{4k}(B(k, 1), \beta(k, 1)) =  \bb{Z}_8 $$
and Morita's theorem relating the signature and the Brown-Kervaire invariant of the Pontryagin square have important applications to the study of the signature of a fibration.
The chain complex proof of Morita's theorem will be useful in Chapter \ref{mod-eight-proof}.
We now discuss the application of Theorem \ref{sign-theta}.

In \cite{Weiss-Ker} the following exact sequence including the twisted $Q$-groups is defined,
$$
\begin{array}{rcccccccc} \dots \widehat{Q}^{n+1}(C)  & \longrightarrow & Q_n(C, \gamma) & \overset{N_{\gamma}}{\longrightarrow} & Q^n(C) & \overset{J_{\gamma}}{\longrightarrow} & \widehat{Q}^n(C) & \longrightarrow \dots \\
            &                 & (\phi, \theta) & \longmapsto & \phi & \longmapsto & {\tiny J(\phi)- (\phi_0)^{\%}(S^n \gamma)} & & \\
\theta & \longmapsto &(0, \theta) &   &&  & \\
\end{array}
$$

This exact sequence involving twisted $Q$-groups is given in the particular case of the chain bundle $(B(k, 1), \beta(k, 1))$ by
\begin{equation}\label{ex-q-twisted} Q^{4k+1}(B(k, 1)) \xrightarrow{0} \widehat{Q}^{4k+1}(B(k,1)) \to Q_{4k}(B(k, 1), \beta(k, 1)) \to Q^{4k}(B(k, 1)) \to \widehat{Q}^{4k}(B(k,1)).
\end{equation}

From Theorem \ref{sign-theta} we know that the mod $8$ signature of a symmetric chain complex can be viewed as an element in the twisted $Q$-group $Q_{4k}(B(k, 1), \beta(k, 1))$, so we are interested in explicit descriptions of the $Q$-groups in the exact sequence \ref{ex-q-twisted}.
The following computation of appears in the preprint by A. Ranicki and L. Taylor \cite{Mod8} and it can be deduced from the computations in Banagl and Ranicki \cite[Proposition 52, Corollary 61]{BanRan}
\begin{theorem} \label{Q-groups computation} (\hspace{-1pt}\cite{Mod8})
\begin{itemize}
\item[(i)] The $4k$-dimensional symmetric $Q$-group of $B(k, m)$ is
$$Q^{4k}(B(k,m)) = \bb{Z}_{4m} $$
with an isomorphism
$$ Q^{4k}(B(k,m)) \to \bb{Z}_{4m} ; \phi \to \phi_0(1,1) + 2 \phi_1(1,1).$$
\item[(ii)] The $4k$-dimensional \textbf{twisted} $Q$-group of $(B(k,m), \beta(k,m))$ is
$$Q_{4k}(B(k,m), \beta(k,m)) = \left\{ \begin{array}{ccc} \bb{Z}_{8m} &if & m \textnormal{ is odd} \\
\bb{Z}_{4m} &if & m \textnormal{ is even}\end{array} \right. $$
\end{itemize}
\end{theorem}

In Theorem \ref{sign-theta} we have proved that
$$\widehat{\sigma}(C, \phi, \gamma, \theta) = \phi_0(v, v) + 2 \phi_1(v, u) + 4 \theta_{-2}(u, u) \in Q_{4k}(B(k, 1), \beta(k, 1)) =  \bb{Z}_8.$$
So we  are particularly interested in the computations in Theorem \ref{Q-groups computation} when $m=1$, that is
$$Q^{4k}(B(k,1)) \to \bb{Z}_{4} ; \phi \mapsto \phi_0(1,1) + 2 \phi_1(1,1) $$
and

$$Q_{4k}(B(k,1), \beta(k,1)) \to \bb{Z}_8 ; (\phi, \theta) \mapsto \phi_0(1,1) + 2 \phi_1(1,1) + 4\theta_{-2}(1,1). $$

Furthermore we are also interested in the following computations of $Q$ groups of the chain complex $B(k, 1):  B(k, 1)_{2k+1}=\zz \xrightarrow{2} B(k, 1)_{2k}= \zz$,
\begin{proposition} We have the following isomorphisms,

\begin{itemize}
\item[(i)] $\widehat{Q}^{4k+1}(B(k,1)) \cong \bb{Z}_2.$
\item[(ii)] $\widehat{Q}^{4k}(B(k,1)) \cong \bb{Z}_2.$
\item[(iii)] $Q^{4k+1}(B(k, 1)) \cong \bb{Z}_2.$
\item[(iv)] The map $Q^{4k+1}(B(k, 1)) \cong \bb{Z}_2 \xrightarrow{J_{\beta}=0}\widehat{Q}^{4k+1}(B(k,1)) \cong \bb{Z}_2$ is the zero map.
\item[(v)] The map $Q^{4k}(B(k, 1)) \cong \bb{Z}_4 \xrightarrow{J_{\beta}=0}\widehat{Q}^{4k}(B(k,1)) \cong \bb{Z}_2$ is the zero map.
\end{itemize}
\end{proposition}

\begin{proof}

\begin{itemize}
\item[(i)], (ii) Using \cite[Proposition 13 (i)]{BanRan} we know that $\widehat{Q}^{4k+1}(B(k,1))$  and $\widehat{Q}^{4k}(B(k,1))$ fit into the exact sequence
\begin{multline}
\dots \to H^1(\bb{Z}_2; \bb{Z}) = 0 \xrightarrow{\widehat{d}^{\%}} H^1(\bb{Z}_2; \bb{Z}) =0  \to \widehat{Q}^{4k+1}(B(k,1)) \to H^0(\bb{Z}_2; \bb{Z})= \bb{Z}_2 \xrightarrow{\widehat{d}^{\%}}  \\   H^0(\bb{Z}_2; \bb{Z}) = \bb{Z}_2 \to   \widehat{Q}^{4k}(B(k,1)) \to H^{-1}(\bb{Z}_2; \bb{Z})=0
\end{multline}
As the differential $d=2$ and $\bb{Z}$ is a commutative ring with even involution then
$\widehat{d}^{\%}=0$. So there are defined isomorphisms
$$\widehat{Q}^{4k+1}(B(k,1))\xrightarrow{\cong}  H^0(\bb{Z}_2; \bb{Z})= \bb{Z}_2$$
and
 $$H^0(\bb{Z}_2; \bb{Z}) = \bb{Z}_2 \xrightarrow{\cong}   \widehat{Q}^{4k}(B(k,1)).$$

\item[(iii)] We know from \cite{BanRan} that $Q^{4k+1}(B(k, 1)) \neq 0$. We also know that
$$Q_{4k+1}(B(k,1) \beta(k,1)) \xrightarrow{\cong} Q_{4k+1}(B(\infty), \beta(\infty)) \xrightarrow{\cong} \widehat{L}^{4k+1}(\bb{Z}) = \bb{Z}_2,$$
so from the exact sequence \eqref{ex-q-twisted}, we deduce that $Q^{4k+1}(B(k, 1)) = \bb{Z}_2.$
\item[(iv)] The map $$\widehat{Q}^{4k+1}(B(k, 1)) = \bb{Z}_2 \xrightarrow{4} Q_{4k}(B(k, 1), \beta(k,1))= \bb{Z}_8$$ is injective. This forces the map $Q^{4k+1}(B(k, 1)) = \bb{Z}_2 \xrightarrow{J_{\beta}=0}\widehat{Q}^{4k+1}(B(k,1)) =\bb{Z}_2$ in the exact sequence of $Q$-groups in \eqref{ex-q-twisted} to be the zero map.
\item[(v)] From the proof of Proposition \ref{signature-Psq} we know that the map $Q_{4k}(B(k,1), \beta(k,1)) \to Q^{4k}(B(k,1))$ is surjective. This implies that the next map in the exact sequence \eqref{ex-q-twisted} has to be the zero map. Hence the map $Q^{4k}(B(k, 1)) = \bb{Z}_4 \xrightarrow{J_{\beta}=0}\widehat{Q}^{4k}(B(k,1)) =\bb{Z}_2$ is the zero map.

\end{itemize}
\end{proof}

\begin{theorem} \label{obstruction-theta}
Let $(C,\phi)$ be a $4k$-dimensional symmetric Poincar\'e complex over $\zz$, and let $(\gamma,\theta)$  be the
canonical algebraic normal structure, with classifying chain bundle map $(f,\chi):(C,\gamma) \to (B(k,1),\beta(k,1))$.
The induced morphism $(f,\chi)_{\%}:Q_{4k}(C,\gamma) \to Q_{4k}(B(k,1),\beta(k,1))$ is such that
$$(f,\chi)_{\%}(\phi,\theta)~=~ \phi_0(v,v)~=~\sigma (C,\phi) \in Q_{4k}(B(k,1),\beta(k,1))=\zz_8$$
and is the signature of $(C,\phi)$ modulo $8$, for any lift of $v_{2k}(\phi) \in H^{2k}(C; \zz_2)$ to a class $v \in H^{2k}(C)$,
with image
$$f^{\%}(\phi)~=~\mc{P}_2(v_{2k}(\phi))~=~\sigma(C,\phi) \in Q^{4k}(B(k,1))=\zz_4$$
independent of $\theta$. Here, $\mc{P}_2:H^{2k}(C;\zz_2) \to \zz_4$ is the algebraic Pontryagin square.
If $\sigma(C,\phi) \equiv 0 \pmod{4}$ then
$$(f,\chi)_{\%}(\phi,\theta)=\phi_0(v,v) \in \textnormal{Im}(\widehat{Q}^{4k+1}(B(k,1)) \to Q_{4k}(B(k, 1), \beta(k, 1)) )= 4\zz_2  \subset \zz_8$$
depends on $\theta$.
\end{theorem}


\begin{proof}

By van der Blij's theorem \cite{Blij} the signature mod 8 of a nonsingular symmetric form
$(C,\phi)$ over $\zz$ is determined by any Wu class $v \in E$
$$\sigma(C,\phi)\equiv\phi_0(v,v) \bmod 8~.$$
We also know from \cite{Mod8} that
$$\sigma(C, \phi) \in Q_{4k}(B(k,1), \beta(k,1))= \zz_8.$$
(The statement and proof of this result were reproduced in Theorem \ref{sign-theta} )

Using the computations above we see that the exact sequence involving twisted $Q$-groups given in \cite{Weiss-Ker} is given in the particular case of the chain bundle $(B(k, 1), \beta(k, 1))$ by
$$Q^{4k+1}(B(k, 1)) \xrightarrow{0} \widehat{Q}^{4k+1}(B(k,1)) \to Q_{4k}(B(k, 1), \beta(k, 1)) \to Q^{4k}(B(k, 1)) \to \widehat{Q}^{4k}(B(k,1)) $$
that is,
$$\bb{Z}_2 \xrightarrow{0} \bb{Z}_2 \xrightarrow{4} \bb{Z}_8 \to \bb{Z}_4 \xrightarrow{0} \bb{Z}_2.$$
From this exact sequence we see that if
$$\theta = 1 \in \widehat{Q}^{4k+1}(B(k, 1)) = \bb{Z}_2,$$
then this is mapped to
$$\sigma(C, \phi) = 4 \in Q_{4k}(B(k, 1), \beta(k, 1)) =  \bb{Z}_8.$$
\end{proof}

\begin{remark}
This is specially interesting in the case when the $(C, \phi)$ is the chain complex of the total space of a fibration with base and fibre of dimensions congruent to $2$ modulo $4$.  In this case we know that the signature will be a multiple of $4$. So $\theta \in  \widehat{Q}^{4k+1}(B(k, 1))$ can be used to detect when the signature is divisible by $8$.
\end{remark}

\subsection{The Arf invariant} \label{BK-Arf-in-topology}
From Theorem \ref{Morita-theorem-cx} we know that if $(C, \phi)$ is a $4k$-dimensional symmetric Poincar\'e complex over $\mathbb{Z}$, then
$$\sigma (C, \phi) = \textnormal{BK}(H^{2k}(C; \mathbb{Z}_2), \phi_0, \mathcal{P}_2(\phi))  \in \zz_8.$$

In Proposition  \ref{BK-and-4Arf} we proved that if the Brown-Kervaire invariant of a $\zz_4$-valued nonsingular quadratic form $\textnormal{BK}(V, \lambda, q) \in \zz_8$ takes values $0$ or $4$ in $\zz_8$, then this Brown-Kervaire invariant can be expressed as  the classical Arf invariant of a $\zz_2$-valued form.

Combining Theorem \ref{Morita-theorem-cx}  with the result in Proposition  \ref{BK-and-4Arf} we can state the following theorem.
\begin{theorem} \label{4Arf-Algebra}
If the signature of a symmetric Poincar\'e complex $(C, \phi)$ takes value $0$ modulo $4$, then this signature modulo $8$ can be expressed as an Arf invariant as follows,
$$\sigma(C, \phi)=\textnormal{BK}(H^{2k}(C;\zz_2), \phi_0, \mc{P}_2(\phi)) = 4 \textnormal{Arf}\left( L^{\perp}/L , [\phi_0],  \frac{\left[\mc{P}_2(\phi) \right]}{2} \right) \in \zz_8,$$
where $L^{\perp} =\left\{x \in H^{2k}(C;\zz_2) \vert \phi_0(x, x)= 0 \in \zz_2 \right\}$ and the Wu sublagrangian $L = \langle v \rangle \subset L^{\perp}$, with $v$ the algebraic Wu class $v \in H^{2k}(C;\zz_2).$
\end{theorem}
\begin{proof}
This is a direct application of the proof of Proposition  \ref{BK-and-4Arf}.

\end{proof}

This algebraic theorem has the  following analogue in topology.\begin{theorem}\label{4Arf-topology}An oriented $4k$-dimensional geometric Poincar\'e space $M$ has signature $0$ mod $4$ if and only if
$L=\langle v_{2k}(M) \rangle \subset H^{2k}(M;\zz_2)$ is a sublagrangian of $(H^{2k}(M;\zz_2),\lambda,q)$. If such is the case, there is
defined a sublagrangian quotient nonsingular symmetric form over $\zz_2$ with a $\zz_2$-valued enhancement
$$(W,\mu,h) = (L^{\perp}/L , [\lambda] , h = [q]/2)$$
and the signature mod $8$ is given by
$$\sigma(M) = \textnormal{BK}(H^{2k}(M;\zz_2),\lambda,q) = 4\textnormal{Arf}(W,\mu,h) \in  4\zz_2 \subset \zz_8.$$
\end{theorem}

The proof is a direct application of an algebraic result, which we give in Proposition \ref{BK-and-4Arf}.


\begin{remark} Let $(C, \phi)$ be a $4k$-dimensional symmetric Poincar\'e complex over $\zz$, such that  the signature is divisible by $4$,
$$\sigma(C, \phi) = 0 \in 4 \zz \subset L^{4k}(\zz)=\zz.$$
 Then by Theorem \ref{4Arf-Algebra}  we can express the $\zz_2$-obstruction to divisibility by $8$ as an Arf invariant, since
$$\sigma(C, \phi)/4 ~=~ \textnormal{Arf}\left( L^{\perp}/L , [\phi_0],  \frac{\left[\mc{P}_2(\phi) \right]}{2} \right) \in \zz_2,$$
where $L^{\perp} =\left\{x \in H^{2k}(C;\zz_2) \vert \phi_0(x, x)= 0 \in \zz_2 \right\}$ and the Wu sublagrangian $L = \langle v \rangle \subset L^{\perp}$, with $v$ the algebraic Wu class $v \in H^{2k}(C;\zz_2).$

\end{remark}

We shall now give examples which illustrate some relevant applications of Theorem \ref{4Arf-topology}.

Let $\epsilon=\pm 1$.
A nonsingular $\epsilon$-symmetric form $(E,\phi)$ over $\zz$  has a signature $\sigma(E,\phi) \in \zz$ and a Wu class $v \in E$ such that
$$\phi(x,x) \equiv \phi(x,v) \bmod 2~(x \in E)~.$$
The Wu class is  unique up to $2E$. If $\epsilon=-1$ then $\sigma(E,\phi)=0$ and $v=0$.

By van der Blij's theorem \cite{Blij} the signature mod 8 of a nonsingular symmetric form
$(E,\phi)$ over $\zz$ is determined by any Wu class $v \in E$
$$\sigma(E,\phi)\equiv\phi(v,v) \bmod 8~.$$
As usual, $(E,\phi)$ has an associated $\zz_4$-enhanced nonsingular symmetric form
$(V,\lambda,q)$ over ${\mathbb Z}_2$ with
$$V~=~E/2E~,~\lambda([x],[y])~=~[\phi(x,y)] \in \zz_2~,~q([x])~=~[\phi(x,x)]\in \zz_4~(x,y \in E)~,$$
such that the isomorphism $BK:L\langle v_1 \rangle^0(\zz_2) \to \zz_8$ sends the Witt class of $(V,\lambda,q)$ to
$$\textnormal{BK}(V,\lambda,q)~=~[\sigma(E,\phi)]~=~[\phi(v,v)]\in {\mathbb Z}_8~.$$
The image of a Wu class $v \in E$ is the unique Wu class $[v] \in V$ for $(V,\lambda)$, with
$$\lambda(a,a)=\lambda(a,[v]) \in {\zz}_2~(a \in V)~,$$
and is such that
$$q([v])~=~[\textnormal{BK}(V,\lambda,q)]~=~[\sigma(E,\phi)]~=[\phi(v,v)] \in \zz_8/\zz_2~=~\zz_4~.$$
The following condition on $(E,\phi)$ are equivalent
\begin{enumerate}
\item $\sigma(E,\phi) \equiv 0 \bmod 4$,
\item $\phi(v,v) \equiv  0 \bmod 4$,
\item $q([v])=0 \in \zz_4$,
\item $\textnormal{BK}(V,\lambda,q) \in {\rm im}(4:\zz_2 \to \zz_8)$.
\end{enumerate}
If these conditions are satisfied the maximal isotropic subquotient of $(V,\lambda)$
$$(W,\mu)~=~(\,\langle [v] \rangle^{\perp}/\langle [v] \rangle\,,\,[\lambda]\,)$$
has
$$\begin{array}{ll}
W&=~\{y \in V\vert \lambda(y,[v])=0 \in \zz_2)\}/\langle [v]\rangle\\[1ex]
&=~\{ y \in V\vert q(y) \in {\rm im}(2:\zz_2 \to \zz_4)\}/\langle [v]\rangle
\end{array}$$
and admits a $\zz_2$-enhancement
$$h~:~W \to \zz_2~;~[x] \mapsto q([x])/2= [\phi(x,x)/2]$$
such that
$$[\sigma(E,\phi)]~=~[\textnormal{BK}(V,\lambda,q)]~=~4{\rm Arf}(W,\mu,h) \in 4\zz_2 \subset \zz_8~.$$

\begin{example}
Let
$$(E,\phi)=\bigoplus\limits_4 (\zz,1)~,~v~=~(1,1,1,1) \in E~,~
(V,\lambda,q)~=~\bigoplus\limits_4 (\zz_2,1,1)~.$$
The $\zz_2$-enhanced symmetric form $(W,\mu,h)$ over $\zz_2$ is given by
$$\begin{array}{l}
W~=~\{(x_1,x_2,x_3,x_4) \in V\vert x_1+x_2+x_3+x_4=0 \in \zz_2\}/\langle (1,1,1,1)\rangle~,\\[1ex]
\mu~:~W \times W \to \zz_2~;~((x_1,x_2,x_3,x_4),(y_1,y_2,y_3,y_4))
\mapsto \sum\limits^4_{i=1}x_iy_i~,\\[1ex]
h~:~W \to \zz_2~;~x\mapsto
0~{\rm if}~x=0~,~x \mapsto 1 ~{\rm if}~x \neq 0~.
\end{array}$$
As $h(x) = 1$ for each of the three elements $x \neq 0 \in W$
${\rm Arf}(W,\mu,h)=1 \in \zz_2$, in accordance with
$$\sigma(E,\phi)~=~4{\rm Arf}(W,\mu,h)~=~4 \in \zz_8~.$$

\bigskip
\end{example}

\begin{example}
\indent  Let $(B,\phi_B)$, $(F,\phi_F)$ be nonsingular
$\epsilon$-symmetric forms over $\mathbb Z$ with Wu classes $v_B \in B$, $v_F \in F$. The product nonsingular symmetric form over $\zz$
$$(E',\phi')~=~(B\otimes_{\mathbb Z}F,\phi_B\otimes \phi_F)$$
has signature
$$\sigma(E',\phi')~=~\sigma(B,\phi_F)\sigma(F,\phi_F) \in \zz~.$$
For any $a=\sum\limits_i x_i \otimes y_i \in E'$
$$\phi'(a,a)~=~\sum\limits_i\phi_B(x_i,x_i)\phi_F(y_i,y_i)+
2 \sum\limits_{i<j}\phi_B(x_i,x_j)\phi_F(y_i,y_j) \in \zz~,$$
so that $v' = v_B \otimes v_F\in E$ is a Wu class for $(E',\phi')$,
and the associated $\zz_4$-enhanced form over $\zz_2$ is
$$(V',\lambda',q')~=~(\,E'/2E'\,,\,([a],[b]) \mapsto [\phi'(a,b)]\,,\,[a] \mapsto [\phi'(a,a)]\,)$$
with
$$[\sigma(E',\phi')]~=~\textnormal{BK}(V',\lambda',q')~=~[\phi'(v',v')]~=~[\phi_B(v_B,v_B)][\phi_F(v_F,v_F)] \in \zz_8~.$$
Let $(E,\phi)$ be a nonsingular symmetric form over $\zz$ with Wu
class $v \in E$ and associated $\zz_4$-enhanced form over $\zz_2$
$$(V,\lambda,q)~=~(E/2E,([x],[y]) \mapsto [\phi(x,y)],[x] \mapsto [\phi(x,x)])~.$$
If
$$\sigma(E,\phi)~\equiv~\sigma(E',\phi') \bmod 4$$
then
$$(E'',\phi'')~=~(E,\phi) \oplus (E',-\phi')$$
is a nonsingular symmetric form over $\zz$ with Wu class $v''=(v,-v') \in E''$. The signature
$$\sigma(E'',\phi'')~=~\sigma(E,\phi)-\sigma(E',\phi') \in \zz$$
is such that
$$\sigma(E'',\phi'')~\equiv~0 \bmod 4~.$$
The associated $\zz_4$-enhanced form over $\zz_2$
$$(V'',\lambda'',q'')~=~(V,\lambda,q)\oplus (V',-\lambda',-q')$$
is such that
$$q''([v''])~=~q([v])-q'([v'])~=~0 \in \zz_4~.$$
The maximal isotropic subquotient of $(V'',\lambda'')$
$$(W,\mu)~=~(\,\langle[ v'' ]\rangle^\perp/\langle [v''] \rangle,[\lambda''])$$
admits a $\zz_2$-enhancement $h:W \to \zz_2$ and
$$\sigma(E'',\phi'')~=~4 {\rm Arf}(W,\mu,h) \in \zz_8~.$$
\end{example}
\begin{example}
There are two special cases of the previous example which are of interest:
\begin{enumerate}
\item If there exists an isomorphism $f:(V,\lambda)\to (V',\lambda')$ then
$$f([v])~=~[v'] \in V'$$
and
$$[\sigma(E,\phi)]-[\sigma(E',\phi')]~=~
[\textnormal{BK}(V,\lambda,q)] - [\textnormal{BK}(V',\lambda',q')]~=~0 \in \zz_8/2\zz_4~=~\zz_2$$
so that
$$\sigma(E,\phi) \equiv \sigma(E',\phi') \bmod 2~.$$
Furthermore
$$\begin{array}{ll}
[\sigma(E,\phi)]-[\sigma(E',\phi')]&=~
[\textnormal{BK}(V,\lambda,q)] - [\textnormal{BK}(V',\lambda',q')]\\[1ex]
&=~q''([v,v'])~=~q([v])-q'([v']) \\[1ex]
&\in
{\rm ker}(\zz_8 \to \zz_2)/{\rm im}(\zz_4 \to \zz_8)~=~\zz_4/2\zz_4~.
\end{array}
$$
Thus $\sigma(E,\phi) \equiv \sigma(E',\phi') \bmod 4$ if and only if
$q([v])=q'([v']) \in \zz_4$, if and only if the Wu class
$[v'']=([v],[v']) \in E \oplus E'$ for $(E'',\phi'')=(E,\phi) \oplus (E',-\phi')$
is such that $q''([v''])=0\in \zz_4$, in which case
$$\sigma(E'',\phi'')~=~\sigma(E,\phi)-\sigma(E',\phi') \equiv 4{\rm Arf}(W,\mu,h) \bmod 8$$
with
$$(W,\mu,h)~=~(\{(x,x') \in V'' |q(x)-q'(x') \in 2\zz_4 \subset \zz_4\}/\langle [v'']\rangle \,,\,\lambda \oplus -\lambda',(q \oplus -q')/2\,)~.$$
\item If there exists an isomorphism $f:(V,\lambda,q)\to (V',\lambda',q')$ (this is the case of a $\zz_4$-trivial $\pi_1(B)$-action as will be discussed in Chapter \ref{mod-eight-proof}) then
$$\begin{array}{ll}
[\sigma(E'',\phi'')]&=~ \textnormal{BK}(V'',\lambda'',q'')\\[1ex]
&=~\textnormal{BK}(V,\lambda,q)- \textnormal{BK}(V',\lambda',q')\\[1ex]
&=~0 \in \zz_8
\end{array}$$
and
$$\sigma(E,\phi)~\equiv~\sigma(E',\phi')~=~\sigma(B,\phi_B)\sigma(F,\phi_F)
\bmod\, 8~.$$
Note that
$$L~=~\{(x,f(x))\vert x\in V\}/\langle [v'']\rangle \subset W$$
is a lagrangian of $(W,\mu,h)$, so that ${\rm Arf}(W,\mu,h)=0 \in \zz_2$, in accordance with
$$\sigma(E'',\phi'')~=~\sigma(E,\phi) - \sigma(E',\phi')~=~4{\rm Arf}(W,\mu,h)~=~0 \in 4\zz_2 \subset \zz_8~.$$
\end{enumerate}

\end{example}

\subsection{The Hasse-Witt invariant in $K$-theory}\label{Hasse-Witt}

The Witt ring of $\bb{R}$, $W(\bb{R}) = L^0(\bb{R})$ is the abelian group of Witt-equivalence classes of nonsingular symmetric bilinear forms over $\bb{R}$. These are classified by the signature, so there is an isomorphism
$$\sigma: W(\bb{R}) \xrightarrow{\cong} \bb{Z}.$$
The fundamental ideal of $W(\bb{R})$ consisting of all Witt classes with even rank is given by $I = 2\bb{Z}$.
In \cite{Milnhus} the chain of ideals $I \supset I^2 \supset I^3 \supset \dots$  is discussed. Here we are specially interested in the quotient $I^2/ I^3$ in the case when $I= 2 \bb{Z}$, so the case when the chain of ideals is
$$2\bb{Z} \supset 4\bb{Z} \supset 8 \bb{Z} \supset \dots $$

Each of these ideals gives us information about the divisibility of the signature by $2$, by $4$ and by $8$ respectively.
Each quotient $I^n/ I^{n+1}$is a vector space over the field
$$W(\bb{R})/ I \cong \bb{Z}_2.$$

\begin{theorem} \cite[Theorem 5.8, chapter III]{Milnhus}
For each symbol $\phi$ the restriction of the Hasse-Witt function $h_{\phi}$ to the ideal $I^2=4\bb{Z}$ yields a well defined homomorphism
$$h_{\phi} : I^2 \to \bb{Z}_2.$$
An element $w \in I^2$ is annihilated by every one of these homomorphisms $h_{\phi}$ if and only if $w \in I^3.$
\end{theorem}

Applying this theorem to the context of the divisibility of the signature, we see that a form $w \in I^3= 8\bb{Z} \subset I^2 = 4\bb{Z}$ is divisible by $8$ when $h_{\phi} (w)$ is trivial. If $h_{\phi}(w')$ is non-trivial then $w'$ is a form in $I^2$ but not in $I^3$, so that the form $w'$ has signature divisible by $4$ but not by $8$. In other words the Hasse-Witt invariant detects when a form has signature divisible by $8$.

The relation between the Hasse-Witt invariant and the Arf invariant is discussed in \cite{Giffen}.

\part{The signature of a fibration}

\chapter*{Introduction to part II}
\addcontentsline{toc}{chapter}{Introduction to Part II}

A key feature for our study of the signature of a fibration is obtaining a model for the chain complex of the total space which gives us enough information to compute its signature. It has been known since the work in \cite{HirzebruchSerreChern}, \cite{Meyerpaper} and \cite{Korzen} that the signature of the total space depends only on the action of the fundamental group of the base $\pi_1(B)$ on the cohomology of the fibres.
Clearly it is not possible to construct the chain complex of the total space by taking into account only the action of $\pi_1(B).$ For example, the base space of the Hopf fibration
$S^1 \to S^3 \to S^2 $
has trivial fundamental group $\pi_1(S^2)= \{ 1 \}$, but the chain complex of the total space in this case is clearly not a product.
So taking into account the information from the chain complexes of the base and fibre and the action of $\pi_1(B)$ is not enough to construct the chain complex of the total space, but it is enough to construct a model that will detect the signature.

The model that we will develop in this chapter and use in subsequent chapters is inspired by the transfer map in quadratic $L$-theory given in \cite{SurTransfer}. This model was previously used in \cite{modfour} and in \cite{Korzen}.
In \cite{SurTransfer} the surgery transfer of a fibration $F \to E \to B$ with fibre of dimension $m$ and base of dimension $n$ is given by the map
$$\begin{array}{ccc}p^{!}: L_n(\bb{Z}[\pi_1(B)]) &\to & L_{n+m}(\bb{Z}[\pi_1(E)]). \end{array}$$

In \cite{SurTransfer} it is also proven that the surgery transfer map in quadratic $L$-theory agrees with the geometrically defined transfer maps.

A similar transfer map does not exist in symmetric $L$-theory. So it is \textbf{not} always possible to define a map
$$p^{!}: L^n(\bb{Z}[\pi_1(B)]) \to  L^{n+m}(\bb{Z}[\pi_1(E)]). $$
There are two obstructions to lifting a symmetric chain complex $(C, \phi) \in L^n(\bb{Z}[\pi_1(B)])$ to an $(m+n)$-dimensional chain complex $p^!(C, \phi) \in L^{n+m}(\bb{Z}[\pi_1(E)])$ which are described in the appendix of \cite{SurTransfer}.
Basically the difference with the quadratic $L$-theory transfer lies in the fact that the symmetric $L$-groups are not $4$-periodic, so that one cannot assume that we may perform surgery below the middle dimension to make $(C, \phi)$ is highly-connected.

The chain model that we shall discuss in this chapter will provide a well defined map
$$L^n(\bb{Z}[\pi_1(B)]) \to  L^{n+m}(\bb{Z}) $$
This map was constructed in \cite[chapters 3, 4]{Korzen}. We review the main ideas of the construction here, as it is relevant for further results in part III.

The construction uses the fact that the chain complex of the total space is \textit{filtered}, so the construction of our model for the total space is similar to that of a Serre spectral sequence. This was the approach taken by Meyer in \cite{Meyerpaper} where he describes the intersection form of the total space of a surface bundle in terms the intersection form on the base with coefficients in a local coefficient system.



\chapter{The algebraic model for the signature of a fibration} \label{model}

The ideas and notation in the following diagram have not yet been introduced. The purpose of the diagram is to give an overview of the key steps in the construction of a suitable model for a fibration $F \to E \to B$, which we will explain in detail in this chapter.
\vspace{-5pt}

\begin{figure}[ht!] \label{figure}
\labellist
\small\hair 2pt
\pinlabel ${\textnormal{Theorem \ref{3.15}}}$ at 130 445
\pinlabel ${E \textnormal{ is homotopy equivalent} }$ at 130 425
\pinlabel ${\textnormal{to a filtered $CW$-complex $X$} }$ at 130 405
\pinlabel ${\textnormal{Theorem \ref{chain-iso}}}$ at 340 445
\pinlabel ${\textnormal{There is an isomorphism} }$ at 350 425
\pinlabel ${G_*C(X) \cong C(\wtB) \otimes (C(\wtF), U) }$ at 350 405
\pinlabel ${\sigma(E) = \sigma (X) \in \bb{Z}}$ at 130 300
\pinlabel ${\sigma_{\bb{D}(\bb{Z})}(G_*C(X) ) }$ at 350 310
\pinlabel ${= \sigma_{\bb{D}(\bb{Z})}(C(\wtB) \otimes (C(F), U)) \in \bb{Z}}$ at 350 290
\pinlabel ${\textnormal{Proposition \ref{sig-ass-complex}}}$ at 240 195
\pinlabel ${\sigma (X) = \sigma_{\bb{D}(\bb{Z})}(G_*C(X) )  \in \bb{Z}}$ at 240 165
\pinlabel ${\textnormal{Theorem \ref{3.15} and Proposition \ref{two-functors}}}$ at 240 70
\pinlabel ${\sigma (E) = \sigma_{\bb{D}(\bb{Z})} ((C(\wtB), \phi) \otimes (C(F), \alpha, U))}$ at 240 45
\pinlabel ${ =\sigma ((C(\wtB), \phi) \otimes (H^m(F), \bar{\alpha}, \bar{ U}))}$ at 250 20
\endlabellist
\centering
\includegraphics[scale=0.8]{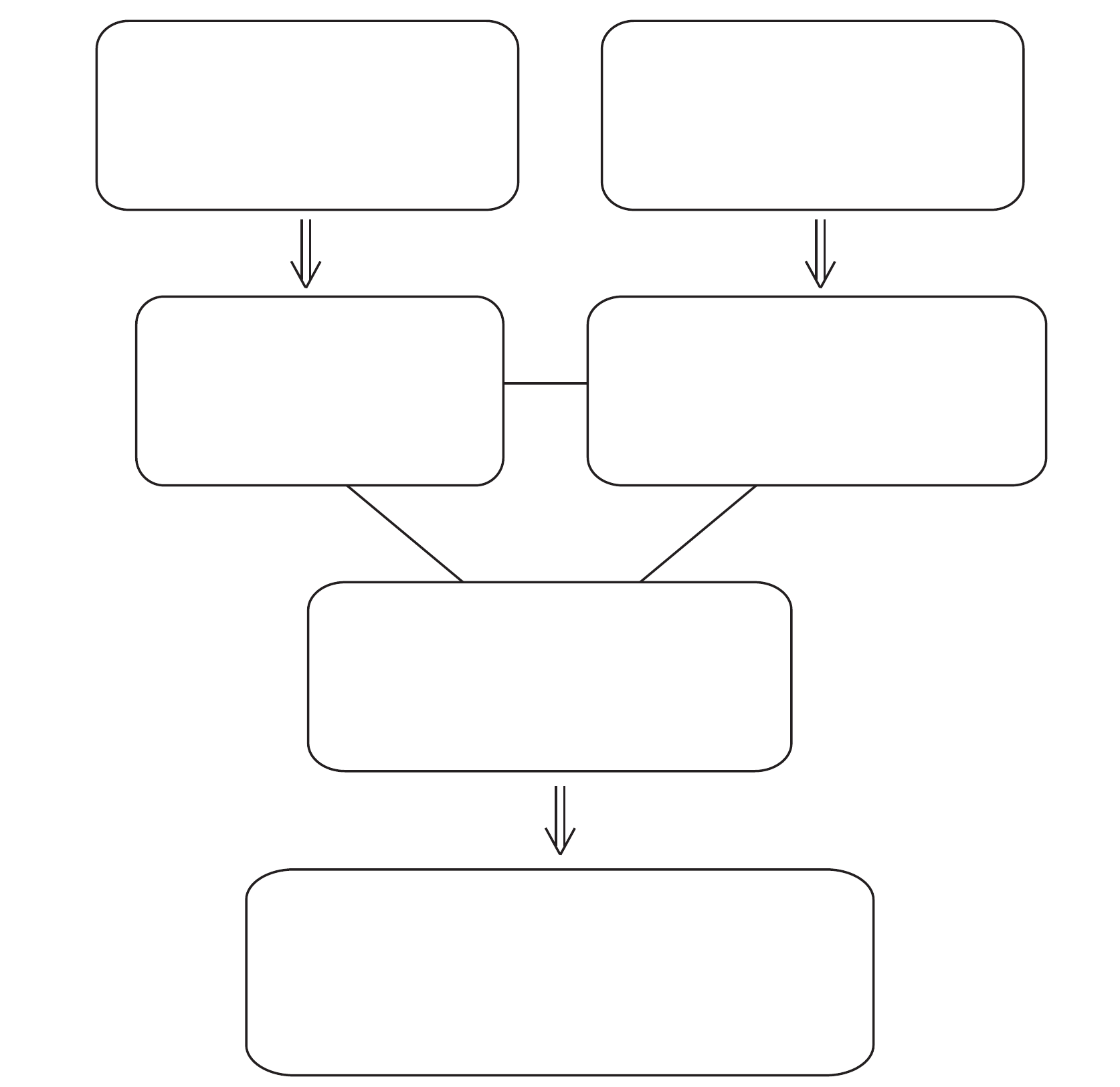}
\label{fig:cobo}
\end{figure}

\section{Fibrations and $\Gamma$-fibrations}

\begin{definition} A Hurewicz fibration is a continuous function $p: E \to B$ between topological spaces  satisfying the homotopy lifting property. This property states that for all spaces $X$ and maps $f: X \to E$ there exist commutative diagrams
\begin{displaymath}
\xymatrix{
X \ar[r]^{f} \ar[d]_{x \mapsto (x, 0)} & E \ar[d]^{p} \\
X \times I \ar[r]^{F} & B }
\end{displaymath}
with $F(x, 0) = f(x)$ such that
\begin{displaymath}
\xymatrix{
X \ar[r]^{f} \ar[d]_{x \mapsto (x, 0)} & E \ar[d]^{p} \\
X \times I \ar[ur]^{H} \ar[r]^{F} & B }
\end{displaymath}
commutes. That is, there exists a map $H: X \times I \to E$ such that $H(x, 0) = f(x)$ and $p(H(x, t)) = F(x, t).$
\end{definition}

For the fibrations that we will consider, the base space $B$ will be a path connected $CW$-complex and for any point $b_0 \in B$, the fibre $F= p^{-1}(b_0)$ has the homotopy type of a finite $CW$-complex. Furthermore all the fibres have the same homotopy type.

In order to consider the most general possible setting for the construction of our chain complex algebraic model of a fibration we will need to use the definition of $\Gamma$-fibration from \cite{Transfer-K}.

\begin{definition} (\hspace{-1pt}\cite[Definition 1.1]{Transfer-K})
Let $\Gamma$ be a discrete group. A $\Gamma$-fibration is a $\Gamma$-equivariant map $p' : E' \to B'$ with $E'$ a $\Gamma$-space  such that $\Gamma$ acts trivially on $B'$ and $p'$ has the $\Gamma$-equivariant lifting property for any $\Gamma$-space $X$
\begin{displaymath}
\xymatrix{X \ar[r]^{f} \ar[d]_{x \mapsto (x, 0)} & E' \ar[d]^{p'} \\
X \times I \ar[ur]^{H'} \ar[r]^{F} & B'.}
\end{displaymath}

\end{definition}

We will consider a fibration $p: E \to B$ with fibre $F = p^{-1}(b)$. $q: \widetilde{E}\to E$ will denote the universal cover of $E$. The composition $p \circ q = \widetilde{p} : \widetilde{E} \to B$ is a $\Gamma$-fibration with $\Gamma= \pe$. This fibration has fibre $\widetilde{p}^{-1}(b)=\widetilde{F}$, which is the cover of $F$ induced from the universal cover of $E$.

\begin{displaymath}
\xymatrix{\widetilde{F} \ar[r] \ar[d] & \widetilde{E} \ar[d]^{q} \ar[dr]^{\widetilde{p}= p \circ q} \\
F \ar[r] & E \ar[r]^p & B.}
\end{displaymath}

\section{Filtrations}

The argument to construct an algebraic model of the total space appropriate for the computation of the signature is motivated by the well known result that the total space of a fibration is filtered. The following Theorem \ref{E and X} was proved in \cite{Stasheff} using an inductive argument and by \cite{Schoen} using a different argument based on the $CW$ approximation theorem of Whitehead \cite{Whitehead}, \cite[p. 412]{Spanier}.
A similar argument to that of \cite{Stasheff} was used in \cite{Transfer-K}, \cite{modfour} and \cite{Korzen}.

\begin{theorem} \label{E and X} (\hspace{-1pt}\cite{Stasheff, Schoen})
Let $F \to E \to B$ be a Hurewicz fibration where $B$ and $F$ have the homotopy type of $CW$ complexes, then the total space $E$ is weakly homotopically equivalent to a $CW$ complex $X$.
\end{theorem}

\subsection{Filtered spaces}
Here we shall define what is meant by filtered spaces. We only consider compactly generated spaces, such as CW complexes.

\begin{definition}\label{522}
A $k$-filtered topological space $X$ is a topological space which is equipped with a series of subspaces
$$X_{-1}= \emptyset \subset X_0 \subset X_1 \subset \dots \subset X_k=X.$$
In the context of this thesis we assume that each of the inclusions $X_j \subset X_{j+1}$ is a \textit{cofibration}.
\end{definition}
The condition that  $X_j \subset X_{j+1}$ is a  cofibration implies that the pair $(X_{j+1}, X_j)$ has the \textit{homotopy extension property.} (\hspace{-1pt}\cite[page 14]{Hatcher}). In particular, the inclusion of a CW subcomplex is a cofibration, and has the homotopy extension property.

The following gives a definition of a filtered map and a filtered homotopy equivalence.
\begin{definition} Let $X$ and $Y$ be two filtered spaces,
\begin{itemize}
\item A map $f:X \to Y$ is a filtered map if $f(X_j) \subset Y_j$.
\item A filtered homotopy between the maps $f :X \to Y$ and $g: X \to Y$ is a homotopy $H : X \times I \to Y$ such that $H(X_j \times I) \subset Y_j$.
\item A filtered homotopy equivalence  between two filtered maps $f : X \to Y$ and $h : Y \to X$ is such that $fh \simeq Id \simeq hf$.
\end{itemize}
\end{definition}

In the context established in this section, namely that each of the inclusions $X_j \subset X_{j+1}$ (or $Y_j \subset Y_{j+1}$) is a cofibration, the following lemma holds.
\begin{lemma} A filtered map
 $f: X \to Y$ is a filtered homotopy equivalence if and only if each $f_j: X_j \to Y_j$  is a homotopy equivalence of unfiltered spaces.
\end{lemma}


\begin{proof} See \cite[7.4.1]{Brown-book}.
\end{proof}

\begin{definition}
A $k$-filtered $CW$-complex $X$ is a $CW$-complex $X$ together with a series of subcomplexes
$$X_{-1}=\emptyset \subset X_0 \subset X_1 \subset \dots \subset X_k=X.$$
The cellular chain complex $C(X)$ is filtered with
$$F_jC(X) = C(X_j).$$
\end{definition}

A $k$-filtered $CW$ satisfies the conditions of \ref{522}.


The main application will be a filtered complex in the context of Theorem \ref{E and X}, a fibration $E \xrightarrow{p} B$ with $B$ a $CW$-complex. We will consider $B$ with a filtration given by its skeleta which induces a filtered structure on $E$ by defining $E_k := p^{-1}(B_k),$ where $B_k$ is the $k$-th skeleton of $B$. Note that here the inclusions $E_{k-1} \subset E_k$ are cofibrations.

\subsection{Filtered complexes}

Here $\bb{A}$ will denote an additive category.
\begin{definition}\label{FM} (\hspace{-1pt}\cite{modfour})
\begin{itemize}
\item[(i)] Let $M$ be an object in the additive category $\bb{A}$ and let $M$ have a direct sum decomposition,
$$M = M_0 \oplus M_1 \oplus \dots \oplus M_k,$$
so that M has a filtration of length $k$,
$$F_{-1}M=0 \subseteq F_0M \subseteq F_1M \subseteq \dots \subseteq F_kM = M$$
where
$$F_iM = M_0 \oplus M_1 \oplus \dots \oplus M_i.$$
A \textbf{$k$-filtered object $F_*M \in \bb{A}$} is the object $M \in \bb{A}$ together with the direct sum decomposition of $M$.

\item[(ii)]
Let $F_*M$ and $F_*N$ be two $k$-filtered objects in  the additive category $\bb{A}$.
A filtered morphism is given by
$$
f = \left( \begin{array}{ccccc} f_0 & f_1 & f_2 & \dots & f_k \\
                                              0  & f_0 & f_1 & \dots & f_{k-1} \\
                                              0 &    0  & f_0 & \dots & f_{k-2} \\
                                              \vdots & \vdots & \ddots & \vdots& \vdots \\
                                               0 & 0 & 0 & \dots & f_0 \end{array} \right) : M = \bigoplus^{k}_{s=0} M_s \to N= \bigoplus^{k}_{s=0} N_s.
$$
\end{itemize}

\end{definition}

In the context of chain complexes in the additive category $\bb{A}$, a $k$-filtered complex $F_*C$ is defined as follows.

\begin{definition}
Let $C :C_n \to \dots \to C_r \to C_{r-1} \to \dots  \to C_{0}$ be a chain complex, and let each $C_r$ be $k$-filtered, that is, $$C_r = C_{r, 0} \oplus C_{r, 1} \oplus \dots \oplus C_{r, s} \oplus \dots \oplus C_{r, k} = \bigoplus^k_{s=0} C_{r,s}.$$
Then a $k$-filtered complex $F_*C$ in $\bb{A}$ is a finite chain complex $C$ in $\bb{A}$ where each of the chain groups is $k$-filtered and each of the differentials $d: F_*C_r \xrightarrow{d} F_*C_{r-1}$ is a filtered morphism. The matrix components of $d$ are the maps $C_{r,s} \xrightarrow{d_j} C_{r-1, s-j}.$

\end{definition}

\begin{example}
The following diagram represents a $k$-filtered $n$-dimensional chain complex $F_*C$ in $\bb{A}$.

\begin{displaymath}
\xymatrix{ C_n \ar[r]  & \dots  \ar[r]  & C_r   \ar[r]  &C_{r-1}  \ar[r] &\dots \ar[r] & C_0  \\
\vspace{-10pt}
C_{n, k} \ar[r] \ar[d] & \dots  \ar[r] & C_{r, k}   \ar[r]\ar[d] &C_{r-1, k}  \ar[r] \ar[d] &\dots \ar[r]& C_{0, k} \ar[d] \\
\vdots  \ar[d] &  &\vdots   \ar[d] &\vdots  \ar[d] & & \vdots \ar[d] \\
C_{n, s} \ar[r] \ar[d]& \dots \ar[r] & C_{r, s} \ar[dr]^{d_1}   \ar[dddr]^{d_j} \ar[r]^{d_0} \ar[d] &C_{r-1, s}  \ar[r]\ar[d]&\dots \ar[r] & C_{0, s} \ar[d] \\
C_{n, s-1} \ar[r]\ar[d] & \dots  \ar[r]& C_{r, s-1}   \ar[r] \ar[d]&C_{r-1, s-1}  \ar[r] \ar[d]&\dots \ar[r] & C_{0, s-1} \ar[d] \\
\vdots  \ar[d] &  &\vdots   \ar[d] &\vdots  \ar[d] & & \vdots \ar[d] \\
C_{n, s-j} \ar[r]\ar[d] & \dots  \ar[r]& C_{r, s-j}   \ar[r] \ar[d]&C_{r-1, s-j}  \ar[r] \ar[d]&\dots \ar[r] & C_{0, s-j} \ar[d] \\
\vdots  \ar[d] &  &\vdots   \ar[d] &\vdots  \ar[d] & & \vdots \ar[d] \\
C_{n, 0} \ar[r] & \dots  \ar[r] & C_{r, 0}   \ar[r]&C_{r-1, 0}  \ar[r] &\dots \ar[r] & C_{0, 0} \\
}
\end{displaymath}
$F_sC_r = \sum\limits^s_{i=0}C_{r,i}$, and $F_sC_r/F_{s-1}C_r = C_{r,s}$. So $C_{r,s}$ represents the $s$-th filtration quotient of $C_r$.
\end{example}

\section{The associated complex of a filtered complex} \label{derived section}

\subsection{A complex in the derived category $\bb{D}(\bb{A})$}

Given an additive category $\bb{A}$ we write $\bb{D}(\bb{A})$ for the homotopy category of $\bb{A}$, which is the additive category of finite chain complexes in $\bb{A}$ and chain homotopy classes of chain maps with
$$\textnormal{Hom}_{\bb{D(A)}}(C, D) = H_0(\textnormal{Hom}_\bb{A} (C, D)).$$
(See \cite[Definition 1.5]{SurTransfer}).

 A $k$-filtered complex $F_*C $  in $\bb{A}$ has an associated chain complex in the derived category $\bbDA$. The associated complex of a $k$-filtered space is denoted by $G_*(C)$  and is $k$-dimensional.
\begin{displaymath}
G_*(C): G_k (C)\to \dots \to G_r(C) \to G_{r-1}(C) \to \dots  \to G_{0}(C).
\end{displaymath}

The \textbf{morphisms} in $G_*(C)$ are given by the (filtered) differentials
$$(-)^s d_1 : (G_kC)_s = C_{k+s, r}\to (G_kC)_{s-1}=C_{k+s-1, r-1}.$$

Each of the individual terms $G_r(C)$ is an \textbf{object} in $\bbDA$, hence a chain complex in $\bb{A}$,
$$G_r(C) : \dots \to (G_rC)_s \to (G_rC)_{s-1} \to (G_rC)_{s-2} \to \dots .$$
As a chain complex, $G_rC$ has differentials,
$$d_{G_r(C)} = d_0 : G_r(C)_s = C_{r+s, r} \to G_r(C)_{s-1}= C_{r+s-1, r}$$
such that
$$(d_{G_r(C)})^2 = d_0^2 =0 : G_r(C)_s= C_{r+s, r} \to G_r(C)_{s-2}=C_{r+s-2, r} $$
Note that the differentials of a filtered complex $d: C_k \to C_{k-1}$ are such that  $d^2=0: C_r \to C_{r-2}$. These differentials are upper triangular matrices. If we write $d_0$ for the diagonal, $d_1$ for the superdiagonal we obtain some relations:
$$d_0^2 = 0 : C_{r, s} \to C_{r-2, s}$$
$$d_0d_1 + d_1 d_0=0 : C_{r,s} \to C_{r-2, s-1}$$
$$(d_1)^2 + d_0 d_2 + d_2d_0 =0 : C_{r,s} \to C_{r-2, s-2} $$
up to sign.
These are the required relations for the objects in the associated complex to be in $\bbDA$ and for the differentials to be morphisms in $\bbDA$ with square $0$.

\begin{example}
The following diagram is an $k$-dimensional $G_*C$ complex in $\bb{D(A)}$,
\begin{displaymath}
\xymatrix{ G_kC : \ar[d] &\dots \ar[r] & (G_kC)_s \ar[d] \ar[r] & (G_kC)_{s-1} \ar[d] \ar[r] &(G_kC)_{s-2} \ar[d] \ar[r] &\dots \\
\vdots \ar[d] & & \vdots \ar[d] & \vdots \ar[d] & \vdots \ar[d]\\
G_rC : \ar[d] &\dots \ar[r] & (G_rC)_s \ar[d]^{(-)^sd_1} \ar[r]^{d_0} & (G_rC)_{s-1} \ar[d] \ar[r] &(G_rC)_{s-2}\ar[d] \ar[r] &\dots  \\
G_{r-1}C : \ar[d] &\dots  \ar[r] & (G_{r-1}C)_s \ar[d] \ar[r] & (G_{r-1}C)_{s-1} \ar[r] \ar[d]&(G_{r-1}C)_{s-2} \ar[r] \ar[d]&\dots \\
\vdots \ar[d] & & \vdots \ar[d] & \vdots \ar[d] & \vdots \ar[d]\\
G_0C : &\dots \ar[r] & (G_0C)_s \ar[r] & (G_0C)_{s-1} \ar[r] &(G_0C)_{s-2} \ar[r] &\dots
}
\end{displaymath}
\end{example}

\begin{example}
For a filtered $CW$-complex $X$,
$$G_kC(X) = S^{-k}C(X_k, X_{k-1}).$$
\end{example}

When defining the symmetric structure of a filtered complex, we shall need to understand the behaviour of tensor products of filtered complexes. This is carefully described in \cite[section 12.2]{modfour}.
\begin{definition}
Let $F_*C$ and $F_*D$ be filtered chain complexes over $\bb{A}(R)$ and $\bb{A}(S)$, where $R$ and $S$ are rings, then the  tensor product is a filtered complex,
$$F_k(C \otimes_{\bb{Z}} D) = \bigoplus_{i+j = k} F_iC \otimes_{\bb{Z}} F_jD.$$

\end{definition}


\subsection{Duality for a filtered complex and its associated complex in $\bbDA$}
The definition of a filtered dual chain complex in an additive category with involution $\bb{A}$ is given in section 12.6 of \cite{modfour}. Here we will review the ideas that will be necessary in the construction of the symmetric structure in the derived category $\bbDA$.

Namely we will need to define what the dual $\Fd_*C$ of a $k$-filtered chain complex $F_*C$ is.

\begin{definition} (\hspace{-1pt}\cite[Definition 12.21 (ii)]{modfour})
Let $F_*C$ be a $k$-filtered $n$-dimensional chain complex in $\bb{A}$.

\begin{itemize}
\item[(i)] The filtered dual $\Fd_*C$ of $F_*C$ is the $k$-filtered complex with modules $$(\Fd_*C)_{r, s} = C^*_{n-r, k-s},$$
where $0\leq r \leq n$ and $0 \leq s \leq k$.
\item[(ii)] The dual of the differential $C_r \xrightarrow{d} C_{r-1}$ is given by $d^{\textnormal{dual}}: (\Fd_*C)_r \to (\Fd_*C)_{r-1}$. This dual differential is also $k$-filtered,
$$C^*_{n-r, k-s} \xrightarrow{(-)^{r+s +j(n+r)} d^*_j}C^*_{n-(r-1), k-(s-j)}.$$
\end{itemize}
\end{definition}

\begin{definition} (See \cite[12.20]{modfour})
The associated complex $G_*(F^{\textnormal{dual}}_*C)$ is the $k$-filtered dual of $G_*C$.
\end{definition}

\section{The total space of a fibration is homotopy equivalent to a filtered complex} \label{E-filtered}
The aim of this section will be to prove that the total space of a fibration $F \to E \to B$ (with $F$ and $B$ $CW$-complexes) is homotopically equivalent to a filtered $CW$-complex $X$, (see Theorem \ref{3.15}).

The argument given here is similar to that given in \cite{Transfer-K}. This is a $\pi_1(E)$-equivariant  version of the inductive proof in \cite{Stasheff}.

Before going into the proof of Theorem \ref{3.15}, we need to recall some theory from \cite{Transfer-K} about the behaviour of the pull-back of a bundle over an attaching map on the base.

We consider  $(D, d)$ to be a pointed contractible space and $f: D \to B$ a map. Given a morphism $\xi : f(d) \to b$ in $\pi_1(B)$ i.e. a homotopy class relative to $\{ 0, 1 \}$ of paths from $b$ to $f(b)$ in $B$, we have the following definition:

\begin{definition}(\hspace{-1pt}\cite[Definition 7.3]{Transfer-K})
Let $h$ be any homotopy between the constant map $c : D \to \{ b\}$ and $f$ such that $h(d, i)$ represents $\xi(i)$. Define
$$ T_{\xi}: D \times \widetilde{F} \to f^*\widetilde{E},$$
so that we can write a map $$ T(f, \xi) : D \times \widetilde{F} \to \widetilde{E}$$
by composing the pullback map with $T_{\xi}$
\begin{displaymath}
\xymatrix{
D \times \widetilde{F} \ar@/^2pc/[rr]^{T(f, \xi)} \ar[r]^{T_{\xi}} & f^*\widetilde{E} \ar[r] \ar[d]& \widetilde{E} \ar[d]^p \\
&D \ar[r]^f & B.}
\end{displaymath}
\end{definition}

\begin{remark}
We will denote the $k$-cells in $B$ by an ordered set $I_k$ and the $k$-th skeleton of $B$ by $B_k$.
\end{remark}

\begin{lemma} \label{trivialization} (\hspace{-1pt}\cite[Corollary 7.5]{Transfer-K})
Let $D$ be a contractible space and the map $h: (D^k, S^{k-1}) \times I \to (B_k, B_{k-1})$ be a homotopy between two maps $f_1, f_2: D^k \to B$, $\xi_1, \xi_2$ paths from $f_1(d)$ and $f_2(d)$ to $b$ respectively and $\xi : f_1(d) \to f_2(d)$. The following diagram commutes up to homotopy,
\begin{displaymath}
 \xymatrix{ (D^k, S^{k-1}) \times \widetilde{F}  \ar[dr]^{T(f_1, \xi_1)} \ar[dd]_{\textnormal{ID} \times (\xi_1 * \xi_2^{-1} *\xi^{-1}) }& \\
& (\widetilde{E}_k, \widetilde{E}_{k-1}) \\
(D^k, S^{k-1}) \times \widetilde{F}  \ar[ur]_{T(f_2, \xi_2)} &}
\end{displaymath}

\end{lemma}
The map $T(f, \xi)$ provides a trivialization of the bundle along the path $\xi$. This trivialization is an important element in the proof of Theorem \ref{3.15}.

To follow the same notation as in \cite{Transfer-K}, we shall denote the attaching maps by $$(Q(j), q(j)) : (D^k, S^{k-1}) \to (B_k, B_{k-1}).$$

\begin{theorem}\label{3.15} (\hspace{-1pt}\cite{Stasheff, Schoen})
Let $F \to E \to B$ be a fibration such that $B$ and $F$ have the homotopy type of  $CW$ complexes. Then  $\wtE$ is $\pi_1(E)$-homotopically equivalent to a filtered $\pi_1(E)$-complex $\wtX$.

\end{theorem}

\begin{proof}

This proof is based on \cite[pages 114, 115]{Transfer-K}.
The proof is by induction. We shall denote by $\wtE_0$ the restriction of $\wtE$ to $B_0$. As $B_0$ is the zeroth skeleton of $B$, then $\wtE_0$ is a disjoint union of spaces homotopy equivalent to $\wtF$. If we set $\wtX_0$ to be a number of copies of $\wtF$ one copy for each $0$-cell in $B$, then there is a $\pi_1(E)$-homotopy equivalence between $\wtE_0$ and $\wtX_0$.

For the induction step, suppose we have a $\pi_1(E)$-homotopy equivalence between $\wtE_{k-1}$ and $\wtX_{k-1}$. Writing $\widetilde{p}$ for the map $\wtE \xrightarrow{\widetilde{p}} B$, we have that $\wtE_{k-1}:= \widetilde{p}^{-1}(B_{k-1})$ so that the filtration of $B$ by its skeleta induces a filtration on $\wtE_{k-1}$.  The attaching maps
$$(Q(j), q(j)) : (D^k, S^{k-1}) \to (B_k, B_{k-1}) $$
give rise to the following pullback diagrams

\begin{minipage}[c]{0.4\linewidth}
\vspace{-10pt}
\begin{displaymath}
 \xymatrix{ Q(j)^*\wtE_k \ar[r] \ar[d] & \wtE_k \ar[d] \\
D^k \ar[r]^{Q(j)} & B_{k}, }
\end{displaymath}

\end{minipage}
\begin{minipage}[c]{0.30\linewidth}

\begin{displaymath}
 \xymatrix{ q(j)^*\wtE_{k-1} \ar[r] \ar[d] & \wtE_{k-1} \ar[d] \\
S^{k-1} \ar[r]^{q(j)} & B_{k-1}. }
\end{displaymath}
\vspace{10pt}
\end{minipage}

These give rise to the following pushout diagram
\begin{align}\label{pushout}
\xymatrix{Q(j)^*\wtE_k \ar[r] & \wtE_k \\
q(j)^* \wtE_{k-1} \ar[r] \ar[u]& \wtE_{k-1}. \ar[u]
}
\end{align}

Using the trivialization maps from \ref{trivialization} the above diagram \ref{pushout} extends to
\begin{displaymath}
\xymatrix{\sum_{I_k} D^k \times \wtF \ar[r] &Q(j)^*\wtE_k \ar[r] & \wtE_k \\
\sum_{I_{k-1}} S^{k-1} \times \wtF \ar[u] \ar[r] &q(j)^* \wtE_{k-1} \ar[r] \ar[u]& \wtE_{k-1}. \ar[u]
}
\end{displaymath}
From the induction step we have established that there exists a  homotopy equivalence between $\wtE_{k-1}$ and $\wtX_{k-1}$, so the bottom row of the previous diagram extends to
 \begin{displaymath}
\xymatrix{\sum_{I_k} D^k \times \wtF  & & & \\
\sum_{I_{k-1}} S^{k-1} \times \wtF \ar[u] \ar[r] &q(j)^* \wtE_{k-1} \ar[r] & \wtE_{k-1} \ar[r] & \wtX_{k-1}.
}
\end{displaymath}

So $\wtX_k$ can be obtained as a pushout from this diagram,
 \begin{displaymath}
\xymatrix{\sum_{I_k} D^k \times \wtF \ar[r] & \wtX_{k} \\
\sum_{I_{k-1}} S^{k-1} \times \wtF \ar[u] \ar[r] &\wtX_{k-1}. \ar[u]
}
\end{displaymath}

Hence the attaching maps for $\wtX_k$ are given by the attaching from the pushout diagram $\sum_{I_k} D^k \times \wtF \to \wtX_{k}$ and by the attaching maps induced from those of $\wtX_{k-1}.$ So we deduce that $\wtX_{k}$ is a $\pi_1(E)$-filtered space.
To prove that there is a homotopy equivalence between $\wtE_k$ and $\wtX_k$ we reproduce the first diagram in \cite[page 115]{Transfer-K}.

 \begin{displaymath}
\xymatrix{\sum_{j \in I_k}Q(j)^*\wtE_{k-1} \ar[d] & \sum_{j \in I_k} q(j)^*\wtE_{k-1} \ar[r] \ar[l] \ar[d]& \wtE_{k-1} \ar[d] \\
\sum_{j \in I_k}D^k \times \wtF \ar[d] & \sum_{j \in I_k} S^{k-1} \times \wtF \ar[l] \ar[r] \ar[d]& \wtE_{k-1} \ar[d]^{f_{k-1}} \\
\sum_{j \in I_k}D^k \times \wtF \times I & \sum_{j \in I_k} S^{k-1} \times \wtF \times I \ar[l] \ar[r] & \wtX_{k-1}  \\
\sum_{j \in I_k}D^k \times \wtF \ar[u] & \sum_{j \in I_k} S^{k-1} \times \wtF \ar[u] \ar[l] \ar[r] & \wtX_{k-1}. \ar[u] \\
}
\end{displaymath}

We first note from this diagram that $\wtX_k$ is the pushout of the bottom row and that $\wtE_k$ is the pushout of the top row

\begin{minipage}[c]{0.5\linewidth}
\vspace{-10pt}
\begin{displaymath}
\xymatrix{\sum_{I_k} D^k \times \wtF \ar[r] & \wtX_{k} \\
\sum_{I_{k-1}} S^{k-1} \times \wtF \ar[u] \ar[r] &\wtX_{k-1}, \ar[u]
}
\end{displaymath}
\end{minipage}
\begin{minipage}[c]{0.20\linewidth}

\begin{displaymath}
\xymatrix{Q(j)^*\wtE_{k} \ar[r] & \wtE_k \\
q(j)^* \wtE_{k-1} \ar[r] \ar[u]& \wtE_{k-1}. \ar[u]
}
\end{displaymath}
\vspace{3pt}
\end{minipage}

The push-outs of the two central rows give filtered spaces. Hence by the argument in \cite[page 246]{Brown-book} and \cite[page 161]{Dieck} the diagram defines a $\pi_1(E)$-homotopy equivalence between $\wtE_k$ and $\wtX_k$, which completes the induction step.
Thus the result follows.
\end{proof}

\section{The symmetric construction for the derived symmetric structure}

We will now use the symmetric construction to define the derived symmetric structure $G_* \phi$.
We have shown in the previous section that we can find a $\pi_1(E)$-space $\wtX$ filtered homotopy equivalent to $\widetilde{E}$. So we can construct a filtered chain diagonal approximation
$$\Delta^{X} : C(\widetilde{X}) \to C(\widetilde{X})^t \otimes_{\bb{Z}[\pi_1(X)]} C(\widetilde{X}).$$

This chain diagonal approximation induces a map

$$G_*\Delta^{X} : G_*C(\widetilde{X}) \to G_* (C(\widetilde{X})^t \otimes_{\bb{Z}[\pi_1(X)]} C(\widetilde{X}) ).$$
Since there exists a chain equivalence,
$$\theta_{X,X} : G_* (C(\widetilde{X})^t \otimes_{\bb{Z}[\pi_1(X)]} C(\widetilde{X}) ) \to G_* (C(\widetilde{X}))^t \otimes_{\bb{Z}[\pi_1(X)]} G_*(C(\widetilde{X}) ), $$
we can compose this chain equivalence with $G_* \Delta^X$ to obtain a chain diagonal approximation in the derived category,

$$ \theta_{X,X} \circ G_*\Delta^{X} : G_*C(\widetilde{X}) \to  G_* (C(\widetilde{X}))^t \otimes_{\bb{Z}[\pi_1(X)]} G_*(C(\widetilde{X}) ).$$

The slant isomorphism allows us to identify
$$G_* (C(\widetilde{X}))^t \otimes_{\bb{Z}[\pi_1(X)]} G_*C(\wtX) = \textnormal{Hom}_{\bb{D}_m \bb{Z}} (G_*C(\wtX)^{-*}, G_*C(\wtX) ).$$

Let $W$ be the standard free $\bb{Z}[\bb{Z}_2]$ resolution of $\bb{Z}$
$$W: \dots \to \bb{Z}[\bb{Z}_2] \xrightarrow{1-T} \bb{Z}[\bb{Z}_2] \xrightarrow{1+T}   \bb{Z}[\bb{Z}_2]  \xrightarrow{1-T}   \bb{Z}[\bb{Z}_2] \to 0.$$
Using the action of $T \in \bb{Z}_2$ on $G_*C(\wtX)^t \otimes_{\bb{D}_m \bb{Z}} G_*C(\wtX)$ by the transposition involution we define the $\bb{Z}$ module chain complex
$$W^{\%} G_*C(\wtX) = \textnormal{Hom}_{\bb{Z}[\bb{Z}_2]}(W, G_*C(\wtX)^t \otimes_{\bb{D}_m \bb{Z}} G_*C(\wtX)).$$

\begin{definition}
An $n$-dimensional derived symmetric structure on the associated chain complex in $\bb{D}( \bb{Z})$ of a filtered complex $C(\wtX)$ in $\bb{D}(\bb{Z})$ is a cycle in
$(W^\% G_*C(\wtX))_n$
\end{definition}

\begin{definition}
The $n$-dimensional $Q$-group $Q^n(G_*C(\wtX))$ of a derived chain complex is the abelian group of equivalence classes of  derived $n$-dimensional symmetric structure on $G_*C(\wtX)$,
$$Q^n(G_*C(\wtX)) = H^n (\bb{Z}_2; G_*C(\wtX)^t \otimes_{\bb{D}_m \bb{Z}} G_*C(\wtX)) = H_n(W^\% G_*C(\wtX)).$$
\end{definition}

\section{Transfer functor associated to a fibration}\label{transfer-sec}

In a fibration $F \to E \xrightarrow{p} B$ a point in the base lifts to a copy of the fibre $F$.
The total space of a fibre bundle is determined by the homotopy action of $\Omega B$ on the fibre $U: \Omega B \to  \textnormal{Map}(F, F) $  with a homotopy equivalence $E \simeq EB \times_{\Omega B} F $

\vspace{25pt}
\begin{figure} [ht]
\centering
\includegraphics[scale=0.95, trim=10 100 0 100]{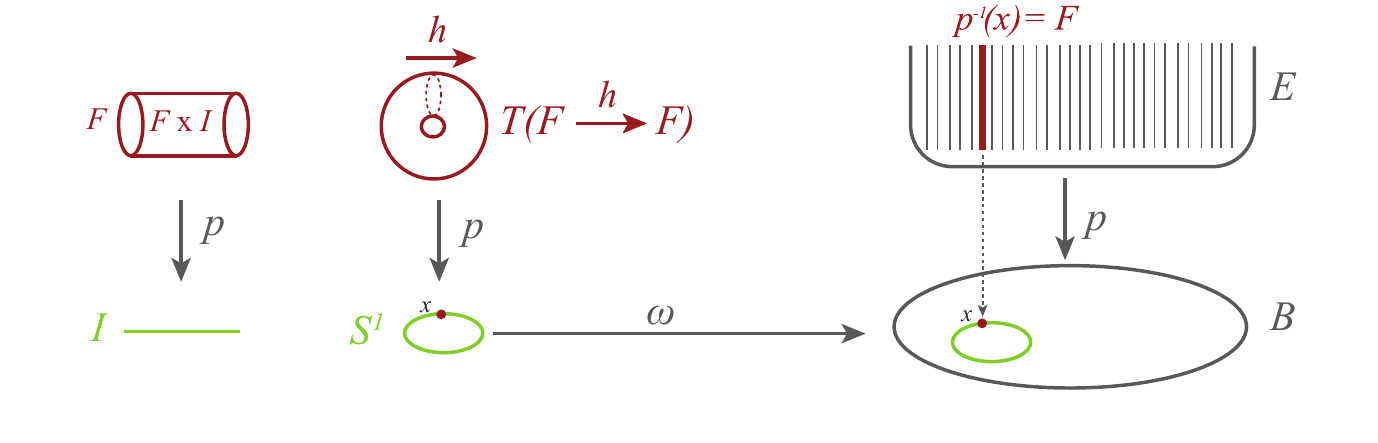}
\vspace{50pt}
\end{figure}
\vspace{1pt}
\vspace{30pt}
For any loop $\omega : S^1 \to B$ the pullback $F \to \omega^*E \to S^1 $ is the mapping torus $\omega^*E= T(h)$ of the monodromy automorphism $h= U(\omega): F \to F.$

 So we are considering the homotopy action of $\Omega B$ on the fibre $F$,
$$\Omega B \to \textnormal{Map}(F, F); \omega \mapsto h.$$
If $\omega$ is a loop in $\Omega B$, the homotopy class of its corresponding map $h_{\omega}: F \to F$ only depends on the homotopy class of $\omega \in \pi_1(B)$. So there is a homotopy action of the fundamental group $\pi_1(B)$ on the fibre $F$ given by the fibre transport
$$u: \pi_1(B) \to [F, F].$$
The homotopy action of $\Omega B$ on $F$ determines the chain homotopy action
$$ u : \pi_1(B) \to [C(F), C(F)]^{op}.$$
which extends to a ring morphism
\begin{equation}\label{ring-morph}
U : H_0(\Omega B) = \bb{Z}[\pi_1(B)] \to H_0(\textnormal{Hom}_{\bb{Z}}(C(F), C(F))).
\end{equation}

It is also possible to give a $\pi_1(E)$-equivariant version of the fibre transport by considering the $\pi_1(E)$-fibration  $\wtE \to B$, with fibres $\wtF$ the pullback to $F$ of the universal cover $\wtE$ of $E$.
In this case the ring morphism is given by
\begin{equation}\label{equiv-ring-morph}
U : H_0(\Omega B) = \bb{Z}[\pi_1(B)] \to H_0(\textnormal{Hom}_{\bb{Z}[\pi_1(E)]}(C(\wtF), C(\wtF))).
\end{equation}

The ring morphisms from \ref{ring-morph} and \ref{equiv-ring-morph} induce transfer functors as we now explain.

The idea of the transfer functor associated to a fibration  $F \to E \xrightarrow{p} B$ was developed in \cite{SurTransfer}.

\begin{definition} (\hspace{-1pt}\cite[Definition 1.1]{SurTransfer}) A representation $(A, U)$ of a ring $R$ in an additive category $\bb{A}$ is an object $A$ in $\bb{A}$ together with a morphism of rings $U : R \to \textnormal{Hom}_{\bb{A}}(A, A)^{op}.$
\end{definition}

Following the notation in \cite{SurTransfer} we will denote by $\bb{B}(R)$ the additive category of based finitely generated free $R$-modules, where $R$ is an associative ring with unity. As above, $\bb{D(A)}$ denotes the derived homotopy category of the additive category $\bb{A}$.
From \cite[Definition 1.5]{SurTransfer} we know that
$$\textnormal{Hom}_{\bb{D(A)}}(C, D) = H_0(\textnormal{Hom}_{\bb{A}}(C, D)).$$

\begin{definition}(\hspace{-1pt}\cite{SurTransfer})
A representation $(A, U)$ determines a transfer functor\\ $-\otimes (A, U) = F : \bb{B}(R) \to \bb{D(A)}$ as follows
$$
\begin{array}{l}F(R^n) = A^n, \\
F\left((a_{ij} : R^n \to R^m) \right) = \left(U(a_{ij})\right) : A^n \to A^m.
\end{array}
$$
\end{definition}

The ring morphism induced by the fibre transport,
$$U: \bb{Z}[\pi_1(B)] \to H_0(\textnormal{Hom}_{\bb{Z}[\pi_1(E)]} (C(\wtF), C(\wtF))^{op}.$$
determines the representation $(C(\wtF), U)$ of the ring $\bb{Z}[\pi_1(B)]$ into the derived category $\bb{D}_m(\bb{A}(\bb{Z}[\pi_1(E)])$,
$$-\otimes (C(\wtF), U) : \bb{B}(\bb{Z}[\pi_1(B)]) \to \bb{D}_m(\bb{A}(\bb{Z}[\pi_1(E)]) $$
where $\bb{D}_{m}(\bb{A}(\bb{Z}[\pi_1(E)])$ is the derived  category $\bb{D}(\bb{A}(\bb{Z}[\pi_1(E)])$ with the $m$-duality involution $C \mapsto C^{m-*}$.
This functor induced maps in the $L$-groups.
In \cite{SurTransfer} it is defined and used to construct a transfer map in the quadratic $L$-groups associated to a fibration $F^m \to E^{n+m} \to B^n$,
$$L_n(\bb{Z}[\pi_1(\wtB)]) \xrightarrow{-\otimes (C(\wtF), U)} L_{n}(\bb{D}_m(\bb{A}(\bb{Z}[\pi_1(E)])) \xrightarrow{\mu} L_{n+m}(\bb{Z}[\pi_1(E)]),  $$
where $\mu$ is the Morita map given in \cite{SurTransfer}.

In symmetric $L$-theory the transfer functor still induces a map,
$$L^n(\bb{Z}[\pi_1(\wtB)]) \xrightarrow{-\otimes (C(\wtF), U)} L^{n}(\bb{D}_m(\bb{A}(\bb{Z}[\pi_1(E)])),$$
but unlike in quadratic $L$-theory, in general in  symmetric $L$-theory a well-defined Morita map does not exist.

In our context, since we are only interested in obtaining information about the ordinary signature of $E$, we can forget about $\pi_1(E)$. That is, we only need to work with the map which we will make precise in the next section,
$$L^n(\bb{Z}[\pi_1(B)])\xrightarrow{-\otimes (C(F), \alpha, U)} L^{n}(\bb{D}_m\bb{Z})  \xrightarrow{\sigma_{\bb{D}_m {\bb{Z}}}} \bb{Z},$$
This is defined in \cite[Lemma 4.10]{Korzen}.

\section{The signature of a filtered complex and its associated complex in the derived category}

In section \ref{E-filtered} we have given the proof that the total space of a fibration $F^{2m} \to E \to B^{2n}$  is homotopically equivalent to a filtered $CW$-complex $X$ with the filtration induced from the cellular structure of the base space.
As $X$ is a filtered complex over $\bb{Z}$, it has an associated complex $(G_*C(X) , G_*\phi \in Q^{2n}(G_*C(X))$ in $\bb{D}_{2m}(\bb{Z})$, where $\bb{D}_{2m}(\zz)$ is the derived  category $\bb{D}(\zz)$ with the $2m$-duality involution $C \mapsto C^{2m-*}$.

\begin{proposition}\label{L-derived} (\hspace{-1pt}\cite[Lemma 4.10]{Korzen})
Let $m, n$ be such that $2m+2n=0 \pmod{4}$, the derived signature of the associated chain complex of a filtered complex induces a well defined map
$$\sigma_{\bb{D}_{2m}\bb{Z}} : L^{2n}(\bb{D}_{2m}(\bb{Z})) \to L^{2n+2m}(\zz)= \bb{Z}.$$

\end{proposition}
\begin{proof} The computation of  the signature $\sigma_{\bb{D}_{2m}\bb{Z}}(G_*C(X), G_*\phi)$ is done by a construction similar to that of the spectral sequence of a filtered complex. For details see \cite{Korzen}
\end{proof}

\begin{proposition}\label{sig-ass-complex} (\hspace{-1pt}\cite[Lemma 4.12]{Korzen})
Let $F^{2m} \to E \to B^{2n}$ be a fibration, and let $X$ be a filtered complex homotopy equivalent to $E$. The signature of $X$ is equal to the signature of its associated complex $(G_*C(X), G_*\phi)$ in the derived category $\bb{D}_{2m}(\bb{Z})$
$$\sigma(C(X), \phi) = \sigma_{\bb{D}_{2m}(\bb{Z})}(G_*C(X), G_*\phi).$$
\end{proposition}
\begin{proof}
The signature of a filtered complex only depends on the second page of its Serre spectral sequence. For details about the spectral sequence construction see \cite{Korzen}.

\end{proof}

For a fibration $F^{2m} \to E \xrightarrow{p} B^{2n}$, with the fibre a $2m$-dimensional geometric Poincar\'e complex, the derived signature of the associated chain complex of a filtered complex induces a well-defined derived signature map
$$\sigma_{\bb{D}_{2m}(\zz)} : L^{2n}(\bb{D}_{2m}(\zz)) \to L^{2m+2n}(\zz)=\zz.$$
and the composite
$$p^!: L^{2n}(\bb{Z}[\pi_1(B)]) \xrightarrow{-\otimes (C(F), \alpha, U)} L^{2n}(\bb{D}_m\bb{Z})  \xrightarrow{\sigma_{\bb{D}_m {\bb{Z}}}} L^{2m+2n}(\zz)= \bb{Z},$$
is a transfer map in symmetric $L$-theory. The functor $-\otimes (C(F), \alpha, U)$ was discussed in section \ref{transfer-sec} and will be reviewed again in section \ref{functors-sec}.

If $B$ is a $2n$-dimensional geometric Poincar\'e complex then $E$ is a $(2m+2n)$-dimensional geometric Poincar\'e complex with the transfer of the
symmetric signature $\sigma^*(B) \in L^{2n}(Z[\pi_1(B)])$ the ordinary signature
$\sigma(E) = p^!(\sigma^*(B)) \in L^{2m+2n}(\zz) = \zz.$



\begin{example} With this example we shall illustrate Proposition  \ref{L-derived} in two special cases, when the base is a point and when the fibre is a point.
\begin{itemize}
\item Case 1: Let $F^{2m} \to E \to \{pt\} $ be a fibration with base a point, that is, we take $n=0$. Then $F \to E$ is a homotopy equivalence, and there is a symmetric $L$-theory transfer
$$p^! : L^0(\zz) \to L^0(\bb{D}_{2m}(\zz)) \to L^{2m}(\zz)=\zz$$

In general the ring morphism $\zz[\pi_1(B)] \to \textnormal{Hom}_{\zz}(C(F), C(F))$ induces a map in $L$-theory $L^0(\zz[\pi_1(B)]) \to L^0 (\textnormal{Hom}_{\zz}(C(F), C(F)))$. Composing this with the canonical map  $L^0(\textnormal{Hom}_{\zz}(C( F ),C( F ))) \to L^0(\bb{D}_{2m}(\zz))$ we obtain the first functor in the transfer map.
If $B= \{ pt \}$, the ring morphism is $\zz \to \textnormal{Hom}_{\zz}(C( F ),C( F )); 1 \mapsto 1$, and the canonical map sends $1$ to $C(F)$ in the $0$-th filtration.
Therefore the transfer map is
$$p^!: L^0(\zz)=\zz \to L^{2m}(\zz) ; 1 \mapsto \sigma(E)=\sigma(F).$$
and
$$p^!(\sigma^*(B)) = \sigma(F)= \sigma( E). $$

\item Case 2:  Let $\{pt\} \to E \to B^{2n}$ be a fibration with fibre a point, that is we take $m=0$. Then $p:E \to B$ is a homotopy equivalence, and
$$p^!:L^{2n}(\zz[\pi_1(B)]) \to L^{2n}(\zz)=\zz$$
is the forgetful map induced by the augmentation $\zz[\pi_1(B)]\to \zz,$ and
$$p^!(\sigma^*( B)) = \sigma(B)= \sigma( E) $$


\end{itemize}
\end{example}

\section{An isomorphism of derived complexes}
At this point we know that for a fibration $F^m \to E \to B^n$,
 $$\sigma (E) = \sigma(X) =  \sigma_{\bb{D}({\zz})}(G_*C(X)) \in \zz,$$
 and that there is a well defined functor $$L^n(\bb{Z}[\pi_1(B)])\xrightarrow{-\otimes (C(F), \alpha, U)} L^{n}(\bb{D}_m\bb{Z})  \xrightarrow{\sigma_{\bb{D}_m {\bb{Z}}}} \bb{Z}.$$
To show that this functor describes the signature of the total space we will need that $\sigma_{\bb{D}({\zz})}(G_*C(X)) = \sigma_{\bb{D}({\zz})}(C(\wtB) \otimes (C(F), \alpha, U)). $


In Lemma \ref{trivialization} we saw that there exists a map
$$ (D^k, S^{k-1}) \times \wtF \to (\wtE_k, \wtE_{k-1}).$$
We compose this map with $(\wtE_k, \wtE_{k-1})  \xrightarrow{(f_k, f_{k-1})} (\wtX_k, \wtX_{k-1})$, and write a map of chain complexes,
$$C((D^k, S^{k-1}) \times \wtF) \to C(\wtX_{k}, \wtX_{k-1}).$$
The complex $C((D^k, S^{k-1}) \times \wtF)$ is canonically isomorphic to $C(\wtF)$ and we already showed in section \ref{derived section} that
$$S^{-k}C(\wtX_{k}, \wtX_{k-1})  \cong G_kC(\wtX).$$
So we take a $k$th-desuspension of the map $C((D^k, S^{k-1}) \times \wtF) \to C(\wtX_{k}, \wtX_{k-1}) $ and obtain,
$$ C(\wtF) \cong S^{-k}C((D^k, S^{k-1}) \times \wtF) \to S^{-k}C(\wtX_{k}, \wtX_{k-1}) \cong   G_kC(\wtX).$$

\begin{theorem}\label{chain-iso} (\hspace{-1pt}\cite[Theorem 4.5]{Korzen}) Given a fibration $p: E \to B$, let $X$ be a filtered complex homotopically equivalent to $E$. Then there is
\begin{itemize}
\item[(i)] a chain isomorphism of chain complexes in the derived category $\bb{D}(\bb{Z}[\pi_1(B)])$
$$\lambda : G_*C(X)  \cong C(\wtB) \otimes (C(F), U).$$
\item[(ii)] an equivalence of the derived symmetric structure defined on $C(\wtB) \otimes (C(F), U)$ and the derived symmetric structure on $G_*C(X).$
\end{itemize}
\end{theorem}

\begin{proof}

\begin{itemize}
\item[(i)]
Firstly we observe that there is a commutative diagram
$$
\xymatrix{ C(\wtF) \ar[r] \ar[d]_{Id} & G_kC(\wtX) \ar[d]^{G_kd} \\
C(\wtF) \ar[r] & G_{k-1}C(\wtX),
} $$
 and hence we can construct the  following diagram
$$
\xymatrix{ \bigoplus_{j \in I_k} C(\wtF)  \ar@/^2pc/[rr]^{Id} \ar[r] \ar[d]_{Id}&  G_kC(\wtX) \ar[dd]^{G_kd} \ar[r]^{\lambda_k} & C_k(\wtB) \otimes C(\wtF) \ar[dd]^{d_{C(\wtB) \otimes (C(\wtF), U)}} \\
\bigoplus_{j \in I_k}C(\wtF) \ar[dr] \ar[d] & & \\
\bigoplus_{j \in I_{k-1}}C(\wtF) \ar@/^-2pc/[rr]_{Id} \ar[r] & G_{k-1}C(\wtX) \ar[r]^{\lambda_{k-1}} & C(\wtB)_{k-1} \otimes C(\wtF). \\
}
$$
\raggedbottom

Here the square on the left
$$\xymatrix{\bigoplus_{j \in I_k} C(\wtF) \ar[r] \ar[d] & G_kC(\wtX) \ar[d]^{G_kd} \\
\bigoplus_{j \in I_k} C(\wtF)  \ar[r] & G_{k-1}C(\wtX)
} $$
commutes. The outer square also commutes
$$\xymatrix{\bigoplus_{j \in I_k} C(\wtF) \ar[r] \ar[d]_{Id} & C_k(\wtB) \otimes C(\wtF) \ar[d]^{d_{C(\wtB)\otimes (C(\wtF), U)}} \\
\bigoplus_{j \in I_k} C(\wtF)  \ar[r] &C(\wtB)_{k-1} \otimes C(\wtF)
} $$
with the above map being just the identity and the lower map being $d_{C(\wtB)\otimes (C(\wtF), U)}$.

Since the left and the outer squares commute, we deduce that the right square
$$
\xymatrix{  G_kC(\wtX) \ar[d]_{G_kd} \ar[r]^{\lambda_k} & C_k(\wtB) \otimes C(\wtF) \ar[d]^{d_{C(\wtB) \otimes (C(\wtF), U)}} \\
G_{k-1}C(\wtX) \ar[r]^{\lambda_{k-1}} & C(\wtB)_{k-1} \otimes C(\wtF) \\
}
$$
\raggedbottom
also commutes. Hence $\lambda: G_{k}C(\wtX) \to C(\wtB)_{k} \otimes (C(\wtF), U)$ is a chain map.

Using again the same argument now with the diagram
$$
\xymatrix{ \bigoplus_{j \in I_k} C(\wtF)  \ar[r] \ar[d]_{Id} & C_k(\wtB) \otimes C(\wtF)\ar[r]^{\epsilon_k} \ar[dd]^{d_{C(\wtB) \otimes (C(\wtF), U)}} &  G_kC(\wtX) \ar[dd]^{G_kd}  \\
\bigoplus_{j \in I_k}C(\wtF) \ar[dr] \ar[d] & & \\
\bigoplus_{j \in I_{k-1}}C(\wtF)  \ar[r] &  C(\wtB)_{k-1} \otimes C(\wtF) \ar[r]^{\hspace{15pt}\epsilon_{k-1}} & G_{k-1}C(\wtX) \\
}
$$
\raggedbottom
We can show that the square
$$
\xymatrix{  C_k(\wtB) \otimes C(\wtF) \ar[d]_{d_{C(\wtB) \otimes (C(\wtF), U)}} \ar[r]^{\epsilon_k} &  G_kC(\wtX) \ar[d]^{G_kd}\\
C(\wtB)_{k-1} \otimes C(\wtF) \ar[r]^{\hspace{15pt}\epsilon_{k-1}} & G_{k-1}C(\wtX) \\
}
$$
commutes, so that there is also a chain equivalence $$\epsilon:  G_{k}C(\wtX) \to C(\wtB)_{k} \otimes (C(\wtF), U).$$
Furthermore $\epsilon$ and $\lambda$ are chain homotopy equivalent. Since $G_kC(\wtX)$ and $C(\wtB) \otimes (C(\wtF), U)$ are homotopy equivalent chain complexes in the derived category $\bb{D}(\bb{Z}[\pi_1(B)])$ then this implies that they are isomorphic.

\item[(ii)]
We have already proved that there is a chain equivalence $\lambda: G_*C(X) \to C(\wtB) \otimes (C(F), U)$ and we also already described the filtered chain approximation $\Delta^X: C(X) \to C(\wtX) \otimes_{\bb{Z}[\pi_1(X)]} C(\wtX)$. The filtered chain approximation on $C(\wtB) \otimes (C(F), U)$ can be constructed in the same way.
{\small
$$\xymatrix{
G_*C(X) \ar[r]^{\lambda \otimes_{\bb{Z}[\pi_1(X)]} \bb{Z}} \ar[d]_{G_*C(\Delta^X)} & C(\wtB) \otimes (C(F), U) \ar[d]^{\Delta^{\wtB} \otimes \Delta^{\wtF}}\\
G_*(C(\wtX) \otimes_{\bb{Z}[\pi_1(X)]} C(\wtX)) \ar[d]_{\theta^{X, X}}& (C(\wtB) \otimes_{\bb{Z}[\pi_1(B)]} C(\wtB)) \otimes (C(\wtF) \otimes_{\bb{Z}[\pi_1(E)]} C(\wtF), U \otimes U) \ar[d]^{\theta^{\wtB, \wtF}} \\
G_*C(\wtX) \otimes_{\bb{Z}[\pi_1(X)]} G_*C(\wtX) \ar[r]^{\hspace{-50pt}\lambda \otimes \lambda} & (C(\wtB) \otimes (C(\wtF), U)) \otimes_{\bb{Z}[\pi_1(E)]}(C(\wtB) \otimes (C(\wtF), U)).
}
$$
}

Clearly this diagram commutes. Hence, there is an equivalence of both symmetric structures induced from the chain isomorphism $\lambda$.

\end{itemize}

\end{proof}

\begin{remark}
Using the results of Theorem \ref{3.15},  Proposition \ref{sig-ass-complex} and    Theorem \ref{chain-iso} we know at this point that
$$\sigma (E) = \sigma(X) =  \sigma_{\bb{D}({Z})}(G_*C(X)) = \sigma_{\bb{D}({Z})}(C(\wtB) \otimes (C(F), \alpha, U)) \in \bb{Z}.$$
\end{remark}

\section{Two equivalent functors for the signature of a fibration}\label{functors-sec}

We start by giving the definition of a $(\zz, m)$-representation.

\begin{definition} A $(\zz, m)$-symmetric representation $(A, \alpha, U)$ of a group ring $\zz[\pi]$ is a symmetric representation of $\zz[\pi]$ in the category $\mathbb{A}(\zz^{(m)})$, where $\zz^{(m)}$ is the ring $\zz$ with involution given by $a^*= (-1)^m a ~~ (a \in A).$
\end{definition}

Let $(C(F), \alpha, U)$ be the $(\bb{Z},m)$-symmetric representation of the group ring $\bb{Z}[\pi_1(B)]$ in $\bb{D}_{2m}(\bb{Z})$ associated to the fibration $F^{2m} \to E \to B^{2n}$.
\begin{definition}\label{functor-chain}
The representation $(C(F), \phi, U)$ gives rise to the following functor,
$$- \otimes (C(F), \alpha, U) : \bb{A}(\bb{Z}[\pi_1(B)]) \to \bb{D}_{2m}(\bb{Z}).$$
\end{definition}
 From this representation we can construct another $(\bb{Z}, m)$-representation $(A, \alpha, U)$ associated to the same fibration which is given by
$$
\begin{array}{l}
A = H_m(C(F))/ torsion, \\
\alpha: A=   H^m(C(F))/ torsion \to A^*= H_m(C(F))/ torsion, \\
U : \bb{Z}[\pi_1(B)] \to H_0(\textnormal{Hom}_{\bb{Z}}(A, A))^{op}.
\end{array}
$$

\begin{definition} \label{representation A}
The $(\bb{Z}, m)$-symmetric representation $(A, \alpha, U)$ associated to a fibration $F^{2m} \to E \to B^{2n}$ is a $(-1)^m$-symmetric form $(A, \alpha)$ together with the representation $U$ of the group ring $\bb{Z}[\pi_1(B)]$ in the additive category with involution $\bb{A}(\bb{Z}^{(m)})$.
\end{definition}

\begin{definition}\label{functor-homology}
The representation $(A, \alpha, U)$ gives rise to the following functor,
$$- \otimes (A, \alpha, U) : \bb{A}(\bb{Z}[\pi_1(B)]) \to \bb{A}(\bb{Z}^{(m)}) $$
where $\bb{Z}^{(m)}$ is the ring of integers with involution given by $e^* = (-1)^m e.$
\end{definition}

The two functors from Definitions \ref{functor-chain} and \ref{functor-homology} induce maps in symmetric $L$-theory,
$$- \otimes (C(F), \alpha, U) : L^{2n}(\bb{Z}[\pi_1(B)]) \to L^{2n}(\bb{D}_{2m}(\bb{Z})) $$
and
$$- \otimes (A, \alpha, U) : L^{2n}(\bb{Z}[\pi_1(B)]) \to L^{2n}(\bb{Z}).$$

\begin{remark} Note that $(A, \alpha, U) \in L^0(\bb{Z}, (-1)^{m})$, so the second functor does not immediately give us  a chain complex of the dimension of the total space. Nevertheless, since the $\epsilon$-symmetric $L$-groups over $\bb{Z}$ are $4$-periodic, we have an isomorphism given by skew suspension,
\begin{equation}\label{Skew-suspension}
\bar{S}^{m}: L^{2n}(\bb{Z}, (-1)^{n+m}) \xrightarrow \cong L^{2n+2m}(\bb{Z})
\end{equation}
sending the $2n$-dimensional $(-1)^{n+m}$-symmetric   complex $(C(\widetilde{B}), \phi) \otimes (A, \alpha, U)$ to the $2n+2m$-dimensional symmetric complex $\bar{S}(C(\widetilde{B}), \phi) \otimes (A, \alpha, U)$ so that  the signatures are defined,

$${\small
\xymatrix{
\txt{$L^{2n}(\bb{Z},(-1)^{n+m})$\\$(C(\widetilde{B}), \phi) \otimes (A, \alpha, U) $} \ar@<-.8ex>[dr] \ar@<.8ex>[dr]^{\sigma} \ar@{->}[rr]^{\cong}
&
&\txt{$L^{2n+2m}(\bb{Z},1)$\\$\bar{S}^m(C(\widetilde{B}), \phi) \otimes (A, \alpha, U) $}\ar@<-.8ex>[dl]_{\sigma} \ar@<.8ex>[dl]\\
&\txt { $\bb{Z}$\\ $\sigma((C(\widetilde{B}), \phi) \otimes (A, \alpha, U)) =$ \\ $\sigma(\bar{S}^{m}(C(\widetilde{B}), \phi) \otimes (A, \alpha, U))$}}}
$$

\end{remark}
\begin{proposition} \label{two-functors}(\hspace{-1pt}\cite[Lemma 4.11]{Korzen})
Using the functors for a fibration $F^{2m} \to E \to B^{2n}$ described in Definitions \ref{functor-chain} and \ref{functor-homology}, the following diagram commutes.
\begin{displaymath}
\xymatrix{
 L^{2n}(\bb{Z}[\pi_1(B)])  \ar[d]_{- \otimes (A, \alpha, U)} \ar[rr]^{- \otimes (C(F), \alpha, U)} &  & L^{2n}(\bb{D}_{2m}(\bb{Z})) \ar[dd]^{\sigma_{\bb{D}_{2m}(\bb{Z})}} \\
L^{2n}(\bb{Z}, (-)^{n+m}) \ar[d]_{\bar{S}^m} & \\
L^{2n+2m}(\bb{Z}) \ar[rr]^{\sigma} & & \bb{Z}
}
\end{displaymath}

\end{proposition}

\begin{proof}
$$
\begin{array}{ccl}
\sigma_{\bb{D}_{2m}\bb{Z}}(-\otimes (C, \alpha, U)) & = & \sigma_{\bb{D}_{2m}\bb{Z}}(-\otimes (H_*(C\otimes \bb{R}), \alpha \otimes R, U\otimes R)) \\
 & = & \sigma_{\bb{D}_{2m}\bb{Z}}(-\otimes (H_0(C\otimes \bb{R}), \alpha \otimes R, U\otimes R)) + \\
& & \hspace{20pt} \sigma_{\bb{D}_{2m}\bb{Z}}(-\otimes (H_1(C\otimes \bb{R}), \alpha \otimes R, U\otimes R)) + \\
& & \hspace{20pt} \dots + \sigma_{\bb{D}_{2m}\bb{Z}}(-\otimes (H_m(C\otimes \bb{R}), \alpha \otimes R, U\otimes R)) + \\
& & \hspace{20pt} \dots +\sigma_{\bb{D}_{2m}\bb{Z}}(-\otimes (H_{2m}(C\otimes \bb{R}), \alpha \otimes R, U\otimes R))
\end{array}
$$
The signature only depends on the middle homology so the only non-zero term is $\sigma_{\bb{D}_{2m}\bb{Z}}(-\otimes (H_m(C\otimes \bb{R}), \alpha \otimes R, U\otimes R))$ the other terms are just hyperbolic modules which are $0 \in L^{2m+2n}(\bb{Z})$. hence,
$$
\begin{array}{ccl}
\sigma_{\bb{D}_{2m}\bb{Z}}(-\otimes (C, \alpha, U)) & = & \sigma_{\bb{D}_{2m}\bb{Z}}(-\otimes (H_m(C\otimes \bb{R}), \alpha \otimes \bb{R}, U\otimes \bb{R})) \\
& = & \sigma (- \otimes (A \otimes \bb{R}, \alpha \otimes \bb{R}, U \otimes \bb{R})).
\end{array}
$$
\end{proof}

\begin{remark}
We had already noted that
$$\sigma(E) = \sigma_{\bb{D}_{2m}\bb{Z}} (C(\wtB), \phi) \otimes (C(F), \alpha, U). $$
Combining this result with Proposition \ref{two-functors} we have that
$$\sigma(E) = \sigma((C(\wtB), \phi) \otimes (A, \alpha, U)).$$
\end{remark}

\part{The signature of a fibration modulo $8$}

\chapter*{Introduction to part III}
\addcontentsline{toc}{chapter}{Introduction to Part III}

In this part of the thesis we shall combine ideas from the two previous parts to prove results concerning the signature modulo $8$ of a fibre bundle.
As was mentioned before, in  \cite{HirzebruchSerreChern} Chern, Hirzebruch and Serre proved that multiplicativity of the signature of a fibre bundle holds when the action of the fundamental group $\pi_1(B)$ is trivial on the cohomology ring of the fibres $H^*(F; \bb{Q})$,
$$\sigma(E) - \sigma(B)\sigma(F) =0 \in \bb{Z}.$$
Later on Kodaira, Hirzebruch and Atiyah constructed nonmultiplicative fibre bundles, with the action of $\pi_1(B)$ action on $H^*(F;\bb{Q})$ necessarily nontrivial.

The signature of a fibre bundle is multiplicative modulo 4, whatever the action. This was proved in \cite{Meyerpaper} for surface bundles and in \cite{modfour} for high dimensions. In Theorem \ref{4Arf-general-fibration} in chapter \ref{mod-eight-proof} we identify the obstruction to multiplicativity of the signature modulo $8$ of a fibration with a $\zz_2$-valued Arf invariant. When shall prove that if the action of $\pi_1(B)$ is trivial on $H^m(F, \bb{Z})/torsion \otimes \bb{Z}_4$, this Arf invariant takes value $0$. And we shall prove the following theorem,
\vspace{1pt}

\textbf{Theorem \ref{mod8-theorem}.} \textit{Let $F^{2m} \to E^{4k} \to B^{2n}$ be an oriented Poincar\'e duality fibration. If the action of $\pi_1(B)$ on $H^m(F, \bb{Z})/torsion \otimes \bb{Z}_4$ is trivial, then
$$\sigma(E)  -\sigma(F) \sigma(B) =0\in \zz_8.$$}
\vspace{-20pt}


Chapter \ref{examples-chapter} is devoted to investigating examples with non-trivial signature modulo $8$, that is example of fibre bundles which have signature equal to $4$. The first of these examples, will be a review of the construction by Endo of a surface bundle with signature $4$.  A Python module is used  to find some further nontrivial examples of  non-multiplicativity modulo $8$. The Python module used for these computations is included in the Appendix.

\chapter{Multiplicativity of the signature modulo $8$} \label{mod-eight-proof}

\section{Notation}
In this chapter we shall be using the notation presented in the following table. This notation has been introduced  mainly in chapter \ref{model}.

\begin{center}
  \begin{tabular}{ l  l }
 \\ \textbf{In Algebra:}  \\   \\ \hline
      $(C, \phi)$ & A $2n$-dimensional  $(-1)^n$-symmetric Poincar\'e \\  & complex over $\zz[\pi]$. (Section \ref{symm}) \\ \hline

  $(A, \alpha)$  & A nonsingular $(-1)^m$-symmetric form over $\zz$  \\ & with $\alpha: A \to A^*$ (Definition \ref{representation A})  \\ \hline
  $(A, \alpha, U)$ & A $(\zz, m)$-symmetric representation with \\ & $U : \bb{Z}[\pi] \to H_0( \textnormal{Hom}_{\zz}(A, A))^{op}$ \\ & (Definition \ref{functor-homology}) \\ \hline
$\bar{S}^m(A, \alpha, U)$ & Skew-suspension of $(A, \alpha, U)$ \\ &  (Equation  \ref{Skew-suspension}) \\ \hline
$(C, \phi) \otimes \bar{S}^m(A, \alpha, U) = (D, \Gamma)$ {\color{White}aaaaaaaaa}& A $4k=2m+2n$-dimensional symmetric  \\  & complex over $\zz$ with the action of $\pi$\\ &  given by $U$. (Proposition \ref{two-functors}). \\ \hline
  $(C, \phi) \otimes \bar{S}^m(A, \alpha, \epsilon) = (D', \Gamma')$  & The $4k$-dimensional symmetric complex  \\  & over $\zz$ with the action of $\pi$ given by \\ & the trivial action $\epsilon$.  (Proposition \ref{two-functors})  \\ \hline
  \end{tabular}
\end{center}

\begin{center}
  \begin{tabular}{ l  l }
 $C^{(2n)-r+s} \otimes  \bar{S}^mA^* \xrightarrow{U(\phi_s)  (\alpha)}C_r \otimes  \bar{S}^mA $ {\color{White}{aaaaaaa}}& The symmetric structure of a twisted product.  \\ & (Theorem \ref{chain-iso} and \cite{SurTransfer}) \\ \hline
  $\mc{P}_2$  & Pontryagin square depending on  \\ & the symmetric  structure  $\Gamma$ (Chapter \ref{Pontryagin-squares chapter}) \\ \hline
 $\mc{P}'_2$ & Pontryagin square depending on  \\ & the symmetric  structure $\Gamma'$ (Chapter \ref{Pontryagin-squares chapter}) \\ \hline
  \end{tabular}
\end{center}

\begin{center}
  \begin{tabular}{ l  l }
 \\ \textbf{In Topology:}  \\   \\ \hline
      $(C(\tilde{B}), \phi)$ & The $2m$-dimensional symmetric  complex \\ &  over $\zz[\pi_1(B)]$ of the universal cover \\  &of the base. (Section \ref{symm}) \\ \hline
  $A = H^{m}(F, \bb{Z})$    & The middle dimensional cohomology \\& of the fibre $F^{2m}$ (Definition \ref{representation A}) \\ \hline
  $(A, \alpha)$  & The nonsingular $(-1)^m$-symmetric form   \\ & with $\alpha= (-1)^m \alpha^*: A \to A^*$ \\ & (Definition \ref{representation A})  \\ \hline
  $(A, \alpha, U)$ & The symmetric representation with \\ & $U : \bb{Z}[\pi_1(B)] \to H_0( \textnormal{Hom}_{\zz}(A, A))^{op}$ \\ & (Definition \ref{functor-homology}) \\ \hline
$\bar{S}^m(A, \alpha, U)$ & $2m$-dimensional $(-1)^m$-symmetric complex \\& given by skew-suspension of $(A, \alpha, U)$ \\ &  (Equation  \ref{Skew-suspension}) \\ \hline
$(C(\tilde{B}), \phi) \otimes_{\zz[\pi_1(B)]} \bar{S}^m(A, \alpha, U)$ {\color{White}aaaaaaaaa}& The chain complex model for the \\  & total space with the action of $\pi_1(B)$\\ &  given by $U$. (Proposition \ref{two-functors})\\ \hline
$\sigma \left((C(\tilde{B}), \phi) \otimes_{\zz[\pi_1(B)]} \bar{S}^m(A, \alpha, U) \right)$  {\color{White}aaaaaaaaa}& The signature of the total space \\ $= \sigma(E) \in \zz$& is the signature of the chain complex model \\  & for the total space with the action of $\pi_1(B)$\\ &  given by $U$. (\cite[Theorem 4.9]{Korzen})\\ \hline

  $(C(\tilde{B}), \phi) \otimes_{\zz[\pi_1(B)]} \bar{S}(A, \alpha, \epsilon)$  & The chain complex model for the  \\  & total space with trivial  action of $\pi_1(B)$, \\ & which we denote by $\epsilon$.  (Proposition \ref{two-functors}) \\ \hline

$\sigma \left((C(\tilde{B}), \phi) \otimes_{\zz[\pi_1(B)]} \bar{S}^m(A, \alpha, \epsilon) \right)$  {\color{White}aaaaaaaaa}& The signature of the trivial product $B \times F$ \\ $=\sigma(C(\tilde{B}), \phi) \otimes_{\zz[\pi_1(B)]} \sigma(\bar{S}^m(A, \alpha, \epsilon)) $& is the signature of the chain complex model  \\  $= \sigma(B) \sigma(F) \in \zz$& for the total space with the action of $\pi_1(B)$\\ &  given by $\epsilon$. (\cite[Theorem 4.9]{Korzen})\\ \hline
  \end{tabular}
\end{center}

\begin{center}
  \begin{tabular}{ l  l }
 $C^{(2n)-r+s}(\tilde{B}) \otimes  \bar{S}A^* \xrightarrow{U(\phi_s)  (\alpha)}C_r(\tilde{B})\otimes  \bar{S}A $ & symmetric structure of a twisted product.  \\ & (Theorem \ref{chain-iso} and \cite{SurTransfer}) \\ \hline
  $\mc{P}_2$   &Pontryagin square of twisted product \\ &depending on the symmetric  structure \\ &  of  $(C(\tilde{B}), \phi) \otimes \bar{S}(A, \alpha, U)$ (Chapter \ref{Pontryagin-squares chapter}) \\ \hline

 $\mc{P}'_2$ &Pontryagin square of untwisted product \\ &depending on the symmetric  structure \\ &  of  $(C(\tilde{B}), \phi) \otimes \bar{S}(A, \alpha, \epsilon)$ (Chapter \ref{Pontryagin-squares chapter}) \\ \hline

  \end{tabular}
\end{center}

\section{Obstructions to multiplicativity modulo 8 of a fibration}

By the results in \cite{Meyerpaper} and \cite{modfour} the signature of a fibration
$F^{2m} \to E^{4k} \to B^{2n}$ of  geometric Poincar\'e complexes is multiplicative mod $4$
$$\sigma(E) - \sigma(B) \sigma(F) =0 \in \zz_4 .$$
If we set $M = E \sqcup - (B \times F)$, where $-$ reverses the orientation, then $M$ has signature
$$\sigma(M) = \sigma(E) - \sigma(B)\sigma(F) \in \zz,$$
so that $\sigma(M) = 0 \in \zz_4$, and Theorem \ref{4Arf-topology} can be applied to $M$.

In Theorem \ref{4Arf-topology} we have proved that when the signature is divisible by $4$, it is detected modulo $8$ by the Arf invariant. This can be applied in the situation of the signature of a fibration.
\begin{theorem}\label{4Arf-general-fibration}
Let $F^{2m} \to E^{4k} \to B^{2n}$ be a   Poincar\'e duality fibration. With $(V, \lambda)= \left(H^{2k}(E, \zz_2), \lambda \right)$ and $(V', \lambda')= \left(H^{2k}(B \times F), \zz_2), \lambda' \right)$, the signatures mod $8$ of the fibre, base and total space are related by
$$
\sigma(E) - \sigma(B \times F) = 4 \textnormal{Arf}\left( L^{\perp}/L , [\lambda \oplus - \lambda'], \frac{\left[ \mc{P}_2 \oplus-\mc{P}'_2 \right]}{2} \right) \in 4\zz_2 \subset \zz_8,
$$
where $L^{\perp} =\left\{(x, x') \in V \oplus V' \vert \lambda(x,x) = \lambda'(x', x') \in \zz_2 \right\}$ and $L = \langle v_{2k} \rangle \subset L^{\perp}$, with $v_{2k}=(v_{2k}(E), v_{2k}(B \times F)) \in V \oplus V'$ the Wu class of $E \sqcup -(B \times F).$
\end{theorem}

\begin{proof}
We first rewrite the signatures $\sigma(E)$ and $\sigma(F \times B)$ in terms of Brown-Kervaire invariants,  and use the additivity properties of the Brown-Kervaire invariant described in \cite[Proposition 2.1 (i)]{Morita} as follows
\begin{align*}
\sigma(E) - \sigma(B \times F) &=\textnormal{BK}(H^{2k}(E;\zz_2), \lambda, \mc{P}_2) - \textnormal{BK}(H^{2k}(B \times F;\zz_2), \lambda', \mc{P}'_2) \in \zz_8 \\
& = \textnormal{BK}\left(V \oplus V', \lambda \oplus -\lambda', \mc{P}_2 \oplus - \mc{P}'_2 \right) \in \zz_8,
\end{align*}

We know by \cite{Meyerpaper} and by \cite[Theorem A]{modfour} that
$$ \sigma(E) - \sigma(B \times F) = 0 \in \zz_4.$$
 Applying Theorem \ref{4Arf-Algebra}, $\textnormal{BK}\left(V \oplus V', \lambda \oplus -\lambda', \mc{P}_2 \oplus - \mc{P}'_2 \right) \in 4\zz_2 \subset \zz_8$ can be written as an Arf invariant,
$$4 \textnormal{Arf}\left( L^{\perp}/L , [\lambda \oplus - \lambda'],  \frac{\left[\mc{P}_2 \oplus-\mc{P}'_2\right] }{2} \right) \in \zz_8,$$
with $L^{\perp} =\left\{(x,x') \in V\oplus V' \vert \lambda (x, x)= \lambda'(x',x')  \in \zz_2 \right\}$ and the Wu sublagrangian $L = \langle (v_{2k}(E), v_{2k}(B \times F)) \rangle \subset L^{\perp}$, with $(v_{2k}(E), v_{2k}(B \times F))$ the  Wu class given by $(v_{2k}(E), v_{2k}(B \times F)) \in H^{2k}(E;\zz_2) \oplus H^{2k}(B \times F;\zz_2) = V \oplus V'.$
\end{proof}

In the following theorem we state the chain complex version of Theorem \ref{4Arf-general-fibration}.

\begin{theorem}\label{algebraic-version-4Arf-obstruction}
Let $(C, \phi)$ be the $2n$-dimensional $(-1)^n$-symmetric Poincar\'e complex, and let $(A, \alpha, U)$ be a $(\bb{Z}, m)$-symmetric representation. We shall write as before $(D, \Gamma) = (C, \phi) \otimes \bar{S}^m(A, \alpha, U)$ and $(D', \Gamma')=(C, \phi) \otimes \bar{S}^m(A, \alpha, \epsilon)$. Here $(D, \Gamma)$ and $(D', \Gamma')$ are $(2n+2m)$-dimensional symmetric complexes and $(2n+2m)$ is divisible by $4$. Then,
$$
\sigma(D, \Gamma) - \sigma(D', \Gamma') = 4 \textnormal{Arf}\left( L^{\perp}/L , [\Gamma_0 \oplus -\Gamma'_0], \frac{ \left[\mc{P}_2 \oplus-\mc{P}'_2\right] }{2} \right) \in 4\zz_2 \subset \zz_8,
$$
where $L^{\perp} =\left\{(x,x') \in H^{2k}(D;\zz_2) \oplus H^{2k}(D';\zz_2) \vert \Gamma_0(x, x) =\Gamma'_0(x', x') \in \zz_2 \right\}$ and the Wu sublagrangian $L = \langle (v_{2k}, v'_{2k}) \rangle \subset L^{\perp}$, with $(v_{2k}, v'_{2k})$ the algebraic Wu class $(v_{2k}, v'_{2k}) \in H^{2k}(D;\zz_2) \oplus H^{2k}(D';\zz_2).$
\end{theorem}

Note that when both $m$ and $n$ are odd in the fibration $F^{2m} \to E^{4k} \to B^{2n}$ then by \cite{Meyerpaper} and \cite{modfour} we have that
$$\sigma(E) = 0 \in \zz_4.$$
So  the general formula for the signature mod $8$ of a fibration given geometrically in
\ref{4Arf-general-fibration} and algebraically in \ref{algebraic-version-4Arf-obstruction} simplifies in the case of a fibration $F^{4i+2} \to E^{4k} \to B^{4j+2}$ to the expression in Proposition \ref{4Arf-odd-m-n-geometric}  geometrically or \ref{algebraic-version-4Arf-obstruction-4i+2} algebraically.

\begin{proposition}\label{4Arf-odd-m-n-geometric}
Let $F^{4i+2} \to E^{4k} \to B^{4j+2}$ be an oriented Poincar\'e duality fibration, then
$$\sigma(E) =\textnormal{BK}(H^{2k}(E;\zz_2), \lambda, \mc{P}_2) = 4 \textnormal{Arf}\left( L^{\perp}/L , [\lambda],  \frac{\left[\mc{P}_2\right]}{2}  \right) \in \zz_8,$$
where $L^{\perp} =\left\{x \in H^{2k}(E;\zz_2) \vert \lambda (x,x)= 0 \in \zz_2 \right\}$ and $L = \langle v_{2k} \rangle \subset L^{\perp}$, with $v_{2k}(E)$ the  Wu class $v_{2k}(E) \in H^{2k}(E;\zz_2).$
\end{proposition}

\begin{proof}
Here for dimension reasons $\sigma(F)$ and $\sigma(B)$ are both $0.$ Thus, by \cite[Theorem A]{modfour}, we know that the signature of $E$ is divisible by $4$.
So that we can write
\begin{align*}
\sigma(E) - \sigma(B \times F) &= \sigma(E)        =\textnormal{BK}(H^{2k}(E;\zz_2), \lambda, \mc{P}_2)  \in \zz_8 \\
\end{align*}
and since $\sigma(E)= \textnormal{BK}(H^{2k}(E;\zz_2), \lambda, \mc{P}_2) =0 \in \zz_4$, then the result follows as an application of Theorem \ref{4Arf-Algebra}.
\end{proof}

Algebraically Proposition \ref{4Arf-odd-m-n-geometric} is restated as follows,

\begin{proposition}\label{algebraic-version-4Arf-obstruction-4i+2}
Let $(C, \phi)$ be a $4i+2$-dimensional $(-1)$-symmetric Poincar\'e complex, and let $(A, \alpha, U)$ be a $(\bb{Z}, 2j+1)$-symmetric representation. We shall write $(D, \Gamma) = (C, \phi) \otimes \bar{S}^{2j+1}(A, \alpha, U)$ and $(D', \Gamma')=(C, \phi) \otimes \bar{S}^{2j+1}(A, \alpha, \epsilon)$,
then $\sigma(D', \Gamma')=0 \in \zz$, $\sigma(D, \Gamma)=0 \in \zz_4$ and
$$
\sigma(D, \Gamma) = 4 \textnormal{Arf}\left( L^{\perp}/L , [\Gamma_0],  \frac{\left[\mc{P}_2\right]}{2}  \right) \in 4\zz_2 \subset \zz_8,
$$
where $L^{\perp} =\left\{x \in H^{2k}(D;\zz_2) \vert \Gamma_0(x, x)= 0  \in \zz_2 \right\}$ and $L = \langle v_{2k} \rangle \subset L^{\perp}$, with $v_{2k}$ the algebraic Wu class $v_{2k} \in H^{2k}(D;\zz_2).$
\end{proposition}

\subsection{Relation to other expressions in the literature for the signature of a fibre bundle}
\subsubsection{The Arf invariant and the second Stiefel-Whitney class}
In \cite{Meyerpaper} Meyer studied the signature of a surface bundle $F^2 \to E^4 \to B^2$, where both $F$ and $B$ are orientable surfaces of genus $h$ and $g$ respectively. Meyer expressed the signature of the total space in terms of the first Chern class of the complex vector bundle $\beta: B \to BU(h)$ associated to the local coefficient system $\widetilde{B} \times_{\pi_1(B)} \bb{R}^{2h},$
$$\sigma(E) = 4 c_1(\beta) \in \zz.$$
So that,
\begin{equation}\label{Chern-Stiefel}
\frac{\sigma(E)}{4} = c_1(\beta) \in \zz,
\end{equation}

From \cite[problem 14-B]{Mil-Sta} the Chern classes of an $h$-dimensional complex vector bundle $\beta: B \to BU(h)$ are integral lifts of the Stiefel-Whitney classes of the underlying oriented $2h$-dimensional real vector bundle $\beta^{\bb{R}}: B \to BSO(2h)$. That is, the mod $2$ reduction of the first Chern class is the second Stiefel-Whitney class.
Hence the mod $2$ reduction of Equation \eqref{Chern-Stiefel} expresses the second Stiefel-Whitney class of the real vector bundle $\beta^{\bb{R}}: B \to BSO(2h)$ associated to the local coefficient system $\widetilde{B} \times_{\pi_1(B)} \bb{R}^{2h}$ as
\begin{equation}\label{sign-stiefel-surf}\frac{\sigma(E)}{4} = w_2(\beta^{\bb{R}}) \in \zz_2.\end{equation}

In Proposition \ref{4Arf-odd-m-n-geometric} we have shown that for a fibration $F^{4i+2} \to E^{4k} \to B^{4j+2}$, the signature mod $8$  can be expressed in terms of the Arf invariant.
In particular for a surface bundle,
\begin{equation}\label{sign-Arf-surf}\sigma(E)  = 4 \textnormal{Arf}\left( L^{\perp}/L , [\lambda],  \frac{\left[\mc{P}_2\right]}{2}  \right) \in \zz_8.\end{equation}

Combining both expressions in \eqref{sign-stiefel-surf} and \eqref{sign-Arf-surf}

$$\textnormal{Arf}\left( L^{\perp}/L , [\lambda],  \frac{\left[\mc{P}_2\right]}{2}  \right) = w_2(\beta^{\bb{R}}) \in \zz_2,$$
with $L^{\perp} =\left\{x \in H^{2k}(E;\zz_2) \vert \lambda(x,x)= 0 \in \zz_2 \right\}$ and $L = \langle v_{2k} \rangle \subset L^{\perp}.$

\subsubsection{The Arf invariant and the Todd genus}
In \cite{Atiyah-cov} Atiyah considers a $4$-manifold $E$ which arises as a complex algebraic surface with a holomorphic projection $\pi:E \to B$ for some complex structure on $B$. The fibres $F_b = \pi^{-1}(b)$ are algebraic curves but the complex structure varies with $b$.
Atiyah establishes that the Todd genera of $E$, $B$ and $F$ are related to the signature of $E$ by the equation,
$$ \frac{\sigma(E)}{4} = T(E)-T(B)T(F) \in \zz.$$
Following this equation the signature modulo $8$ of the total space is detected by the reduction mod $2$ of difference $T(E)-T(B)T(F)$, since $\sigma(E) = 4(T(E)-T(B)T(F)) \in \zz_8$ is equivalent to $ \frac{\sigma(E)}{4} = T(E)-T(B)T(F) \in \zz_2.$
Comparing this with Proposition \ref{4Arf-odd-m-n-geometric} we have that
$$\textnormal{Arf}\left( L^{\perp}/L , [\lambda],  \frac{\left[\mc{P}_2\right]}{2}  \right) = T(E)-T(B)T(F) \in \zz_2,$$
with $L^{\perp} =\left\{x \in H^{2k}(E;\zz_2) \vert \lambda(x,x)= 0 \in \zz_2 \right\}$ and $L = \langle v_{2k} \rangle \subset L^{\perp}.$

\section{Multiplicativity mod $8$ in the $\zz_4$-trivial case.}

In the previous section we have shown that in general there is an obstruction to multiplicativity modulo $8$ of a fibration  $F^{2m} \to E^{4k} \to B^{2n}$ given by the Arf invariant.

We shall now prove that by imposing the condition that  $\pi_1(B)$ acts trivially on $H^{m}(F, \bb{Z})/torsion \otimes \bb{Z}_4$, the obstruction disappears and the fibration has signature multiplicative modulo $8$. Geometrically the theorem is stated as follows.

\begin{theorem}\label{mod8-theorem} Let $F^{2m} \to E \to B^{2n}$ be an oriented Poincar\'e duality fibration. If the action of $\pi_1(B)$ on $H^m(F, \bb{Z})/torsion \otimes \bb{Z}_4$ is trivial, then
$$\sigma(E) - \sigma(F) \sigma(B) = 0\in \zz_8.$$
\end{theorem}

\begin{remark}
Clearly by the definition of the signature, when $m$ and $n$ are odd $$\sigma(F)\sigma(B)=0 \in \zz.$$
So in this case, the theorem establishes that
$$\sigma(E) \equiv 0 \in \zz_8.$$
\end{remark}

\subsection{Tools for proving Theorem \ref{mod8-theorem}}

We shall prove Theorem \ref{mod8-theorem} by proving its algebraic analogue, which we  state as Theorem \ref{mod8-algebraic}. In this section we prove some results that will be needed in the proof of the algebraic statement of the theorem (as in \ref{mod8-algebraic}).

The \textbf{algebraic analogue} of the condition that  $\pi_1(B)$ acts trivially on $H^{m}(F, \bb{Z})/torsion \otimes \bb{Z}_4$ is defined as follows:
\begin{definition}\label{Z4-triviality}
A $(\bb{Z}, m)$-symmetric representation $(A, \alpha, U)$ of a group ring $\bb{Z}[\pi]$ \textbf{is $\bb{Z}_4$-trivial} if
$$U(r) \otimes 1= \epsilon(r)\otimes1 : A \otimes \bb{Z}_4 \to A \otimes \bb{Z}_4$$
for all $r\in \bb{Z}[\pi],$ where $\epsilon$ denotes the trivial action homomorphism $$\epsilon: \zz[\pi_1(B)] \to H_0(\textnormal{Hom}_{\zz}(A, A))^{op}.$$
\end{definition}

A weaker condition is that of \textbf{$\zz_2$-triviality}:
\begin{definition}\label{Z2-triviality} (\hspace{-1pt}\cite[chapter 8]{Korzen})
A $(\bb{Z}, m)$-symmetric representation $(A, \alpha, U)$ of a group ring $\bb{Z}[\pi]$ \textbf{is $\bb{Z}_2$-trivial} if
$$U(r) \otimes 1= \epsilon(r)\otimes1 : A \otimes \bb{Z}_2 \to A \otimes \bb{Z}_2$$
for all $r\in \bb{Z}[\pi],$ where $\epsilon$ denotes the trivial action homomorphism as in the previous definition.
\end{definition}

\begin{remark} In particular any statement which holds under assumption of a $\zz_2$-trivial action is also true for an assumption of $\zz_4$-triviality. The converse may not be true.
\end{remark}

In the statement of the algebraic analogue of Theorem \ref{mod8-theorem}  we shall let $(C, \phi)$ be a $2n$-dimensional $(-1)^n$-symmetric Poincar\'e complex over $\zz[\pi]$ and $(A, \alpha, U)$ be a $(\zz, m)$-symmetric representation with $\zz_4$-trivial $U : \zz[\pi] \to H_0 (\textnormal{Hom}_{\zz}(A, A))^{op}$. When a result is true with a $\zz_2$-trivial action we shall indicate this accordingly.

With $4k =2m+2n$, we shall write $(D, \Gamma)$ for the $4k$-dimensional symmetric Poincar\'e complex over $\zz$ given by,
$$(D, \Gamma) = (C, \phi) \otimes_{\zz[\pi]} \bar{S}^m(A, \alpha, U),$$
and $(D', \Gamma')$ for the $4k$-dimensional symmetric Poincar\'e complex over $\zz$ given by
$$ (D', \Gamma') = (C, \phi) \otimes_{\zz[\pi]}\bar{S}^m(A, \alpha, \epsilon),$$
where $(A, \alpha, \epsilon)$ is the trivial representation.

\begin{lemma} \label{lemma-mod8-algebraic}
If the representation $(A, \alpha, U)$ is $\zz_2$-trivial then $$\zz_2 \otimes_{\zz}(D, \Gamma) \cong \zz_2 \otimes_{\zz}(D', \Gamma').$$ That is,
\begin{itemize}
\item[(i)] $H^{2k}(D, \zz_2) \cong H^{2k}(D', \zz_2).$
\item[(ii)] the symmetric structure reduced mod $2$ is the same for both symmetric complexes $(D, \Gamma)$ and $(D', \Gamma').$
\end{itemize}
\end{lemma}

\begin{proof} \label{proof-iso} (i) We will first show that if the representation $(A, \alpha, U)$ is $\zz_2$-trivial then:
$$H^{2k}(D, \zz_2) \cong H^{2k}(D', \zz_2).$$

From Definition \ref{functor-homology} we know that the  representation $(A, \alpha, U)$ gives rise to the following symmetric Poincar\'e complex  $(C, \phi) \otimes \bar{S}^m(A, \alpha, U) = (D, \Gamma)$ over $\bb{Z}$,
$$D =C \otimes \bar{S}^m(A, U) : \dots \to C_n \otimes  \bar{S}^m(A,U) \xrightarrow{d_{C \otimes \bar{S}^m(A, U)}}  C_{n-1}\otimes \bar{S}^m(A,U) \to \dots$$

$H^{2k}( D; \mathbb{Z}_2)$ has coefficients in $\bb{Z}_2$ so the chain groups of the chain complex are modules over $\bb{Z}_2$ and the differentials are the mod $2$ reduction of the differentials of the chain complex $D =C \otimes \bar{S}^m(A, U)$ over $\bb{Z}$,
$$D \otimes \zz_2 = C \otimes \bar{S}^m (A, U) \otimes \zz_2 : \dots \to C_n \otimes \bar{S}^m(A,U) \otimes \zz_2 \xrightarrow{d_{C \otimes \bar{S}^m(A, U)\otimes \zz_2}}  C_{n-1} \otimes \bar{S}^m(A,U)\otimes \zz_2 \to \dots$$

Using the definition of $\zz_2$-triviality there is an isomorphism of chain complexes
$$C \otimes \bar{S}^m (A, U) \otimes \zz_2 = C \otimes \bar{S}^m(A, \epsilon) \otimes \zz_2,$$
 since taking the mod $2$ reduction of the chain complex means that the action $U$ becomes the trivial action.
It follows that
$H^{2k}( D ; \mathbb{Z}_2)$ and $H^{2k}(D'; \mathbb{Z}_2)$ are isomorphic vector spaces over $\zz_2$.
 Note that this isomorphism only holds because the coefficients are in $\bb{Z}_2$, and we have assumed $\zz_2$-triviality. The integral cohomologies need not be always isomorphic. So in general with $\zz$ coefficients, we may have:
$$H^{2k}(D ; \mathbb{Z}) \neq H^{2k}( D'; \mathbb{Z}).$$
As the result holds with a $\zz_2$-trivial assumption, this means that it also holds if we assume $\zz_4$-triviality.

(ii) Here we will prove that if the representation $(A, \alpha, U)$ is $\zz_2$-trivial then
the symmetric structure reduced mod $2$ is the same for both symmetric complexes $(D, \Gamma) \otimes_{\zz}\zz_2$ and $(D', \Gamma')\otimes_{\zz}\zz_2$ :

Note that the chain complex  $(D, \Gamma)= (C(\tilde{B}), \phi) \otimes_{\zz[\pi]} \bar{S}^m(A, \alpha, U)$  is a symmetric Poincar\'e complex over $\bb{Z}$ with symmetric structure
$$C^{(2n)-r+s} \otimes \bar{S}A^* \xrightarrow{U(\phi_s)  (\alpha)} C_r \otimes \bar{S}A. $$

The  mod $2$ reduction of this chain complex is a symmetric Poincar\'e complex over $\bb{Z}_2$. Using the definition of $\bb{Z}_2$-triviality (\ref{Z4-triviality}), the symmetric structure of the mod $2$ reduction is given by
$$C^{(2n)-r+s} \otimes \bar{S}A^* \otimes \bb{Z}_2 \xrightarrow{U(\phi_s)  (\alpha) \otimes 1= \epsilon (\phi_s)  (\alpha) \otimes 1} C_r\otimes \bar{S}A \otimes \bb{Z}_2.$$

From this we see that the mod $2$ reductions of the symmetric structures in  the trivial and the $\bb{Z}_2$-trivial cases are equal.
Hence
$$ (D, \Gamma) \otimes_{\zz}\zz_2 \cong (D', \Gamma')\otimes_{\zz}\zz_2.$$
\end{proof}

From the proof of \ref{lemma-mod8-algebraic} (i) and \ref{lemma-mod8-algebraic}(ii) we know that for a $\bb{Z}_2$-trivial twisted product and for an untwisted product, the vector spaces given by the middle dimensional cohomology with $\bb{Z}_2$ coefficients  are isomorphic and the $\bb{Z}_2$-valued symmetric structure is also the same in both cases. Clearly the results in this lemma also hold when the action $U$ is $\zz_4$-trivial

\subsubsection{Comparing Pontryagin squares for the twisted and untwisted product:}
In chapter \ref{Pontryagin-squares chapter}
 we defined algebraic Pontryagin squares depending on the symmetric structure of a symmetric complex.
A $\zz_4$-valued quadratic function such as the Pontryagin square cannot be recovered uniquely from the associated $\bb{Z}_2$-valued bilinear pairing $\Gamma_0 : H^{2k}(D;\zz_2)\times H^{2k}(D;\zz_2) \to \bb{Z}_2$. It is crucial for the proof of Theorem \ref{mod8-theorem}  to note that the Pontryagin square depends on the definition of the symmetric structure on integral cochains.
This was already described in Definition \ref{Pont-square-definition} where the Pontryagin square was defined by
\begin{align*}\mc{P}_2(\Gamma): H^{2k}(D ; \bb{Z}_2) & \to \bb{Z}_4\\
                                                 x & \mapsto \Gamma_0(x, x) + \Gamma_1(x, dx).
\end{align*}

Depending on the integral symmetric structures, we have two different lifts of the mod $2$ symmetric structure on
$$V = H^{2k}( D; \mathbb{Z}_2) = H^{2k}(D'; \mathbb{Z}_2)$$
which give rise to two different Pontryagin squares.

In other words we are considering two chain complexes which are different over $\zz$, $(D, \Gamma)$ and $(D', \Gamma')$ but are chain equivalent when reduced over $\zz_2$,
$$(D, \Gamma) \otimes_{\zz}\zz_2 = (D', \Gamma')\otimes_{\zz}\zz_2.$$
In the situation of Theorem \ref{mod8-algebraic} we know from  the proof of \ref{mod8-algebraic} (i)(a) and (i)(b) that we have the same $\zz_2$-valued symmetric bilinear form for the twisted and untwisted products,
$$(H^{2k}(D;\zz_2),  \Gamma_0) = (H^{2k}(D'; \zz_2),  \Gamma'_0),$$
where
\begin{itemize}
\item $(H^{2k}(D;\zz_2), \Gamma_0) =  \left( H^{2k}( (C, \phi) \otimes_{\zz[\pi]} \bar{S}^m(A, \alpha, U); \mathbb{Z}_2), U(\phi)  (\alpha)  \right),$
\item $(H^{2k}(D'; \zz_2),  \Gamma'_0) =\left( H^{2k}( (C, \phi) \otimes_{\zz[\pi]} \bar{S}^m(A, \alpha, \epsilon); \mathbb{Z}_2),  \epsilon (\phi)  (\alpha) \right),$
\end{itemize}
but the integral symmetric structures $ \Gamma$ and $ \Gamma'$  are different in the twisted and untwisted products, so this gives rise in general to two different Pontryagin squares of the same $\bb{Z}_2$-valued symmetric bilinear form.

\begin{remark} In what follows  we will use the notation $\mc{P}_{2}$ for the twisted Pontryagin square and $\mc{P}'_2$ for the Pontryagin square on an untwisted product.
\end{remark}

The following Proposition \ref{differ-linear-map} applies precisely to the general situation of the two Pontryagin squares $\mc{P}_2$ and $\mc{P}'_2$ that we have just described above.

\begin{proposition} \label{differ-linear-map}
Let $V$ be a $\bb{Z}_2$-valued vector space and $\lambda : V \otimes V\to \bb{Z}_2$ a non-singular symmetric bilinear pairing. Any two quadratic  enhancements $q, q' : V \to \bb{Z}_4$ over $\lambda$ differ by a linear map,
$$q'(x)-q(x) = 2 \lambda(x, t) \in \zz_4$$
for some $t \in V$.
\end{proposition}

\begin{proof}
See \cite{Taylor}, \cite{Deloup}, \cite[page 10]{Mod8}.
\end{proof}

Consequently from this proposition we know that the two Pontryagin squares $\mc{P}_2$ and $\mc{P}'_2$ differ by linear map.
Furthermore the Brown-Kervaire invariants of two quadratic enhancements as in Proposition \ref{differ-linear-map} are related by the following theorem,
\begin{theorem} \cite[Theorem 1.20 (x)]{Brown}
Let $V$ be a $\bb{Z}_2$-valued vector space and $\lambda : V \otimes V\to \bb{Z}_2$ a non-singular symmetric bilinear pairing, then any two quadratic  enhancements $q, q' : V \to \bb{Z}_4$ over $\lambda$ differ by a linear map, $q'(x)-q(x) = 2 \lambda(x, t) \in \zz_4$ and
$$\textnormal{BK}(V,\lambda,q)-\textnormal{BK}(V, \lambda,q')=2q(t) \in \zz_8$$
for some $t \in V$.
\end{theorem}
Note that when $\textnormal{BK}(V,\lambda,q)$ and $\textnormal{BK}(V,\lambda,q')$ are divisible by $4$ we can write this relation in  terms of the Arf invariant,
$$\textnormal{Arf}(W, \mu ,h)-\textnormal{Arf}(W, \mu, h')=h(t) \in \zz_2.$$

We will now show that by setting the condition of a $\zz_4$-trivial action $U$ there is an isomorphism between the untwisted Pontryagin square and the $\zz_4$-twisted Pontryagin square, in which case $h(t)=0.$

\subsubsection{Pontryagin squares in terms of the action $U$}

For the proof of \ref{mod8-algebraic} we will need to express the Pontryagin square of the symmetric Poincar\'e complex $(D, \Gamma)$ over $\zz$ in terms of the equivariant Pontryagin square of the $(-1)^n$-symmetric Poincar\'e complex $(C, \phi)$ over $\zz[\pi]$  and symmetric structure of $(A, \alpha)$. For this we use \cite[Lemma 8.8]{Korzen} and Definition \ref{equivariant-Pont}, which gives the expression of the equivariant Pontryagin square
\begin{align*}
\mc{P}_{d_R} : H^{n}(C; \bb{Z}_2)\to Q^{2n}(B(k, d_R))= R/(I^2+2I) ;\\
(u, v) \mapsto (u, v)^\%(\phi) = \phi_0(v, v) + d_R \phi_1(v, u)
\end{align*}
with $(v, u) \in C^n \oplus C^{n+1}$ such that $d_C^*(v) = d_Ru \in C^{n+1}$, $d_{C}^*(u)=0 \in C^{n+2}$ and where  $R = \bb{Z}[\pi]$ and $I$ is the ideal $I = \textnormal{Ker}(\bb{Z}[\pi] \to \bb{Z}_2)$.

\begin{lemma} \label{Pontryagin-total-space} (\hspace{-1pt}\cite[Lemma 8.8]{Korzen}) Let $(D, \Gamma)=  (C, \phi) \otimes \bar{S}^m(A, \alpha, U)$ be a $4k$-dimensional symmetric Poincar\'e complex over $\zz$, with $U$ the action of  $\zz[\pi]$ on $(A, \alpha)$. Then the Pontryagin square on $(D, \Gamma)$ can be expressed as
        $$\mc{P}_2 (x)=\left( U\left( \mc{P}_{d_R}  \right)  \alpha \right) (x) \in \zz_4,$$
where $x \in H^{2k}(D; \zz_2)$ and $\mc{P}_{d_R}$ is the equivariant Pontryagin square on the symmetric Poincar\'e complex $(C, \phi)$ over $\zz[\pi]$.
\end{lemma}

\begin{lemma}\label{Z4-pont} If the action $U$ is $\zz_4$-trivial then the $\zz_4$-twisted Pontryagin square
$$\mc{P}_2 (x)=\left( U\left( \mc{P}_{d_R}  \right)  \alpha \right) (x) \in \zz_4$$
is isomorphic to the untwisted Pontryagin square
$$\mc{P}'_2 (x)=\left( \epsilon \left( \mc{P}_{d_R}  \right)  \alpha \right) (x) \in \zz_4.$$
\end{lemma}
\begin{proof}
An element $x \in H^{2k}(D; \zz_2) = H^{2k}(D'; \zz_2)$ is described in general as a linear combination of a product, $x = \sum \limits_i c_i \otimes a_i$, so that
\begin{align*}
\mc{P}_2\left(\sum \limits_i c_i \otimes a_i\right) & = \sum \limits_i \mc{P}_2(c_i \otimes a_i) + 2 \sum \limits_{i < j} \Gamma_0 (a_i \otimes c_i, a_j \otimes  c_j).
\end{align*}
As $\Gamma_0$ is isomorphic to $\Gamma'_0$, then the difference between $\mc{P}_2$ and $\mc{P}'_2$ is given by the difference between
$$\sum \limits_i \left( U\left( \mc{P}_{d_R}  \right)  \alpha \right) (c_i \otimes a_i) \in \zz_4,$$
and
$$\sum \limits_i \left( \epsilon \left( \mc{P}_{d_R}  \right)  \alpha \right) (c_i \otimes a_i) \in \zz_4. $$
As the action $U$ is trivial over $\zz_4$, then the terms in these two sums are equal, so that the two Pontryagin squares are isomorphic.
\end{proof}

\subsection{The algebraic analogue of \ref{mod8-theorem}}

We can now state and prove the algebraic analogue of \ref{mod8-theorem}. The proof of the algebraic analogue implies the proof of the geometric statement.
\begin{theorem} \label{mod8-algebraic} Let $(D, \Gamma)=  (C, \phi) \otimes \bar{S}^m(A, \alpha, U)$ be a $4k$-dimensional symmetric Poincar\'e complex over $\zz$, with $U$ the action of the $\zz[\pi]$ on $(A, \alpha)$, if the representation $(A, \alpha, U)$ is $\zz_4$-trivial then:
$$\sigma(D, \Gamma) - \sigma(D', \Gamma') = 0 \in \zz_8,$$
where $(D', \Gamma')=  (C, \phi) \otimes \bar{S}^m(A, \alpha, \epsilon)$ is the trivial product.
\end{theorem}

\begin{proof}
The signatures modulo $8$ of both $(D, \Gamma)$ and $(D', \Gamma')$ are given by Morita's theorem \ref{Morita-theorem-cx} to be
\begin{itemize}
\item $\sigma(D, \Gamma) \equiv \textnormal{BK}(H^{2k}(D; \zz_2), \Gamma_0, \mc{P}_2) \pmod{8},$
\item $\sigma(D', \Gamma') \equiv \textnormal{BK}(H^{2k}(D'; \zz_2), \Gamma_0, \mc{P}'_2) \pmod{8}.$
\end{itemize}

From Lemma \ref{lemma-mod8-algebraic} we know that
$$(H^{2k}(D; \zz_2), \Gamma_0) \cong (H^{2k}(D'; \zz_2), \Gamma'_0)$$
and from Lemma \ref{Z4-pont} we know that the two Pontryagin squares $\mc{P}_2$ and $\mc{P}'_2$ are isomorphic, hence
\begin{align*}
\sigma(D, \Gamma) & = \textnormal{BK}(H^{2k}(D; \zz_2), \Gamma_0, \mc{P}_2) \\
                            & =  \textnormal{BK}(H^{2k}(D'; \zz_2), \Gamma'_0, \mc{P}'_2) \\
                             & = \sigma(D', \Gamma') \in \zz_8.
\end{align*}
so the result follows.
\end{proof}

The geometric proof of Theorem \ref{mod8-theorem}, which states that $\sigma(E)- \sigma(B \times F) = 0 \in \zz_8$ for a fibration $F^{2m}\to E^{4k} \to B^{2n}$ with trivial action of $\pi_1(B)$ on $H^m(F, \zz)/torsion \otimes \zz_4$  is a consequence of the algebraic proof:

For the total space of a Poincar\'e fibration
$F^{2m} \to E^{4k} \to B^{2n}$ with $\zz_4$-trivial action  the
$\zz_4$-enhanced symmetric forms over $\zz_2$
$$(H^{2k}(E;\zz_2),\lambda_E,q_E)~,~
(H^{2n}(B;\zz_2)\otimes_{\zz_2}H^{2m}(F;\zz_2),\lambda_B\otimes \lambda_F,q_B \otimes q_F)$$
of the symmetric Poincar\'e complexes over $\zz_4$
$$(C(E;\zz_4),\phi_E)~,~(C(B;\zz_4),\phi_B) \otimes (C(F;\zz_4),\phi_F)$$
are Witt equivalent, so that
\begin{align*}
\sigma(E)&= \textnormal{BK}(H^{2k}(E;\zz_2),\lambda_E,q_E) \\[1ex]
&=\textnormal{BK}(H^n(B;\zz_2),\lambda_B,q_B)\textnormal{BK}(H^m(F;\zz_2),\lambda_F,q_F)\\[1ex]
&=\sigma(B) \sigma(F) \in \zz_8.
\end{align*}

\subsubsection{Multiplicativity modulo $8$ with a $\zz_2$-trivial action}

In \cite{KT} S. Klaus and P. Teichner conjectured that for an oriented Poincar\'e fibration  $F^{2m} \to E^{4k} \to B^{2n}$ with trivial action of $\pi_1(B)$ on $H^{m}(F; \zz_2)$,
$$\sigma(E)- \sigma(F)\sigma(B) =0 \in \zz_8.$$

Their attempt at a proof involved spectral sequences in the style of \cite{HirzebruchSerreChern}.

 In \cite{Korzen} there is another attempt of proof of this conjecture for the case when both fibre and base have dimensions multiples of $4$, i.e for a fibration $F^{4m} \to E^{4n+ 4m} \to B^{4n}$. This attempt in \cite{Korzen} already used the algebraic model of the total space that we have presented in chapter \ref{model}, so that the algebraic version of the theorem is similar to the one that we are using here.
The argument in \cite{Korzen} was to reduce a general fibration with $\zz_2$-trivial action to the direct sum of a double covers.
The double cover of a $4n$-dimensional complex is known to have multiplicative signature modulo $8$ (\cite{visible}, \cite{bluebook}).
Unfortunately the proof in \cite{Korzen} has the following gap:
On \cite[page 98]{Korzen} there is an assumption that there is always a possible choice of basis $\left\{a_i \right\}$ for $A = H^{2m}(F; \zz) / torsion$ such that the symmetric form $\alpha: A \to A^*$ is diagonal. In general it is not always possible to diagonalize a symmetric form over $\zz$, so this assumption (which is crucial for the rest of the proof) is too strong.


For this conjecture to be true with a $\zz_2$-trivial action, we would need to prove that there exists an isomorphism of the untwisted and the twisted Pontryagin squares with $U$ a $\zz_2$-trivial action. At the moment it is only clear that two such Pontryagin squares differ by a linear map as explained in Proposition \ref{differ-linear-map}. However there is no problem if the action is $\zz_4$-trivial, as shown in Theorem \ref{mod8-theorem}.


\chapter{Examples of surface bundles with non-trivial signature}\label{examples-chapter}

\begin{flushright}

\textit{``The abstractions and the examples have to go hand in hand."}
\end{flushright}

\begin{flushright}
\textit{Sir Michael Atiyah}
\end{flushright}

\vspace{10pt}

\section{Review of the example in \cite{Endo}}

 Atiyah \cite{Atiyah-cov}, Kodaira \cite{Kodaira} and Hirzebruch \cite{Ramified-Hirz} constructed surface bundles with non-trivial action of the fundamental group which have signature divisible by $8$. That these constructions have signature divisible by $8$ is clarified by Hirzebruch in \cite[page 264]{Ramified-Hirz}. Lefschetz fibrations also provide examples of fibrations with non-zero signature. Nevertheless in the case of a Lefschetz fibration, the signature does not depend solely on the action of the fundamental group, but also on the existence of singular fibres.  The example in \cite{Endo} is a construction of a surface bundle with signature $4$ depending only on the non-trivial action of $\pi_1(B)$. In \cite{Endo}, Endo gives the computation of a surface bundle with fibre an orientable surface of genus $3$ and base an orientable surface of genus $111$. Although the action is not given explicitly in \cite{Endo}, it has been possible to compute this from the information given in the paper. This action is defined by 222 matrices in $\textnormal{Sp}(6, \bb{Z})$. Rather giving all 222 matrices we shall explain the how they are constructed and will give some of them explicitly, so that it will become clear that the action of $\pi_1(B)$ is not $\zz_4$-trivial, since it is not even $\zz_2$-trivial.
The method used by Endo for the construction of this example was initially formulated by Meyer in \cite{Meyerpaper}.
\begin{example} In this example we give a description of Endo's construction and give explicitly some of the matrices of the action.

Let $F^2 \to E \to B^2$ be a surface bundle. The genus of the fibre in this example is $h=3$.
The bundle $E$ is determined by its monodromy homomorphism
$$\pi_1(B) \overset{\chi}{\longrightarrow} \mathcal{M}_h$$
where $\mathcal{M}_h$ is the mapping class group of the fibre $F_h$.

Since every element of $\mathcal{M}_h$ leaves the intersection form on $H^1(F, \zz)$ invariant and since $(H^1(F_h, \mathbb{R}), \langle . \rangle)$ is isomorphic to the standard symmetric form $(\mathbb{R}^{2h}, \omega)$, there is a natural representation
$$\mathcal{M}_h \overset{\sigma}{\rightarrow}   \textnormal{Sp}(2h, \mathbb{Z}).$$
Hence the action $U: \pi_1(B) \to \textnormal{Sp}(2h, \mathbb{Z})$ factors through
$$\pi_1(B) \overset{\chi}{\rightarrow} \mathcal{M}_h \overset{\sigma}{\rightarrow}  \textnormal{Sp}(2h, \mathbb{Z}).$$

Endo uses the presentation in \cite{Wajnryb} for the mapping class group $\mathcal{M}_h$ of the fibre.

\begin{remark} The presentation given in the paper \cite{Wajnryb} has some errata in the relations. These where corrected in a later paper
\cite{Errata}.  Although Endo only refers to \cite{Wajnryb} and not to the correction \cite{Errata}, he did not use any of the incorrect relations in \cite{Wajnryb} in his example, so fortunately \cite{Endo} is not affected by these errata.
\end{remark}

The generators of the presentation of the mapping class group $\mathcal{M}_h$ of the fibre $F_h$ are $y_i$, $u_i$ and $z_i$, and these can be interpreted as Dehn twists with respect to the curves in the following figure
\vspace{5pt}

\begin{figure} [ht]
\centering
\includegraphics[scale=0.7, trim=10 80 10 800]{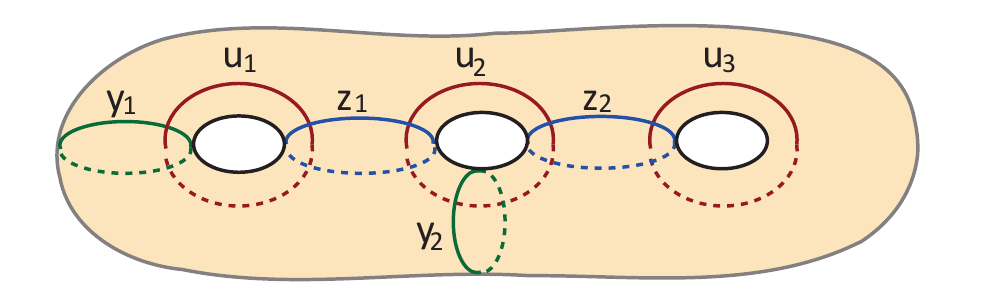}
\vspace{55pt}
\end{figure}

The relations are as given in \cite{Wajnryb} and \cite{Errata}.

The convention in \cite{Endo} is that positive twists are interpreted as right twists:

\vspace{-10pt}
\begin{figure} [ht]
\centering
\includegraphics[scale=0.6, trim=30 120 30 700]{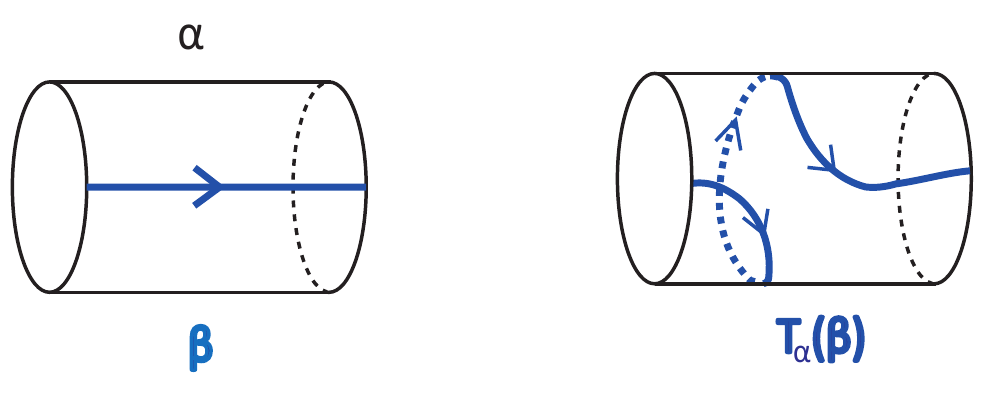}
\vspace{-10pt}
\end{figure}

\vspace{70pt}

For example the action of the generator $u_2$ on homology is given by:
\vspace{50pt}

\begin{figure} [ht]
\centering
\includegraphics[scale=0.6, trim=10 90 30 100]{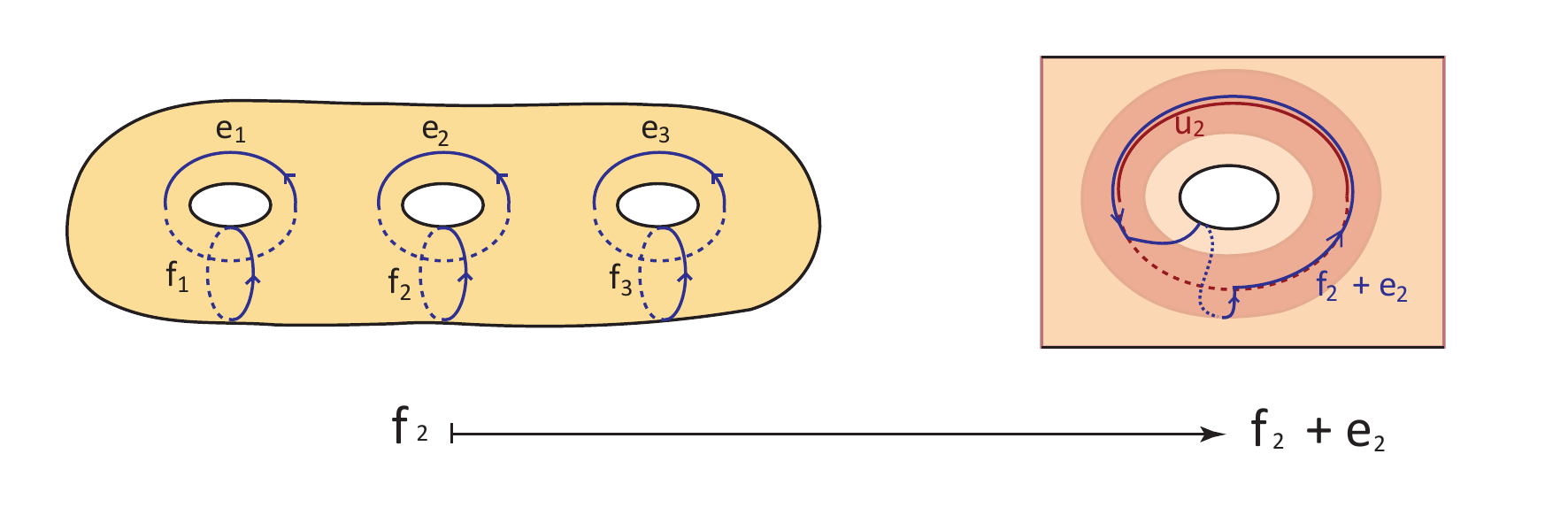}
\vspace{50pt}
\end{figure}
This map is given by a matrix in $\textnormal{Sp}(2h, \mathbb{Z})$,

$$\left( \begin{array}{cccccc}
1& 0&  0& 0& 0& 0 \\
0& 1&  0&0& 1& 0 \\
0& 0&  1&0& 0& 0 \\
0& 0&  0&1& 0& 0 \\
0& 0&  0&0& 1& 0\\
0& 0&  0&0& 0& 1 \\
 \end{array} \right) \left( \begin{array}{c}
0 \\
0 \\
0 \\
0\\
1\\
0 \\
 \end{array} \right) = \left( \begin{array}{c}
0 \\
1 \\
0 \\
0\\
1\\
0 \\
 \end{array} \right).$$
The explicit representation of all the generators is given on \cite[page 920]{Endo}.

To find a non-trivial example Endo requires a long word of generators of the mapping class group $\mathcal{M}_h$.
The construction of this word of generators is described in \cite[page 923]{Endo}. Endo then follows a commutator collection process to rearrange the word as a product of commutators, so that the commutators in $\pi_1(B)$ are mapped to commutators in $\mathcal{M}_h,$
\begin{align*}
\pi_1(B) \overset{\chi}{\longrightarrow} &\hspace{9pt} \mathcal{M}_h \hspace{5pt} \overset{\sigma}{\longrightarrow}  \textnormal{Sp}(2h, \mathbb{Z}) \\
[a_i, b_i] \hspace{1pt} \longmapsto & [A_i, B_i, ].
\end{align*}
with $A_i$ and $B_i$ products of generators of $\mathcal{M}_h.$

As the genus of the base in this example is $111$, the action is defined by 222 matrices in $\textnormal{Sp}(6, \mathbb{Z})$, such that the product of the commutator is equal to $1$. A generator of $\mathcal{M}_h$ is represented by a matrix in $\textnormal{Sp}(6, \mathbb{Z})$. After the collection process on the word of generators we have commutators $[A_i, B_i],$ where each $A_i$ or $B_i$ is the product of generators of $\mathcal{M}_h$. Using the explicit representation given in \cite[page 920]{Endo}, $\sigma: \mathcal{M}_h \to \textnormal{Sp}(6, \mathbb{Z})$, we can find the matrices corresponding to each $A_i$ or $B_i$ by taking the product of the matrices that represent each of the generators in $\textnormal{Sp}(6, \mathbb{Z}).$

Two of them are for example,

\begin{minipage}[c]{0.4\linewidth}
\vspace{-15pt}
$$\left( \begin{array}{cccccc}
1& 0&  0&-1& 0& 0 \\
0& 1&  0&0& 0& 0 \\
0& 0&  1&0& 0& 0 \\
0& 0&  0&1& 0& 0 \\
0& 0&  0&0& 1& 0\\
0& 0&  0&0& 0& 1 \\
 \end{array} \right),$$

\end{minipage}
\begin{minipage}[c]{0.40\linewidth}

\vspace{5pt}
 $$\left( \begin{array}{cccccc}
1& 0&  0&0& 0& 0 \\
0& 1&  0&0& 0& 0 \\
0& 0&  1&0& 0& 0 \\
1& -1&  0&1& 0& 0 \\
-1& 1&  0&0& 1& 0\\
0& 0&  0&0& 0& 1 \\
\end{array} \right).$$
\vspace{10pt}
\end{minipage}

These are clearly not trivial when reduced modulo $4$, as expected from the proof of Theorem \ref{mod8-theorem}. Interestingly these matrices are also trivial when reducing mod $2$, so Endo's example does not provide a counterexample for disproving the conjecture in \cite{KT}.

\end{example}

\section{Constructing further examples}\label{more-examples}

By Novikov additivity we know that if two compact oriented $4k$-dimensional manifolds are glued by an orientation reversing diffeomorphism of their boundaries, then the signature of the union is the sum of the signature of the components. Following this idea Meyer computed the signature of the fibre bundle by dividing the base $B$ into pairs of pants $X$ and then computing the signature of the lift in the total space of each of these pairs of pants.
\begin{figure}[ht!]
\centering
\includegraphics[scale=0.5]{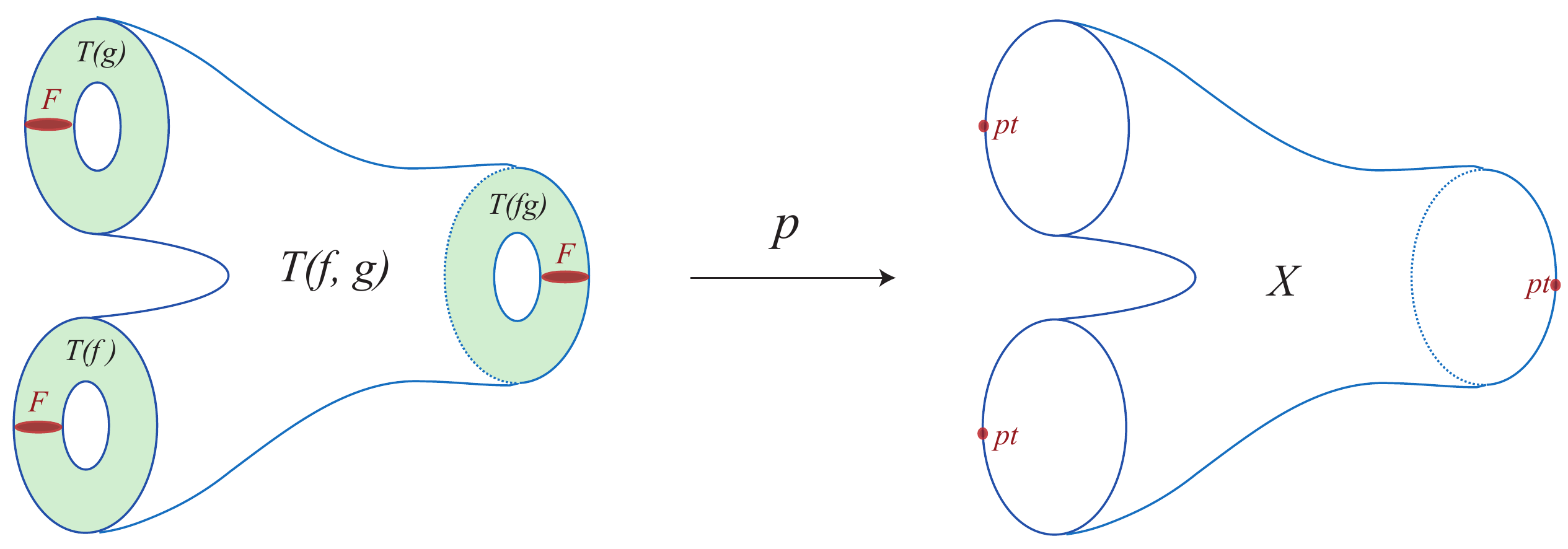}
\label{T(f, g)}
\end{figure}

The manifold on the right in figure \ref{T(f, g)}   is a $4$-dimensional manifold with three boundary components, where each boundary component is a mapping torus given by an automorphism of the fibre. We shall refer to this as $T(f, g)$ as it is the part of the total space depending on the two  symplectic automorphisms $$f, g: (H, \phi) \to (H, \phi)$$ of the skew-symmetric form $(H, \phi) = (H^1(F, \zz), \phi)$. 

Define $(V = H \oplus H, \Phi= \phi \oplus - \phi)$ and three Lagrangians
\begin{align*}
A&= \{(x, x ) \in H \oplus H \vert x \in H  \},   \\
B&= \{(y, f(y) ) \in H \oplus H \vert y \in H  \},  \\
C&= \{(z, g(z)) \in H \oplus H \vert x \in H  \}.
\end{align*}
The form $\Phi : V \times V \to \bb{Z}$ is a skew-symmetric bilinear map with $$\Phi(A\times A)=\Phi(B\times B)=\Phi(C\times C)=0.$$

As in \cite{Wall-non-add}, we consider the additive relation between $A$ and $B$ defined as $a \sim b$ if $\exists c \in C$ such that $a = b+c$. The domain of this relation is the set of $(x,x) = a \in A$ that can be expressed as $(b + c)= (y + z, f(y)+g(z))$, with $b= (y, f(y))$ and $c= (z, g(z))$. That is, the set
$$A \cap (B+C)= \{(x,x) = (y, f(y)) + (z, g(z)) |  x = y+z=f(y)+g(z)\}.$$
We define (as in \cite{Wall-non-add}) a bilinear map $\Psi: A \cap (B+C) \times A \cap (B+C) \to \bb{R}$ by
\begin{align*}
\Psi((x,x), (x', x')) & =  \Phi((x,x), (y', f(y'))) \\
                             & = \phi(x, y') - \phi(x, f(y')) \\
                            & = \phi(x, (1-f)  y').
\end{align*}
The form $\Psi$ is symmetric as we now check,
\begin{align*}
\Psi((x,x), (x', x'))-\Psi((x',x'), (x, x)) &= \phi(x, (1-f)  y') -\phi(x', (1-f)  y) \\
                                                            & = \phi(x, (1-f)  y') +\phi((1-f)  y, x')  \\
                                                         &= \phi(x, y') - \phi(x, f(y')) + \phi(y, x') - \phi(f(y), x') \\
                                                          & = \phi(x+y, x' + y') - \phi(x, x')- \phi(y, y') \\
& \hspace{30pt} - \phi(x+f(y), x'+f(y')) + \phi(x,x')+\phi(f(y), f(y')) \\
                                                      & =   \phi(z, z') - \phi(g(z), g(z')) \\
                                                 & = \Phi((z, g(z)), (z', g(z'))) = 0
\end{align*}


The signature of $(A \cap (B+ C), \Psi)$ is the nonadditivity invariant of \cite{Wall-non-add} such that $\sigma(T(f, g))= \sigma(A \cap (B+ C), \Psi) \in \zz.$

\subsection*{Constructing an isomorphism $H \xrightarrow{\cong} A \cap (B+C)$}
If $1-f$ is an automorphism, then from the equality $y+z = f(y)+g(z)$ we obtain the relation
$$y = -(1-f)^{-1}(1-g)z \hspace{20pt} \textnormal{for any } z \in H $$
and we define,
 \begin{align*}
H  \xrightarrow{\cong} &A \cap (B+C) \\
z \mapsto & (y+z , f(y)+g(z)) = (x, x) \\
 & = (-(1-f)^{-1}(1-g)z + z, f (-(1-f)^{-1}(1-g)z )+ g(z)) \\
& = \left((1-f)^{-1}(-(1-g)+(1-f))z, (1-f)^{-1}(-f(1-g)z+ (1-f)g(z)) \right) \\
& = \left((1-f)^{-1}(g-f)z, (1-f)^{-1}((-f+fg)z+ (g-fg)z) \right) \\
& = \left((1-f)^{-1}(g-f)z, (1-f)^{-1}((g-f)z) \right)
           \end{align*}

That is, the isomorphism is given by
 \begin{align*}
h: H  \xrightarrow{\cong} &A \cap (B+C) \\
z \mapsto &  \left((1-f)^{-1}(g-f)z, (1-f)^{-1}((g-f)z) \right)
\end{align*}

As $\Psi$ above is symmetric, then composition with the isomorphism $h$ gives a symmetric matrix.
$$h^* \Psi h(z, z') = \phi \left(  (1-f)^{-1}(g-f) z, -(1-g) z' \right)$$

We can now choose two arbitrary automorphisms $f$ and $g$ and find that the form $S(f, g)= \phi(1-g^{-1})(1-f)^{-1} (g-f)$ is indeed symmetric.
\vspace{5pt}
The signature of this symmetric form $S(f, g)$ is the signature of $T(f, g)$, which is the part of the total space corresponding to the pair of pants that is determined by the two automorphisms $f$ and $g$. $T(f, g)$ has boundary components three mapping tori  $T(f)$, $T(g)$ and $T(fg)$.

Now consider the space $T(f, gf^{-1}g^{-1})$. This has boundary components, $T(f)$, $T(gf^{-1}g^{-1})$ and $T([f,g])$, where $[f, g]= fgf^{-1}g^{-1}$.

It can be proved that the boundary components $T(f)$, $T(gf^{-1}g^{-1})$  are related by a homeomorphism, and identifying these two boundary components gives a bundle over a punctured torus with one boundary component, $T(fgf^{-1}g^{-1})=T([f,g])$, and this space has the same signature as $T(f, gf^{-1}g^{-1})$

\begin{example}
Our first example is a local coefficient system. The base has genus $2$ and the fibre genus $1$.

As the base has genus $2$ we need to define have four automorphisms $f_1$, $g_1$, $f_2$, $g_2$, all of them symplectic matrices.

\begin{minipage}[c]{0.4\linewidth}
\vspace{-10pt}

$$f_1 = \left( \begin{array}{cccccc}
0 & 1 \\
-1 & 0
 \end{array} \right),$$

\end{minipage}
\begin{minipage}[c]{0.30\linewidth}

$$g_1 = \left( \begin{array}{cccccc}
0 & 1 \\
-1 & 1
 \end{array} \right),$$

\vspace{10pt}
\end{minipage}

\vspace{-25pt}

\begin{minipage}[c]{0.4\linewidth}
\vspace{-10pt}

$$f_2 = \left( \begin{array}{cccccc}
0 & -1 \\
1 & -1
 \end{array} \right),$$
\end{minipage}
\begin{minipage}[c]{0.30\linewidth}

$$g_2 = \left( \begin{array}{cccccc}
0 & 1 \\
-1 & 0
 \end{array} \right),$$

\vspace{10pt}
\end{minipage}

\vspace{-10pt}

Here we have made a choice of the first two automorphisms $f_1$ and $g_1$ and then constructed the other bearing in mind that the product of the commutators has to be equal to $1$, $[f_1, g_1].[f_2, g_2] = 1$. So in this example, $f_1 = g_2$ and $g_1= -f_2$.

All four matrices are invertible. And none of $\textnormal{det}(I-f_1)$, $\textnormal{det}(I-g_1)$, $\textnormal{det}(I-f_2)$, $\textnormal{det}(I-g_2)$ is $0$.

The symmetric form for $f_1$, $g_1$ is:
$$S(f_1,g_1f_1^{-1}g_1^{-1}) =\left( \begin{array}{cccccc}
4 & -3 \\
-3 & 3.5
 \end{array} \right),$$
this has $2$ positive eigenvalues, i.e, signature $2$.

The symmetric form for $f_2$, $g_2$ is:
$$S(f_2,g_2f_2^{-1}g_2^{-1}) =\left( \begin{array}{cccccc}
-4/3 & -2/3 \\
-2/3 & -10/3
 \end{array} \right),$$
this has $2$ negative eigenvalues, i.e, signature $-2$.

Hence the local coefficient system has signature $4$.

\end{example}

\begin{example}
Another example is a local coefficient system. The base has genus $2$ and the fibre genus $1$.

As the base has genus $2$ we need to define have four automorphisms $f_1$, $g_1$, $f_2$, $g_2$, all of them symplectic matrices.

\begin{minipage}[c]{0.4\linewidth}
\vspace{-10pt}

$$f_1 = \left( \begin{array}{cccccc}
0 & 1 \\
-1 & 0
 \end{array} \right),$$

\end{minipage}
\begin{minipage}[c]{0.30\linewidth}

$$g_1 = \left( \begin{array}{cccccc}
0 & 1 \\
-1 & 1
 \end{array} \right),$$
\vspace{10pt}
\end{minipage}

\vspace{-25pt}

\begin{minipage}[c]{0.4\linewidth}
\vspace{-10pt}

$$f_2 =  \left( \begin{array}{cccccc}
4 & -3 \\
7 & -5
 \end{array} \right),$$
\end{minipage}
\begin{minipage}[c]{0.30\linewidth}


$$g_2 = \left( \begin{array}{cccccc}
-3 & 2 \\
-5 & 3
 \end{array} \right),$$
\vspace{10pt}
\end{minipage}

\vspace{-10pt}

The product of the commutators $[f_1, g_1].[f_2, g_2] = 1$
In this example we have made the same choice as in the previous example for the first two commutators $f_1$ and $g_1$ and we have constructed the other two by conjugating both by an suitable matrix. So with
$$D =  \left( \begin{array}{cccccc}
1 & 1 \\
1 & 2
 \end{array} \right),$$
we have the following relations
$$ f_2 = -D . g_1. D^{-1} \textnormal{ \hspace{2pt} and \hspace{5pt}} g_2 = D . f_1. D^{-1}$$
All four matrices are invertible. And none of $\textnormal{det}(I-f_1)$, $\textnormal{det}(I-g_1)$, $\textnormal{det}(I-f_2)$, $\textnormal{det}(I-g_2)$ is $0$.

The symmetric form for $f_1$, $g_1$ is:
$$S(f_1,g_1f_1^{-1}g_1^{-1}) =\left( \begin{array}{cccccc}
4 & -3 \\
-3 & 3.5
 \end{array} \right),$$
this has $2$ positive eigenvalues, i.e, signature $2$.

The symmetric form for $f_2$, $g_2$ is:
$$S(f_2,g_2f_2^{-1}g_2^{-1}) =\left( \begin{array}{cccccc}
-6 & 4 \\
4 & -3.33333333
 \end{array} \right),$$
this has $2$ negative eigenvalues, i.e, signature $-2$.
Hence the local coefficient system has signature $4$.

\end{example}

\appendix
\chapter{Python module for computations in section \ref{more-examples}}
 In this appendix we give the Python module used to compute the examples in section \ref{more-examples}. We shall not give any detailed explanations as this material has already been discussed in section \ref{more-examples}.
What we give here is the Python module, which we have written using the Python 3.4.1. version.

\vspace{10pt}

\begin{lstlisting}
from numpy import linalg as LA
import numpy as NP

#THIS IS THE FORMULA FOR THE SYMMETRIC FORM:
def symm(I1, P, f, g):
    return P*(I1-g.I)*(I1-f).I*(g-f)


#THIS IS THE FORMULA TO FIND THE COMMUTATOR OF TWO MATRICES
def com(a, b):
    return a*b*a.I*b.I


I1 = NP.matrix("1 0; 0 1")
P = NP.matrix("0 2; -1 0")

f = NP.matrix("0 1; -1 0")
g = NP.matrix("0 1; -1 1")

f2 = NP.matrix("0 -1; 1 -1")
g2 = NP.matrix("0 1; -1 0")

print("THIS IS A CHECK THAT THE PRODUCT \
OF THE COMMUTATORS GIVES THE IDENTITY")
print(com(f, g)*com(f2, g2))
d = com(f, g)*com(f2, g2)
e = [[1, 0],[0, 1]]
if (d == e).all() :
    print ('OK')
else :
    print("The product of the commutators is not 1")
    exit()


#---------------------------------------
#         signature for f, g
#---------------------------------------


print("NOW WE COMPUTE THE SIGNATURE \
OF THE FORM CORRESPONDING TO f AND g")


print("Check that the matrices \
given by (I1-f) and (I1-g*f.I*g.I) are nonsingular.")
# So we begin by checking this by computing their determinants


print("The determinant of (I-f) is")
print(LA.det(I1-f))
x = LA.det(I1-f)
if not (x == 0) :
    print ('OK')
else :
    print("I-g is nonsingular")
    exit()

    
print("The determinant of (I-g*f.I*g.I) is")
print(LA.det(I1-g*f.I*g.I))
y = LA.det(I1-g*f.I*g.I)
if not (y == 0) :
    print ('OK')
else :
    print("I-g is nonsingular")
    exit()


print("The symmetric form S(f, g*f^{-1}*g^{-1}) is:")
print(symm(I1, P, f, g*f.I*g.I ))

#print("This has eigenvalues:")

#print(LA.eigvals(symm(I1, P, f, g*f.I*g.I )))

a = LA.eigvals(symm(I1, P, f, g*f.I*g.I ))


#print("The number of positive eigenvalues is")
#print(((0 < a)).sum())
#print("The number of negative eigenvalues is")
#print(((0 > a)).sum())
print("The signature of S(f, g*f^{-1}*g^{-1}) is:")
print(((0 < a)).sum()-((0 > a)).sum())

#print("The number of positive eigenvalues is")
#k = LA.eigvals(symm(I1, P, f, g*f.I*g.I ))
#x = symm(I1, P, f, g*f.I*g.I )
#sum(k>0 for k in x)
#print(sum(k>0 for k in x))


#---------------------------------------
#         signature for f2, g2
#---------------------------------------


print("NOW WE COMPUTE THE SIGNATURE OF \
THE FORM CORRESPONDING TO f2 AND g2")



print("Check that the matrices given \
by (I1-f2) and (I1-g2*f2.I*g2.I) are nonsingular.")
# Begin by checking this by computing their determinants

print("The determinant of (I-f2) is")
print(LA.det(I1-f2))
x = LA.det(I1-f2)
if not (x == 0) :
    print ('OK')
else :
    print("I-f2 is nonsingular")
    exit()

    
print("The determinant of (I-g2*f2.I*g2.I) is")
print(LA.det(I1-g2*f2.I*g2.I))
y = LA.det(I1-g2*f2.I*g2.I)
if not (y == 0) :
    print ('OK')
else :
    print("I-g2 is nonsingular")
    exit()


print("The symmetric form S(f, g*f^{-1}*g^{-1}) is:")
print(symm(I1, P, f2, g2*f2.I*g2.I ))

#print("This has eigenvalues:")

#print(LA.eigvals(symm(I1, P, f2, g2*f2.I*g2.I )))

i = LA.eigvals(symm(I1, P, f2, g2*f2.I*g2.I ))


#print("The number of positive eigenvalues is")
#print(((0 < a)).sum())
#print("The number of negative eigenvalues is")
#print(((0 > a)).sum())
print("The signature of S(f2, g2*f2^{-1}*g2^{-1}) is:")
print(((0 < i)).sum()-((0 > i)).sum())

#print("The number of positive eigenvalues is")
#k = LA.eigvals(symm(I1, P, f, g*f.I*g.I ))
#x = symm(I1, P, f, g*f.I*g.I )
#sum(k>0 for k in x)
#print(sum(k>0 for k in x))

print ("TAKING ORIENTATIONS INTO ACCOUNT, \
THE SIGNATURE FOR THIS LOCAL COEFFICIENT SYSTEM IS:")

s1 =((0 < a)).sum()-((0 > a)).sum()
s2 =((0 < i)).sum()-((0 > i)).sum() 
print(s1-s2)


    
    
\end{lstlisting}

\newcommand{\etalchar}[1]{$^{#1}$}


\end{document}